\documentclass[10pt, reqno]{amsart}
\usepackage[T1]{fontenc}
\usepackage[utf8]{inputenc}
\usepackage{amsmath,amsfonts,amsthm,amssymb,amsxtra}
\usepackage{hyperref}
\usepackage{bbm} 
\usepackage{color}
\usepackage[margin=3cm]{geometry}
\usepackage[pagewise]{lineno}
\usepackage[parfill]{parskip} 
\usepackage{enumitem}
\setlist{listparindent=\parindent,parsep=0pt,}

\newtheorem{theorem}{Theorem}[section]
\newtheorem{proposition}[theorem]{Proposition}
\newtheorem{lemma}[theorem]{Lemma}
\newtheorem{corollary}[theorem]{Corollary}

\theoremstyle{definition}
\newtheorem{definition}[theorem]{Definition}

\theoremstyle{remark}
\newtheorem{remark}[theorem]{Remark}

\numberwithin{equation}{section}

\newcommand{\C}{\mathbb{C}}

\renewcommand{\epsilon}{\varepsilon}

\newcommand{\N}{\mathbb{N}}

\newcommand{\R}{\mathbb{R}}

\newcommand*\diff{\mathop{}\!\mathrm{d}} 
\newcommand{\eps}{\varepsilon}
\newcommand{\xl}{{x, \lambda}}
\newcommand{\zl}{{z, \lambda}}
\newcommand{\pxi}{\partial_{x_i}}
\newcommand{\pxj}{\partial_{x_j}}
\newcommand{\pl}{\partial_{\lambda}}
\newcommand{\ths}{{\widetilde{H}^s(\Omega)}}

\newcommand{\Sph}{{\mathbb{S}^N}}
\newcommand{\Ste}{{\mathcal{S}}}
\newcommand{\Ds}{{(-\Delta)^{s}}}
\newcommand{\ds}{{(-\Delta)^{s/2}}}
\newcommand{\Ps}{{\mathcal{P}_{2s}}}
\newcommand{\ns}{\frac{N-2s}{2}}
\newcommand{\Ns}{{N-2s}}
\newcommand{\dd}{\, \mathrm{d}}

\DeclareMathOperator{\dist}{dist}

\begin{document}

\title[Fractional Brezis--Nirenberg in low dimensions]{Critical functions and blow-up asymptotics for the fractional Brezis--Nirenberg problem in low dimension}

\author[N. De Nitti]{Nicola De Nitti}
\address[N. De Nitti]{Friedrich-Alexander-Universität Erlangen-Nürnberg, Department of Data Science, Chair for Dynamics, Control and Numerics (Alexander von Humboldt Professorship), Cauerstr. 11, 91058 Erlangen, Germany.}
\email{nicola.de.nitti@fau.de}

\author[T. K\"onig]{Tobias K\"onig}
\address[T. König]{Institut de Mathématiques de Jussieu -- Paris Rive Gauche	Université de Paris -- Campus des Grands Moulins Bâtiment Sophie Germain, Boite Courrier 7012	8 Place Aurélie Nemours, 75205 Paris Cedex 13.}
\email{koenig@imj-prg.fr}

\keywords{Fractional Laplacian; Brezis--Nirenberg problem; semilinear elliptic equations with critical nonlinearities; critical Sobolev embeddings; critical dimension; critical functions; optimal Sobolev constant; energy asymptotics. \vspace{1.5mm}}

\thanks{}

\begin{abstract}
For $s \in (0,1)$ and a bounded open set $\Omega \subset \R^N$ with $N > 2s$, we study the fractional Brezis--Nirenberg type minimization problem of finding
\[ S(a) := \inf \frac{\int_{\R^N} |(-\Delta)^{s/2} u|^2 + \int_\Omega a u^2}{\left( \int_\Omega u^\frac{2N}{N-2s} \right)^\frac{N-2s}{N}}, \]
where the infimum is taken over all functions $u \in H^s(\R^N)$ that vanish outside $\Omega$. The function $a$ is assumed to be critical in the sense of Hebey and Vaugon. For low dimensions $N \in (2s, 4s)$, we prove that the Robin function $\phi_a$ satisfies $\inf_{x \in \Omega} \phi_a(x) = 0$, which extends a result obtained by Druet for $s = 1$. In dimensions $N \in (8s/3, 4s)$, we then study the asymptotics of the fractional Brezis--Nirenberg energy $S(a + \eps V)$ for some $V \in L^\infty(\Omega)$ as $\eps \to 0+$. We give a precise description of the blow-up profile of (almost) minimizing sequences and characterize the concentration speed and the location of concentration points.
\end{abstract}

\maketitle


\section{Introduction and main results}

Let $N \in \N$ and $0 < 2s < N$ for some $s \in (0,1)$, and let $\Omega \subset \R^N$ be a bounded open set. The goal of the present paper is to analyze the variational problem of minimizing, for a given $a \in C(\overline{\Omega})$, the quotient functional 
\begin{equation}
\label{quot_func_a}
\mathcal S_{a}[u] := \frac{ \int_{\R^N} |\ds u|^2 \diff y +  \int_\Omega a(y) u(y)^2 \diff y}{\|u\|_{L^{\frac{2N}{N-2s}}(\Omega)}^2}
\end{equation}
over functions in the space
\begin{equation}
    \label{ths definition}
    \ths := \left\{ u \in H^s(\R^N) \, : \, u \equiv 0 \quad \text{ on } \R^N \setminus \Omega \right\},
\end{equation}
where $u \in H^s(\R^N)$ iff 
\begin{align}\label{eq:fr_norm}
\Vert u \Vert_{L^2(\R^N)} + \left(\int_{\R^N} |\ds u|^2 \diff y \right)^{1/2} < \infty,
\end{align} 
and the fractional Laplacian operator $\Ds u$ is defined for any $u \in H^s(\R^N)$ through the Fourier representation 
\begin{equation}
    \label{frac lap def Fourier}
    (-\Delta)^s u = \mathcal F^{-1} (|\xi|^{2s} \mathcal F u). 
\end{equation} 
We also recall the singular integral representation of the fractional Laplacian (see  \cite{BuVa2016,MR3916700}): 
\begin{equation}
\label{frac-lap-sg-int}
(-\Delta)^s u(x) := C_{N,s} P.V. \int_{\R^N} \frac{u(x) - u(y)}{|x-y|^{N+2s}} \diff y,
\end{equation}
where 
\begin{align}\label{eq:constant}
C_{N,s} :=   \frac{s 2^{2s} \Gamma(\frac{N+2s}{2})}{\pi^{N/2} \Gamma (1-s)}. 
\end{align}

The associated infimum,
\begin{equation}
    \label{S(a) def}
    S(a) :=  \inf \left \{ \mathcal S_{a}[u] \, : \, u \in \ths \right \},
\end{equation}
is to be compared with the number $S:= S_{N,s} := S(0)$, which is equal to the best constant in the fractional Sobolev embedding 
\begin{equation}
    \label{sob_ineq}
            \int_{\R^N} |\ds u|^2 \diff y \geq S \| u\|_{L^\frac{2N}{N-2s}(\R^N)}^2 , 
\end{equation}
given by     \begin{equation}
    \label{eq:sobolevconstant}
     S_{N,s} := 2^{2s} \pi^s \frac{\Gamma(\frac{N+2s}{2})}{\Gamma(\frac{N-2s}{2})} \left( \frac{\Gamma(N/2)}{\Gamma(N)}\right)^{2s/N}.
\end{equation}
We note that the embedding $\ths \hookrightarrow L^{p+1}(\Omega)$ and the associated best constant 
are in fact independent of $\Omega$ and equal to the best full-space Sobolev constant $S_{N,s}$ (see \cite{MR2944369}). 
In the classical case $s = 1$, problem \eqref{S(a) def} has been first  studied in the famous paper \cite{BrNi} by Brezis and Nirenberg, who were interested in obtaining positive solutions to the associated elliptic equation. One of the main findings in that paper is that the behavior of \eqref{S(a) def} depends on the space dimension $N$ in a rather striking way. Indeed, when $N \geq 4$, then $S(a) < S$ if and only if $a(x) < 0$ for some $x \in \Omega$. On the other hand, when $N = 3$, then $S(a) = S$ whenever $\|a\|_\infty$ is small enough, leaving open the question of characterizing the cases $S(a) < S$ in terms of $a$. In \cite{Druet2002}, Druet proved that, for $N=3$, the following equivalence holds:
\begin{equation}
    \label{druet equivalence}
    S(a) < S \qquad  \iff \qquad  \phi_a(x) < 0 \quad \text{ for some }  x \in \Omega,
\end{equation} 
where $\phi_a(x)$ denotes the Robin function associated to $a$ (see \eqref{phi a definition} below). This answered positively a conjecture previously formulated by Brezis in  \cite{Brezis1986}. 

For a fractional power $s \in (0,1)$ and assuming $a = - \lambda$ for some constant $\lambda >0$, Brezis--Nirenberg type results  have been obtained by Servadei and Valdinoci:
\begin{enumerate}
	\item[(i)] In \cite{SeVa2015}, they proved that when $N \ge 4s$, then $S(-\lambda) < S$ whenever $\lambda > 0$;
	\item[(ii)] In \cite{SeVa2013}, they proved that for $2s < N < 4s$ and prove that there is $\lambda_{s} \in (0, \lambda_{1,s})$ (where $\lambda_{1,s}$ is the first Dirichlet eigenvalue of $(-\Delta)^s$) such that for every $\lambda \in (\lambda_{s}, \lambda_{1,s})$, one has $S(-\lambda) < S$. 
\end{enumerate}

In this paper, we shall exclusively be concerned with the low-dimensional range $2s < N < 4s$. This is the natural replacement of the classical case $N=3$. $s=1$, as indicated by the results above. One may also notice that when $2s <N$, the Green's function for $(-\Delta)^s$ on $\R^N$ behaves like $G(x,y) \sim |x-y|^{-N+2s}$ near the diagonal and thus fails to be in $L^2_\text{loc}(\R^N)$ precisely if $N \leq 4s$, compare \cite{MR1705383}. 

A central notion to what follows is that of a critical function $a$, which was introduced  by Hebey and Vaugon in \cite{HeVa2001} for for $s = 1$ and readily generalizes to the fractional situation. Indeed, the following definition is naturally suggested by the behavior of $S(a)$ just described. 

\begin{definition}[Critical function]
\label{definition critical function}
 Let $a \in C(\overline{\Omega})$. We say that $a$ is \emph{critical} if $S(a) = S$ and $S(\tilde{a}) < S(a)$ for every $\tilde{a} \in C(\overline{\Omega})$ with $\tilde{a} \leq a$ and $\tilde{a} \not \equiv a$. 
\end{definition}

When $N \geq 4s$, the result of \cite{SeVa2015} implies that the only critical potential is $a \equiv 0$. For this case, or more generally for $N > 2s$ with $a \equiv 0$, the recent literature is rather rich in refined results going beyond \cite{SeVa2015}. Notably, in \cite{ChKiLe2014} and \cite{ChKi2017}, the authors prove the fractional counterpart of some conjectures by Brezis and Peletier \cite{BrPe1989} concerning the blow-up asymptotics of minimizers to the problem $S(-\eps)$ and a related problem with subcritical exponent $p-\eps$ as $\eps \to 0$. In the classical case $s = 1$, these results are due to Han \cite{Han1991} and Rey \cite{Rey1989, Rey1990}. Corresponding existence results, also for non-minimizing multi-bubble solutions, are also given in \cite{ChKiLe2014, ChKi2017}, as well as in \cite{DaLoSi2017, MR4232665}. 

In contrast to this, in the more challenging setting of dimension $2s < N < 4s$, critical functions can have all possible shapes and are necessarily non-zero, compare \cite{Druet2002} and Corollary \ref{corollary critical} below. In this setting, and notably in the presence of a critical function, results of Han--Rey type as just discussed are much more scarce in the literature. Even in the local case $s=1$ and $N=3$, the conjecture of Brezis and Peletier (see \cite[Conjecture 3.(ii)]{BrPe1989})  which involves a (constant) critical function has only been proved recently in \cite{FrKoKo2021blow-up}. For the fractional case $2s < N < 4s$, we are not aware of any results going beyond \cite{SeVa2013}.  The purpose of the present paper is to start filling this gap.

\subsection{Main results}
\label{subsec:mainresults}

For all of our results, a crucial role is played by the Green's function of $\Ds + a$, which we introduce now. For a function $a \in C(\overline{\Omega})$ such that $(-\Delta)^s + a$ is coercive, i.e.  
\[ \int_{\R^{N}} |\ds v|^2 \diff y  + \int_\Omega a v^2 \diff y \geq c \int_{\R^N} |\ds v|^2 \diff y \]
for some $c > 0$, define $G_a: \Omega \times \R^N \to \R$ as the unique function such that for every fixed $x \in \Omega$
\begin{align}\nonumber
\begin{cases}
((- \Delta)^s + a) G_a(x, \cdot) = \gamma_{N,s} \delta_x & \text{ in } \Omega,  \\
G_a(x, \cdot) = 0  & \text{ on } \R^N \setminus  \Omega. 
\end{cases}
\end{align}
Here, we set $\gamma_{N,s} = \frac{2^{2s} \pi^{N/2} \Gamma(s)}{\Gamma(\frac{N-2s}{2})}$, so that $\Ds |y|^{-N+2s} = \gamma_{N,s} \delta_0$ on $\R^N$. Thus this choice of $\gamma_{N,s}$ ensures that we can write $G_a$ as a sum of its singular part and its regular part $H_a(x,y)$ as follows:
\[ G_a(x,y) = \frac{1}{|x-y|^{N-2s}} - H_a(x,y). \]
The function $H_a$ is continuous up to the diagonal, see e.g. Lemma \ref{lemma Ha expansion}. Therefore, we may define the \emph{Robin function} 
\begin{equation}
    \label{phi a definition}
    \phi_a(x) := H_a(x,x), \qquad  x \in \Omega.
\end{equation}
We prove several properties of the Green's functions $G_a$ in Appendix \ref{sec:app:greens}.

Our first main result is the following extension of Druet's theorem from \cite{Druet2002} to the fractional case. 

\begin{theorem}[Characterization of criticality]
	\label{theorem critical}
Let $2s < N < 4s$ and let $a \in C(\overline{\Omega})$ be such that $(-\Delta)^s + a$ is coercive. The following properties are equivalent. 
	\begin{enumerate}
		\item[(i)] There is $x \in \Omega$ such that $\phi_a(x) < 0$.
		\item[(ii)] $S(a) < S$. 
		\item[(iii)] $S(a)$ is achieved by some function $u \in \ths$.
	\end{enumerate}
	\end{theorem}

As an immediate corollary, we can characterize critical functions in terms of their Robin function.  
\begin{corollary}
    \label{corollary critical}
    Let $a$ be critical. Then $\inf_{x \in \Omega} \phi_a(x) = 0$. 
\end{corollary}

The implications $(i) \Rightarrow (ii)$ and $(ii) \Rightarrow (iii)$ in Theorem \ref{theorem critical} are well-known: indeed, $(i) \Rightarrow (ii)$ easily follows by the proper choice of test functions thanks to Theorem \ref{theorem upper bound} below; the implication $(ii) \Rightarrow (iii)$ is the fractional version of the seminal observation in \cite{BrNi} (see \cite[Theorem 2]{SeVa2015}). 

Our proof of $(iii) \Rightarrow (ii)$ is the content of Proposition \ref{proposition nonex-min} below and follows \cite[Step 1]{Druet2002}. The most involved proof is that implication is $(ii) \Rightarrow (i)$, which we give in Section \ref{ssec:criticality}. We adapt the strategy developed by Esposito in \cite{Esposito2004}, who gave an alternative proof of that implication. His approach is based on an expansion of the energy functional $\mathcal S_{a - \eps}[u_\eps]$ as $\eps \to 0$, where $a$ is critical as in Definition \ref{definition critical function} and $u_\eps$ is a minimizer of $S(a - \eps)$. 

In fact, by using the techniques applied in the recent work \cite{FrKoKo2021} for $s = 1$, we are even able to push this expansion of $\mathcal S_{a - \eps}[u_\eps]$ further by one order of $\eps$ and derive precise asymptotics of the energy $S(a -\eps)$ and of the sequence $(u_\eps)$.

To give a precise statement of our results, let us fix some more assumptions and notations. We denote the zero set of the Robin function $\phi_a$ by
$$\mathcal N_a := \{ x \in \Omega: \phi_a(x) = 0\}.$$
It follows from Theorem \ref{theorem critical} that $\inf_\Omega \phi_a(x) = 0$ if and only if $a$ is critical. In particular, $\mathcal N_a$ is not empty if $a$ is critical. 

We will consider perturbations of $a$ of the form $a + \eps V$, with non-constant $V \in L^\infty(\Omega)$. For such $V$, following \cite{FrKoKo2021}, we let  
\[ Q_V(x) := \int_\Omega V(y) G_a(x,y)^2 \diff y \]
and 
\[  \mathcal N_a(V) := \{ x \in \mathcal N_a \, : \, Q_V(x) < 0\} .  \]
Finally, we shall assume that $\Omega$ has $C^2$ boundary and that 
\begin{align}
\label{assumption a}
    a \in C(\bar \Omega) \cap C^1(\Omega), \ a(x) < 0 \quad \text{ for all } x \in \mathcal N_a; 
\end{align}
By Corollary \ref{corollary phia geq 0, a leq 0}, we have a priori that $a(x) \leq 0$ on $\mathcal N_a$. Therefore assumption \eqref{assumption a} is not severe. 

We point out that with our methods we are able to prove the following theorems only for the restricted dimensional range $\tfrac{8}{3}s < N < 4s$, which enters in Section \ref{sec:LB2}. We discuss this assumption in some more detail after the statement of Theorem \ref{thm:17} below. 

The following theorem describes the asymptotics of the perturbed minimal energy $S(a + \eps V)$ as $\eps \to 0+$. It shows in particular the non-obvious fact that the condition $\mathcal N_a(V) \neq \emptyset$ is sufficient to have $S(a+ \eps V) < S$. 

\begin{theorem}[Energy asymptotics]\label{thm:13}
Let $ \tfrac{8}{3} s < N < 4s$.  Let us assume that $\mathcal N_a(V) \neq \emptyset$. Then $S(a+\eps V) < S$ for all $\eps >0$ 
and 
\begin{align*}
    \lim_{\eps \to 0^+} \frac{S(a+\eps V) -S}{\eps^{\frac{2s}{4s-N}}} &=  \sigma_{N,s} \sup_{x \in \mathcal N_a(V)} \frac{|Q_V(x)|^\frac{2s}{4s-N}}{|a(x)|^\frac{N-2s}{4s-N}} ,
\end{align*}
where $\sigma_{N,s} > 0$ is a dimensional constant given explicitly by
 $$   \sigma_{N,s} = A_{N,s}^{-\frac{N-2s}{N}} \left( \alpha_{N,s} + c_{N,s} d_{N,s} b_{N,s} \right)^{-\frac{N-2s}{4s-N}} \left(\frac{N-2s}{2s} \right)^\frac{2s}{4s-N} \frac{4s - N}{N-2s}.
$$
The constants $A_{N,s}$, $\alpha_{N,s}$, $c_{N,s}$, $d_{N,s}$ and $b_{N,s}$ are given explicitly in Lemma \ref{lemma constants} below.
\end{theorem}

On the other hand, when $\mathcal N_a(V) = \emptyset$, the next theorem shows that the asymptotics become trivial provided $Q_V > 0$ on $\mathcal N_a$. Only in the case when $\min_{\mathcal N_a} Q_V = 0$ we do not obtain the precise leading term of $S(a + \eps V) - S$. 

\begin{theorem}[Energy asymptotics, degenerate case]\label{thm:14}
Let $ \tfrac{8}{3} s < N < 4s$.  Let us assume 
that $\mathcal N_a(V) = \emptyset$. Then $S(a+\eps V) = S + o(\eps^2)$ as $\eps \to 0^+$. If, in addition, $Q_V(x) >0$ for all $x \in \mathcal N_a$ then $S(a+\eps V) = S$ for sufficiently small $\eps>0$.
\end{theorem}

For a potential $V$ such that $\mathcal N_a(V) \neq \emptyset$, and thus $S(a + \eps V) < S$ by Theorem \ref{thm:13}, a minimizer $u_\eps$ of $S(a + \eps V)$ exists by Theorem \ref{theorem critical}. We now turn to studying the asymptotic behavior of the sequence $(u_\eps)$. In fact, since our methods are purely variational, we do not need to require that the $u_\eps$ satisfy a corresponding equation and we can equally well treat a sequence of almost minimizers in the sense of \eqref{almost_min} below. 

Since the functional $\mathcal S_{a}$ is merely a perturbation of the standard Sobolev quotient functional, it is not surprising that to leading order, the sequence $u_\eps$ approaches the family of functions
\begin{equation}
\label{bubble s}
U_{x, \lambda}(y) = \left( \frac{\lambda}{1 + \lambda^2|x-y|^2} \right)^\frac{N-2s}{2}, \qquad x \in \R^N, \quad \lambda > 0. 
\end{equation}
The $U_\xl$ are precisely the optimizers of the fractional Sobolev inequality on $\R^N$
\begin{equation}
\label{sobolev_ineq}
\|(-\Delta)^{s/2} u\|_{L^2(\R^N)}^2 \geq S_{N,s} \|u\|_{L^{\frac{2N}{N-2s}}(\R^N)}^2. 
\end{equation}
and satisfy the equation 
\begin{equation}
\label{eq_U_xl} (-\Delta)^s U_\xl(y) = c_{N,s} U_\xl(y)^\frac{N+2s}{N-2s},
\end{equation}
with $c_{N,s} > 0$ given explicitly in Lemma \ref{lemma constants}. 

Since we are working on the bounded set $\Omega$, the first refinement of the approximation consists in 'projecting' the functions $U_\xl$ to $\ths$. That is, we consider the unique function $PU_\xl \in \ths$ satisfying 
\begin{align}\label{eq:projintro}
\begin{cases}
(- \Delta)^s PU_{x, \lambda} = (- \Delta)^s U_{x, \lambda} & \quad \text{ in } \Omega, \\
PU_\xl = 0  & \quad \text{ on } \R^N \setminus  \Omega. \end{cases}
\end{align}
in the weak sense, that is, 
\[ \int_{\R^N} \ds PU_\xl \ds \eta \diff y = \int_{\R^N} \Ds U_\xl \eta \diff y = c_{N,s} \int_\Omega U_\xl^\frac{N+2s}{N-2s} \eta \diff y \]
for every $\eta \in \ths$. 

Finally, we introduce the space
\[ T_\xl = \text{span} \left \{ PU_\xl, \pl PU_\xl, \{\pxi PU_\xl\}_{i=1}^N \right\} \subset \ths \]
and denote by $T_\xl^\perp \subset \ths$ its orthogonal complement in $\ths$ with respect to the scalar product $(u,v) := \int_{\R^N} \ds u \ds v \diff y$. Moreover, let us denote by $\Pi_\xl$ and $\Pi_\xl^\bot$ the projections onto $T_\xl$ and $T_\xl^\bot$ respectively. 

Then we have the following result. 

\begin{theorem}[Concentration of almost-minimizers] \label{thm:17}
Let $ \tfrac{8}{3} s < N < 4s$. Suppose that $(u_\eps) \subset \ths$ is a sequence such that
\begin{equation}
\label{almost_min}
\lim_{\eps \to 0} \frac{\mathcal S_{a+\eps V}[u_\eps] - S(a+\eps V)}{S - S(a+\eps V)} = 0 \quad \text{ and } \quad \int_\Omega u_\eps^p \diff y = \int_{\R^N} U_{0,1}^p \diff y .       
\end{equation}
Then there exist sequences $(x_\eps) \subset \Omega$, $(\lambda_\eps) \subset (0, \infty)$, $(w_\eps) \subset T_{x_\eps, \lambda_\eps}^\bot$, and $(\alpha_\eps) \subset \R$ such that, up to extraction of a subsequence,
\begin{equation}
u_\eps =  \alpha_\eps \left( PU_{x_\eps, \lambda_\eps} + \lambda^{-\ns} \Pi^\perp_{x_\eps,\lambda_\eps}\left( H_0(x_\eps, \cdot) - H_a(x_\eps, \cdot)  \right) + r_\eps\right) 
\end{equation}
Moreover, as $\eps \to 0$, we have 
\begin{align*}
& x_\eps \to x_0 \ \text{ for some $x_0$ such that }  \quad  \frac{|Q_V(x_0)|^{\frac{2s}{4s-N}}}{|a(x_0)|^{\frac{N-2s}{4s-N}}} = \sup_{y \in \mathcal N_a(V)} \frac{|Q_V(y)|^{\frac{2s}{4s-N}}}{|a(y)|^{\frac{N-2s}{4s-N}}}, \\
&\phi_a(x_\eps) =  o(\eps), \\ 
&\lim_{\eps \to 0} \eps^\frac{1}{4s-N} \lambda_\eps =  \left(\frac{2s (\alpha_{N,s} + c_{N,s} d_{N,s} b_{N,s} ) |a(x_0)| }{(N-2s) |Q_V(x_0)| } \right)^\frac{1}{4s-N} , \\
&\alpha_\eps = \xi + \mathcal O(\eps^\frac{N-2s}{4s-N}) \text{ for some } \xi \in \{\pm 1\}.
\end{align*}
Finally, $r_\eps \in T^\perp_{x_\eps,\lambda_\eps}$ and $\|(-\Delta)^{\frac{s}{2}} r \|_{L^2(\R^N)}^2 = o(\eps^{\frac{2s}{4s-N}})$.

The constants $\alpha_{N,s}$, $c_{N,s}$, $d_{N,s}$ and $b_{N,s}$ are given explicitly in Lemma \ref{lemma constants}.
\end{theorem}

Theorem \ref{thm:17} should be seen as the low-dimensional counterpart of \cite[Theorems 1.1 and 1.2]{ChKiLe2014}, which concerns $N > 4s$. The decisive additional complication to be overcome in our case is the presence of a non-zero critical function $a$. More concretely, the coefficient $\phi_a(x)$ of the subleading term of the energy expansion vanishes due to criticality of $a$ (compare Theorem \ref{theorem upper bound} and Lemma \ref{lemma expansion refined quotient}). As a consequence, it is only after further refining the expansion that we are able to conclude the desired information about the concentration behavior of the sequence $u_\eps$. 

In the same vein, the energy expansions from Theorem \ref{thm:13} are harder to obtain than their analogues in higher dimensions $N \geq 4s$. Indeed, for $N > 4s$ we have
\begin{equation}
\label{S(eps V) higher dimension}
   S(\eps V) = S_{N,s} - \tilde{c}_{N,s}  \sup_{ \{x \in \Omega \, : \, V(x) < 0\}} \phi_0(x)^{-\frac{2s}{N-4s}} |V(x)|^\frac{N-2s}{N-4s} \eps^\frac{N-2s}{N-4s} + o(  \eps^\frac{N-2s}{N-4s}),
\end{equation} 
where $\tilde{c}_{N,s} > 0$ is some dimensional constant. In this case, a sharp upper bound on $S(\eps V)$ can already be derived from testing $\mathcal S_{\eps V}$ against the family of functions $PU_\xl$. In contrast, for $2s < N < 4s$ this family needs to be modified by a lower order term in order give the sharp upper bound for Theorem \ref{thm:13} (see \eqref{definition psi} and Theorem \ref{theorem upper bound} below). For details of the computations in case $N \geq 4s$, we refer to the forthcoming work \cite{DNK21}. It is noteworthy that the auxiliary minimization problem giving the subleading coefficient in \eqref{S(eps V) higher dimension} is local in $V$ in the sense that it only involves the pointwise value $V(x)$, whereas that of Theorem \ref{thm:13} contains the non-local quantity $Q_V$.

Let us now describe in more detail the approach we use in the proofs of Theorems \ref{thm:13}, \ref{thm:14} and \ref{thm:17}, which are in fact intimately linked. Firstly, the family of functions $\psi_\xl$ defined in \eqref{definition psi} below yields an upper bound for $S(a + \eps V)$, which will turn out to be sharp. The strategy we use to prove the corresponding lower bound on $\mathcal S_{a + \eps V}[u_\eps]$, for a sequence $(u_\eps)$ as in \eqref{almost_min}, can be traced back at least to work of Rey \cite{Rey1989, Rey1990} and Bahri--Coron \cite{BaCo1988} on the classical Brezis--Nirenberg problem for $s =1$; it was adapted to treat problems with a critical potential $a$ when $s=1$, $N = 3$ in \cite{Esposito2004} and, more recently, in \cite{FrKoKo2021, FrKoKo2021blow-up}. This strategy features two characteristic steps, namely (i) supplementing the initial asymptotic expansion $u_\eps = \alpha_\eps (PU_{x_\eps, \lambda_\eps} + w_\eps)$, obtained by a concentration-compactness argument, by the orthogonality condition $w_\eps \in T_{x_\eps, \lambda_\eps}^\bot$ and (ii) using a certain coercivity inequality, valid for functions in $T_{x_\eps, \lambda_\eps}^\bot$, to improve the bound on the remainder $w_\eps$. The basic instance of this strategy is carried out in Section \ref{sec:LB1}. Indeed, after performing steps (i) and (ii), in Proposition \ref{prop w and d} below we are able to exclude concentration near $\partial \Omega$ and obtain a quantitative bound on $w_\eps = \alpha_\eps^{-1} u_\eps - PU_{x_\eps, \lambda_\eps}$. As in \cite{Rey1989} and \cite{FrKoKo2020}, this information is enough to arrive at \eqref{S(eps V) higher dimension} and similar conclusions when $N > 4s$; see the forthcoming paper \cite{DNK21} for details.  

On the other hand, when $2s < N < 4s$, the bound that Proposition \ref{prop w and d} provides for the modified difference $u_\eps - \psi_{x_\eps, \lambda_\eps}$ is still insufficient. For $s = 1$, it was however observed in \cite{FrKoKo2021} that one can refine the expansion of $u_\eps$ by reiterating steps (i) and (ii). Here, we carry out their strategy in a streamlined version (compare Remark \ref{remark section 5 s=1}) and for fractional $s \in (0,1)$.  That is, one writes $w_\eps = \psi_{x_\eps, \lambda_\eps} - PU_{x_\eps, \lambda_\eps} + q_\eps$, decomposes $q_\eps = t_\eps + r_\eps$ with $r_\eps \in T_{x_\eps, \lambda_\eps}^\bot$ and applies the coercivity inequality a second time. We are able to conclude as long as the technical condition $8s/3 < N$ is met (which is equivalent to $\lambda^{-3N+6s} = o(\lambda^{-2s})$). Indeed, in that case the leading contributions of $t_\eps$ to the energy, which enter to orders $\lambda^{-N+2s}$ and $\lambda^{-2N+4s}$, cancel precisely; see Lemma \ref{lemma N_1 D_1}. If $8s/3 \leq N$, a plethora of additional terms in $t_\eps$, which contribute to orders $\lambda^{-k(N -2s)}$ with $3 \leq k \leq \tfrac{2s}{N-2s}$, will become relevant, and we were not able to treat those in a systemized way. It is natural to expect that the cancellation phenomenon that occurs for $k =1,2$ still persists for $k \geq 3$. This would allow to prove Theorems \ref{thm:13}, \ref{thm:14}, and \ref{thm:17} for general $N > 2s$. For further details of the argument and the difficulties just discussed, we refer to Section \ref{sec:LB2}. 

As far as we know, the role of the threshold configurations given by $k(N -2s) = 2s$ for $k \geq 1$ in the fractional Brezis--Nirenberg problem \eqref{S(a) def} has only been investigated in the literature for $k = 1$ corresponding to $N = 4s$, below which the problem is known to behave differently by the results quoted above. It would be exciting to exhibit some similar, possibly refined, qualitative change in the behavior of \eqref{S(a) def} at one or each of the following thresholds $N = 3s$, $N = 8s/3$, $N= 10s/4$, etc.

To conclude this introduction, let us mention that several works in the literature (see \cite{Tan2011, BaCoPaSa2011, BrCoPaSa2013}) treat the problem corresponding to \eqref{S(a) def} for a different notion of Dirichlet boundary conditions for $(-\Delta)^s$ on $\Omega$, namely the \emph{spectral} fractional Laplacian, defined by classical spectral theory using the $L^2(\Omega)$-ONB of Dirichlet eigenfunctions for $-\Delta$. In contrast to this, the notion of $(-\Delta)^s$ we use in this paper, as defined in \eqref{frac lap def Fourier} or \eqref{frac-lap-sg-int} on $\ths$ given by \eqref{ths definition}, usually goes in the literature by the name of \emph{restricted} fractional Laplacian. A nice discussion of these two operators, as well as a method of unified treatment for both, can be found in \cite{DaLoSi2017} (see also \cite{MR3233760}). 

Our method of proof just described is rather different from most of the previous contributions to the fractional Brezis--Nirenberg problem. Namely, we do not need to pass through the extension formulation for $(-\Delta)^s$ due to either \cite{CaSi2007} for the restricted or to \cite{CaTa2009, StTo2009} for the spectral fractional Laplacian. On the other hand, using the properties of $PU_\xl$ (as given in Lemma \ref{lemma decomp PU}) allows us to avoid lengthy calculations with singular integrals, appearing e.g. in \cite{SeVa2015}, while at the same time optimizing the cutoff procedure with respect to \cite{SeVa2015}.

\subsection{Notation}
\label{ssec:notation}

We will often abbreviate the fractional critical Sobolev exponent by $p := \frac{2N}{N-2s}$. For any $q \geq 1$, we abbreviate $\| \cdot \|_q := \|\cdot \|_{L^q(\R^N)}$. When $q=2$, we sometimes write $\|\cdot\| := \|\cdot\|_2$. 

Unless stated otherwise, we shall always assume $s\in (0,1)$ and $N \in (2s, 4s)$. 

For $x \in \Omega$, we use the shorthand $d(x) = \text{dist}(x, \partial \Omega)$.

For a set $M$ and functions $f,g : M \to \R_+$, we shall write $f(m) \lesssim g(m)$ if there exists a constant $C > 0$, independent of $m$, such that $f(m) \leq C g(m)$ for all $m \in M$, and accordingly for $\gtrsim$. If $f \lesssim g$ and $g \lesssim f$, we write $f \sim g$. 

The various constants appearing throughout the paper and their numerical values are collected in Lemma \ref{lemma constants} in the appendix.

\section{Proof of the upper bound}
\label{sec:UB}

The following theorem gives the asymptotics of $\mathcal S_{a + \eps V}[\psi_\xl]$, for the test function 
\begin{equation}
\label{definition psi}
\psi_\xl (y) := PU_\xl(y) - \lambda^{-\frac{N-2s}{2}} (H_a(x,y) - H_0(x,y)),
\end{equation}
as $\lambda \to \infty$.

\begin{theorem}[Expansion of $\mathcal S_{a + \eps V}{[\psi_\xl]}$]
	\label{theorem upper bound}
	As $\lambda \to \infty$, uniformly for $x$ in compact subsets of $\Omega$ and for $\eps \geq 0$, 
	\begin{align}
	& \quad \|\ds \psi_\xl\|_2^2 + \int_\Omega (a + \eps V) \psi_\xl^2 \diff y  \nonumber \\
	& = c_{N,s} A_{N,s} - c_{N,s} a_{N,s} \phi_a(x) \lambda^{-N+2s}  \label{exp_num} \\
	&\quad + (c_{N,s} d_{N,s} b_{N,s} - \alpha_{N,s})  a(x) \lambda^{-2s} +  \eps \lambda^{-N + 2s} Q_V(x) + o (\lambda^{-2s}) + o(\eps \lambda^{-N+2s}) \nonumber
	\end{align}
	and 
	\begin{align}
	\label{exp_den}
	\int_\Omega \psi_\xl^p \diff y &=  A_{N,s} - p a_{N,s} \phi_a(x) \lambda^{-N+2s} + \mathcal T_1 (\phi_a(x), \lambda) + p d_{N,s} b_{N,s} a(x) \lambda^{-2s}  + o(\lambda^{-2s}).
	\end{align}
	In particular, 
	\begin{align}
	\mathcal S_{a + \eps V} [\psi_\xl] &= S + A_{N,s}^{-\frac{N-2s}{N}} \Big[ a_{N,s} c_{N,s} \phi_a(x) \lambda^{-N+2s} + \mathcal T_2 (\phi_a(x), \lambda) - a(x) \lambda^{-2s} (\alpha_{N,s} + c_{N,s} d_{N,s} b_{N,s} ) \nonumber \\
	&\quad  + \eps \lambda^{-N+2s} Q_V(x) \Big] + o( \lambda^{-2s}) + o(\eps \lambda^{-N+2s}).  \label{exp_quot}
	\end{align}
Here, $\mathcal T_i(\phi, \lambda)$ are (possibly empty) sums of the form 
\begin{equation}
\label{T i definition}
    \mathcal T_i(\phi, \lambda) := \sum_{k = 2}^{K} \gamma_i(k) \phi^k \lambda^{-k(N-2s)}
\end{equation} 
for suitable coefficients $\gamma_i(k) \in \R$, where $K = \lfloor \frac{2s}{N-2s} \rfloor$ is the largest integer less than or equal to $\frac{2s}{N-2s}$.
\end{theorem}

Theorem \ref{theorem upper bound} is valid irrespective of the criticality of $a$. The following corollary states two consequences of Theorem \ref{theorem upper bound}, which concern in particular critical potentials. 

\begin{corollary}[Properties of critical potentials]
\label{corollary phia geq 0, a leq 0}
\begin{enumerate}
    \item[(i)] If $S(a) = S$, then $\phi_a(x) \geq 0$ for all $x \in \Omega$. 
    \item[(ii)] If $S(a) = S$ and $\phi_a(x) = 0$ for some $x \in \Omega$, then $a(x) \leq 0$. 
\end{enumerate}
\end{corollary}

\begin{proof}
Both statements follow from Theorem \ref{theorem upper bound} applied with $\eps = 0$. Indeed, suppose that either $\phi_a(x) < 0$ or $\phi_a(x) = 0 < a(x)$ for some $x \in \Omega$. In both cases, \eqref{exp_quot} gives $S(a) \leq \mathcal S_{a}[\psi_\xl] < S$ for $\lambda > 0$ large enough, contradiction.
\end{proof}

Based on Theorem \ref{theorem upper bound}, we can now easily derive the following upper bound for $S(a+\eps V)$ provided that $\mathcal N_a(V)$ is not empty.

\begin{corollary}
	\label{corollary upper bound}
	Suppose that $\mathcal N_a(V) \neq \emptyset$. Then $S(a+\eps V) < S$ for all $\eps > 0$ and, as $\eps \to 0+$, 
\begin{equation}
    \label{upper bound asymp corollary}
    S(a+\eps V) \leq S - \sigma_{N,s} \sup_{x \in \mathcal N_a(V)} \frac{|Q_V(x)|^\frac{2s}{4s-N}}{|a(x)|^\frac{N-2s}{4s-N}}  \eps^\frac{2s}{4s-N} + o(\eps^\frac{2s}{4s-N}), 
\end{equation}
	where 
\begin{equation}
    \label{sigma Ns definition}
    \sigma_{N,s} = A_{N,s}^{-\frac{N-2s}{N}} \left( \alpha_{N,s} + c_{N,s} d_{N,s} b_{N,s} \right)^{-\frac{N-2s}{4s-N}} \left(\frac{N-2s}{2s} \right)^\frac{2s}{4s-N} \frac{4s - N}{N-2s}.
\end{equation}
\end{corollary}

\begin{proof}
	Let us fix $\eps > 0$ and $x \in \mathcal N_a(V)$. Then by \eqref{exp_quot}
	\begin{align}
	    	S(a+\eps V) &\leq S_{a  +\eps V}[\psi_{x,\lambda}] \\
	    	&= S + A_{N,s}^{-\frac{N-2s}{N}} \left( - (a(x) +o(1)) \lambda^{-2s} (\alpha_{N,s} + c_{N,s} d_{N,s} b_{N,s} ) + \eps \lambda^{-N+2s} (Q_V(x) +o(1)) \right). 
	\end{align} 
We first optimize the right side over $\lambda > 0$. Since $A_\eps := (-a(x) +o(1))  (\alpha_{N,s} + c_{N,s} d_{N,s} b_{N,s} )$ and $B_\eps := -Q_V(x)+o(1)$, are strictly positive by our assumptions, we are in the situation of Lemma \ref{lem-taylor}. Picking $\lambda= \lambda_0(\eps)$ given by \eqref{f eps minimum}, we have $o(\lambda^{-2s}) = o(\eps^\frac{2s}{4s-N})$. Thus, by \eqref{f eps min value}, we get, as $\eps \to 0+$, 
	\begin{align*}
S(a+\eps V) &\leq´S  - \eps^\frac{2s}{4s-N}   \frac{|Q_V(x)|^\frac{2s}{4s-N}}{|a(x)|^\frac{N-2s}{4s-N}}  A_{N,s}^{-\frac{N-2s}{N}} \left( \alpha_{N,s} + c_{N,s} d_{N,s} b_{N,s} \right)^{-\frac{N-2s}{4s-N}} \left(\frac{N-2s}{2s} \right)^\frac{2s}{4s-N} \frac{4s - N}{N-2s} \\
 & \qquad + o(\eps^\frac{2s}{4s-N}).
	\end{align*}
Now, optimizing over $x\in \mathcal N_a(V)$ completes the proof of \eqref{upper bound asymp corollary}. In particular, $S(a + \eps V) < S$ for small enough $\eps > 0$. Since $S(a + \eps V)$ is a concave function of $\eps$ (being the infimum over $u$ of functions $\mathcal S_{a + \eps V}[u]$ linear in $\eps$) and $S(a) = S$, this implies that $S(a + \eps V) < S$ for every $\eps > 0$. 
\end{proof}

\begin{proof}
[Proof of Theorem \ref{theorem upper bound}]
\textbf{Step 1:} \emph{Expansion of the numerator.} Since $(-\Delta)^s H_a(x, \cdot) = a G_a(x, \cdot)$, the function $\psi_\xl$ satisfies
\[ \Ds \psi_\xl = c_{N,s} U_\xl^{p-1} - \lambda^{-\ns} a G_a(x, \cdot). \]
Therefore, recalling Lemma \ref{lemma decomp PU},
\begin{align*}
\| \ds \psi_\xl\|_2^2 &= \int_\Omega \psi_\xl(y) \Ds \psi_\xl(y) \diff y \\
&= \int_\Omega \left( U_\xl - \lambda^{-\ns} H_a(x, \cdot) - f_\xl \right) \left(  c_{N,s} U_\xl^{p-1} - \lambda^{-\ns} a G_a(x, \cdot) \right) \diff y \\
&= c_{N,s} \int_\Omega U_\xl^p \diff y - c_{N,s} \lambda^{-\ns} \int_\Omega U_\xl^{p-1} H_a(x, \cdot) \diff y \\
& \qquad  - \lambda^{-\ns} \int_\Omega a G_a(x, \cdot) \left( U - \lambda^{-\ns} H_a(x, \cdot) \right) \diff y  \\
& \qquad - \int_\Omega f_\xl \left( c_{N,s} U_\xl^{p-1} - \lambda^{-\ns} a G_a(x,\cdot) \right) \diff y. 
\end{align*}
We now treat the four terms on the right side separately. 

A simple computation shows that $\int_{\R^N \setminus B_{d(x)}(x)} U_\xl^p \diff y = \mathcal O(\lambda^{-N})$. Thus the first term is given by 
\[  c_{N,s} \int_\Omega U_\xl^p \diff y = c_{N,s} \left( \int_{\R^N} U_{0,1}^p \diff y \right) + \mathcal O(\lambda^{-N}) = c_{N,s} A_{N,s} + o( \lambda^{-2s}). \]
The second term is, by Lemma \ref{lemma Up-k Hk}, 
\begin{align*}
- c_{N,s} \lambda^{-\ns} \int_\Omega U_\xl^{p-1} H_a(x, \cdot) \diff y &= - c_{N,s} a_{N,s} \phi_a(x) \lambda^{-N+2s} + c_{N,s} d_{N,s} b_{N,s} a(x) \lambda^{-2s} + o( \lambda^{-2s}). 
\end{align*}
The third term will be combined with a term coming from $\int_\Omega (a + \eps V) \psi_\xl^2 \diff y$, see below. 

The fourth term can be bounded, by Lemma \ref{lemma int Uq} and recalling $\|f_\xl\|_\infty \lesssim \lambda^{-\frac{N+4-2s}{2}}$ from Lemma \ref{lemma decomp PU}, by 
\begin{align*}
 \left| \int_\Omega f_\xl \left( c_{N,s} U_\xl^{p-1} - \lambda^{-\ns} a G_a(x,\cdot) \right) \diff y \right| \lesssim \|f_\xl\|_\infty \left( \|U_\xl\|_{p-1}^{p-1} + \lambda^{-\ns} \right) \lesssim \lambda^{-N - 2 + 2s} = o(\lambda^{-2s}). 
\end{align*}
Now we treat the potential term. We have
\begin{align*}
\int_\Omega (a + \eps V) \psi_\xl^2 \diff y &= \int_\Omega (a + \eps V)  \left( U_\xl - \lambda^{-\ns} H_a(x, \cdot) - f_\xl \right) ^2 \diff y \\
&= \int_\Omega (a + \eps V) \left( U_\xl - \lambda^{-\ns} H_a(x, \cdot) \right) ^2 \diff y \\
& \quad  - 2 \int_\Omega (a + \eps V) f_\xl  \left( U_\xl - \lambda^{-\ns} H_a(x, \cdot) \right)\diff y + \int_\Omega (a + \eps V) f_\xl^2 \diff y .
\end{align*}
Similarly to the above, the terms containing $f_\xl$ are bounded by 
\begin{align*}
\left| \int_\Omega (a + \eps V) f_\xl  \left( U_\xl - \lambda^{-\ns} H_a(x, \cdot) \right)\diff y \right| \lesssim \|f_\xl\|_\infty (\|U_\xl\|_1 - \lambda^{-\ns}) \lesssim  \lambda^{-N - 2 + 2s} = o(\lambda^{-2s})
\end{align*}
and 
\begin{align*}
\left| \int_\Omega (a + \eps V) f_\xl^2 \diff y  \right|  \lesssim \|f_\xl\|^2_\infty \lesssim \lambda^{-N-4+2s} = o(\lambda^{-2s}).
\end{align*}
Finally, we combine the main term with the third term in the expansion of $\|\ds \psi_\xl\|_2^2$ from above. Recalling that 
\begin{equation}
    \label{U Ga h identity}
    U_\xl - \lambda^{-\ns} H_a(x, \cdot) - \lambda^{-\ns} G_a(x, \cdot) = U_\xl - \lambda^{-\ns} |x-y|^{-N+ 2s} = -\lambda^\ns h(\lambda(x-y))
\end{equation} 
with $h$ as in Lemma \ref{lemma h}, we get
\begin{align*}
&\quad  - \lambda^{-\ns} \int_\Omega a G_a(x, \cdot) \left( U_\xl - \lambda^{-\ns} H_a(x, \cdot) \right) \diff y  + \int_\Omega (a + \eps V) \left( U_\xl - \lambda^{-\ns} H_a(x, \cdot) \right) ^2 \diff y \\
&= - \int_\Omega a \left( U_\xl- \lambda^{-\ns} H_a(x, \cdot) \right)  \lambda^{\frac{\Ns}{2}} h(\lambda(x-y)) \diff y + \eps \int_\Omega V  \left( U_\xl- \lambda^{-\ns} H_a(x, \cdot) \right)^2 \diff y. 
\end{align*}
Since 
\[ \int_\Omega a H_a(x, \cdot) h(\lambda(x-y)) \diff y \lesssim \lambda^{-N} \|h\|_{L^1(\R^N)} =o(\lambda^{-2s}),  \]
by Lemma \ref{lemma T xl} we have
\[ - \int_\Omega a \left( U_\xl- \lambda^{-\ns} H_a(x, \cdot) \right)  \lambda^{\frac{\Ns}{2}} h(\lambda(x-y)) \diff y = -\alpha_{N,s} \lambda^{-2s} a(x) + o(\lambda^{-2s}). \]
Moreover, again by \eqref{U Ga h identity}, and using that $h \in L^2(\R^N)$ by Lemma \ref{lemma h}, 
\begin{align*}
& \qquad \eps \int_\Omega V  \left( U- \lambda^{-\ns} H_a(x, \cdot) \right)^2 \\
&= \eps \lambda^{-N + 2s} \int_\Omega V G_a(x, \cdot)^2 \diff y - 2 \eps \int_\Omega V G_a(x, \cdot) h(\lambda(x-y)) \diff y + \eps  \lambda^{N-2s} \int_\Omega V h(\lambda(x-y))^2 \diff y,
\end{align*}
with 
\begin{align*}
\eps \int_\Omega V G_a(x, \cdot) h(\lambda(x-y)) \diff y \lesssim \eps \lambda^{-N/2} \|G_a(x, \cdot)\|_2 \|h\|_2 = o(\eps \lambda^{-N+2s})
\end{align*}
and 
\begin{align*}
\eps  \lambda^{N-2s} \int_\Omega V h(\lambda(x-y))^2 \diff y \lesssim \eps \lambda^{-2s} \|h\|_2^2 = o(\eps \lambda^{-N+2s}).
\end{align*}
This completes the proof of the claimed expansion \eqref{exp_num}. 

\textbf{Step 2:} \emph{Expansion of the denominator.} Recall $p = \frac{2N}{N-2s}$. Firstly, writing $\psi_\xl = U_\xl - \lambda^{-\ns} H_a(x, \cdot) - f_\xl$, we have 
\[ \int_\Omega \psi_\xl^p \diff y = \int_\Omega (U_\xl - \lambda^{-\ns} H_a(x, \cdot))^p \diff y + \mathcal O(\|U_\xl - \lambda^{-\ns} H_a(x, \cdot)\|_{p-1}^{p-1} \|f_\xl\|_\infty + \|f_\xl\|_\infty^p). \]
Using Lemma \ref{lemma int Uq} and the bound  $\|f_\xl\|_\infty \lesssim \lambda^{-\frac{N+4-2s}{2}}$, we deduce  that the remainder term is $o(\lambda^{-2s})$. To evaluate the main term, from Taylor's formula for $t \mapsto t^p$, we have
\[ (a + b)^p = a^p - pa^{p-1} b + \sum_{k=2}^K {p \choose k} a^{p-k} b^k + \mathcal O(a^{p-K-1} |b|^{K+1} + |b|^p). \]
Here, ${p \choose k} := \frac{\Gamma(p+1)}{\Gamma(p-k+1) \Gamma(k+1)}$ is the generalized binomial coefficient and $K = \lfloor \frac{2s}{N-2s} \rfloor$ as in the statement of the theorem. Applying this with $a = U_\xl(y)$, $b = - \lambda^{-\ns} H_a(x, \cdot)$, we find 
\begin{align*}
\int_\Omega (U_\xl - \lambda^{-\ns} H_a(x, \cdot))^p \diff y &=  \int_\Omega U_\xl^p \diff y - p \lambda^{-\ns} \int_\Omega U_\xl^{p-1} H_a(x, \cdot) \diff y \\
&\quad + \mathcal O\left( \lambda^{-N+2s} \int_\Omega U_\xl^{p-2} H_a(x, \cdot)^2 \diff y \right) + \mathcal O\left( \lambda^{-N} \int_\Omega H_a(x, \cdot)^p \diff y \right) .
\end{align*} 
By Lemma \ref{lemma Up-k Hk}, the claimed expansion \eqref{exp_den} follows. 

\textbf{Step 3:} \emph{Expansion of the quotient.} 
For $\alpha = 2/p \in (0,1)$, and fixed $a > 0$, we again use the Taylor expansion 
\[ (a + b)^{-\alpha} = a^{-\alpha} - \alpha a^{-\alpha - 1} b + \sum_{k=2}^K {{-\alpha} \choose k} a^{-\alpha - k} b^k + \mathcal O(b^{K+1}). \]
By Step 2, we may apply this with $a = A_{N,s}$ and $b = - p a_{N,s} \phi_a(x) \lambda^{-N+2s} + \mathcal T_1 (\phi_a(x), \lambda) + p d_{N,s} b_{N,s} a(x) \lambda^{-2s}  + o(\lambda^{-2s})$. Since $b = \mathcal O(\lambda^{-N+2s})$ and $K+1 > \frac{2s}{N-2s}$, we have $\mathcal O(b^{K+1}) = o(\lambda^{-2s})$ and thus 
\begin{align}
    \left( \int_\Omega \psi_\xl^p \diff y \right)^{-2/p} &= A_{N,s}^{-\frac{N-2s}{N}} + A_{N,s}^{-\frac{2(N-s)}{N}} \left( 2 a_{N,s} \phi_a(x) \lambda^{-N+2s} - 2 d_{N,s} b_{N,s} a(x) \lambda^{-2s} \right) \\
    & \qquad + \mathcal T_3(\phi_a(x), \lambda) + o(\lambda^{-2s}), 
\end{align}
for some term $\mathcal T_3(\phi, \lambda)$ as in \eqref{T i definition}. Multiplying this expansion with \eqref{exp_num}, we obtain
\begin{align*}
\mathcal S_{a + \eps V} [\psi_\xl] &= c_{n,s} A_{N,s}^\frac{2s}{N} + A_{N,s}^{-\frac{N-2s}{N}} \Big[ a_{N,s} c_{N,s} \phi_a(x) \lambda^{-N+2s} + \mathcal T_2(\phi_a(x),\lambda) \\
&\quad - a(x) \lambda^{-2s} (\alpha_{N,s} + c_{N,s} d_{N,s} b_{N,s}) + \eps \lambda^{-N+2s} Q_V(x) \Big]\\
&\quad  + o( \lambda^{-2s}) + o(\eps \lambda^{-N+2s}). 
\end{align*}
By integrating the equation $(-\Delta)^s U_{0,1} = c_{N,s} U_{0,1}^{p-1}$ and using the fact that $U_{0,1}$ minimizes the Sobolev quotient on $\R^N$ (or by a computation on the numerical values of the constants given in Lemma \ref{lemma constants}), we have $c_{N,s} A_{N,s}^\frac{2s}{N} = S$.  Hence, this is the expansion claimed in \eqref{exp_quot}. 
\end{proof}

\section{Proof of the lower bound I: a first expansion}
\label{sec:LB1}

\subsection{Non-existence of a minimizer for $S(a)$}
\label{ssec:nonex-min}

In this section, we prove that for a critical potential $a$, the infimum $S(a)$ is not attained. As we will see in Section \ref{ssec:profiledeco}, this implies the important basic fact that the minimizers for $S(a+ \eps V)$ must blow up as $\eps \to 0$. 

The following is the main result of this section. 

\begin{proposition}[Non-existence of a minimizer for $S(a)$]
\label{proposition nonex-min}
Suppose that $a \in C(\overline{\Omega})$ is a critical potential. Then 
\[S(a) =  \inf_{u \in \ths, u \not \equiv 0} \frac{\int_\Omega |\ds u|^2 \diff y + \int_\Omega a u^2 \diff y}{\|u\|_\frac{2N}{N-2s}^2} \]
is not achieved. 
\end{proposition}

For $s=1$, Proposition \ref{proposition nonex-min} was proved by Druet \cite{Druet2002} and we follow his strategy. The feature that makes the generalization of \cite{Druet2002} to $s \in (0,1)$ not completely straightforward is its use of the product rule for ordinary derivatives. Instead, we shall use the identity  
\begin{equation}
    \label{frac prod rule}
    \Ds (uv) = u \Ds v + v \Ds u - I_s(u,v),
\end{equation}
where 
\[ I_s(u,v)(x) := C_{N,s} P.V. \int_{\R^N} \frac{(u(x) - u(y))(v(x) - v(y))}{|x-y|^{N+2s}} \diff y,  \]
with $C_{N,s}$ as in \eqref{eq:constant}. While the relation \eqref{frac prod rule} can be verified by a simple computation (see e.g. \cite[Lemma 20.2]{MR3916700}), it leads to more complicated terms than those arising in Druet's proof. To be more precise, the term $\int_\Omega u^2 |\nabla \varphi|^2$ from \cite{Druet2002} is replaced by the term $\mathcal I(\varphi)$ defined in \eqref{definition I varphi}, which is more involved to evaluate for the right choice of $\varphi$.

\begin{proof}
[Proof of Proposition \ref{proposition nonex-min}]
For the sake of finding a contradiction, we suppose that there exists $u$ which achieves $S(a)$, normalized so that 
\begin{equation}
    \label{minimizer norm}
    \int_\Omega u^{\frac{2N}{N-2s}} \diff y = 1. 
\end{equation}
Then $u$ satisfies the equation 
\begin{equation}
    \label{eq minimizer}
    \Ds u + a u = S u^\frac{N-2s}{N-2s}, 
\end{equation}
with Lagrange multiplier $S = S_{N,s}$ equal to the Sobolev constant. (Indeed, this value is determined by integrating the equation against $u$ and using \eqref{minimizer norm}.)

Since $S(a) = S$, we have, for every $\varphi \in C^\infty(\R^N)$ and $\eps > 0$, and abbreviating $p = \frac{2N}{N-2s}$, 
\begin{equation}
    \label{ineq variational}
    S \left( \int_\Omega \left(u(1+ \eps \varphi)\right)^p \diff y \right)^\frac{2}{p} \leq \int_{\R^N} |\ds \left(u(1 + \eps \varphi)\right)|^2 \diff y + \int_\Omega a u^2 (1+ \eps \varphi)^2 \diff y. 
\end{equation}
We shall expand both sides of \eqref{ineq variational} in powers of $\eps$. For the left side, a simple Taylor expansion together with \eqref{minimizer norm} gives
\begin{equation}
    \label{left side expansion druet}
    \left( \int_\Omega \left(u(1+ \eps \varphi) \right)^p \diff y \right)^\frac{2}{p} = 1 + 2 \eps \int_\Omega u^p \varphi \diff y + \eps^2 \left( (p-1) \int_\Omega u^p \varphi^2 \diff y - (p-2)  \left( \int_\Omega u^p \varphi \diff y \right)^2 \right) + o(\eps^2). 
\end{equation}
Expanding the right side is harder and we need to invoke the fractional product rule \eqref{frac prod rule}. Firstly, integrating by parts we have \[ \int_{\R^N} |\ds \left(u(1 + \eps \varphi)\right)|^2 \diff y  = \int_{\R^N} u (1 + \eps \varphi) \Ds  \left(u(1 + \eps \varphi)\right) \diff y.   \]
By \eqref{frac prod rule}, we can write 
\begin{align*}
    \Ds  \left(u(1 + \eps \varphi)\right) = (1 + \eps \varphi) \Ds u + \eps u \Ds \varphi - \eps I_s (u, \varphi).
\end{align*}
Hence 
\[ \int_{\R^N} |\ds \left(u(1 + \eps \varphi)\right)|^2 \diff y = \int_\Omega u \Ds u (1+ \eps \varphi)^2 + \eps^2 \mathcal I(\varphi), \]
where we write 
\begin{align}
\label{definition I varphi}
    \mathcal I(\varphi) := \eps^{-1} \int_\Omega u (1 + \eps \varphi) \left( u \Ds \varphi - I_s(u, \varphi) \right) \diff y. 
\end{align}
Writing out $\Ds \varphi$ as the singular integral given by \eqref{frac-lap-sg-int}, we obtain (we drop the principal value for simplicity)
\begin{align}
    \mathcal I(\varphi) &= \eps^{-1} C_{N,s} \iint_{\R^N \times \R^N} u(x) u(y) (1 + \eps \varphi(x)) \frac{\varphi(x) - \varphi(y)}{|x-y|^{N+2s}} \diff x \diff y \label{I varphi definition} \\
    &= \frac{C_{N,s}}{2} \iint_{\R^N \times \R^N} u(x) u(y) \frac{|\varphi(x) - \varphi(y)|^2}{|x-y|^{N+2s}} \diff x \diff y.  \nonumber
\end{align}
The last equality follows by symmetrizing in the $x$ and $y$ variables. 

Thus we can write the right side of \eqref{ineq variational} as 
\begin{align*}
   &\qquad   \int_{\R^N} |\ds \left(u(1 + \eps \varphi)\right)|^2 \diff y + \int_\Omega a u^2 (1+ \eps \varphi)^2 \diff y \\
   &= \int_\Omega u (\Ds u + a u ) (1 + \eps \varphi)^2 \diff y + \eps^2 \mathcal I(\varphi) \\
   &= S \int_\Omega u^p (1 + \eps \varphi)^2 \diff y + \eps^2 \mathcal I(\varphi),
\end{align*}
where we used equation \eqref{eq minimizer}. After expanding the square $(1 + \eps \varphi)^2$, the terms of orders 1 and $\eps$ on both sides of \eqref{ineq variational} cancel precisely. For the coefficients of $\eps^2$, we thus recover the inequality 
\begin{equation}
    \label{hessian ineq}
    \int_\Omega u^p \varphi^2 \diff y \leq \frac{1}{S(p-2)} \mathcal I(\varphi) + \left( \int_\Omega u^p \varphi \diff y \right)^2. 
\end{equation} 

We now make a suitable choice of $\varphi$, which turns \eqref{hessian ineq} into the desired contradiction. As in \cite{Druet2002}, we choose
\[ \varphi_i(y) := (\mathcal S(y))_i, \quad i = 1, ..., N+1, \]
where $\mathcal S: \R^N \to \mathbb S^N$ is the (inverse) stereographic projection, i.e. \cite[Sec. 4.4]{LiLo2001}
\begin{equation}
    \label{ste pro definition}
    \varphi_i = \frac{2y_i}{1 + |y|^2} \quad \text{ for } i = 1,..., N, \qquad \varphi_{N+1} = \frac{1- |y|^2}{1 + |y|^2}. 
\end{equation} 
Moreover we may assume (up to scaling and translating $\Omega$ if necessary) that the balancing condition 
\begin{equation}
    \label{balance druet}
     \int_\Omega u^p \varphi_i \diff y = 0, \qquad i = 1,..., N+1 
\end{equation}
is satisfied. Since \cite{Druet2002} is rather brief on this point, we include some details in Lemma \ref{lemma fixed point druet} below for the convenience of the reader. 

By definition, we have $\sum_{i = 1}^{N+1} \varphi_i^2 = 1$. Testing \eqref{hessian ineq} with $\varphi_i$ and summing over $i$ thus yields, by \eqref{balance druet},
\begin{equation}
    \label{hessian ineq 2}
    1 = \int_\Omega u^p \diff y \leq \frac{1}{S(p-2)} \sum_{i = 1}^{N+1} \mathcal I(\varphi_i). 
\end{equation}
To obtain a contradiction and finish the proof, we now show that $\sum_{i = 1}^{N+1} \mathcal I(\varphi_i) < S(p-2)$. By definition of $\varphi_i$, we have 
\begin{equation}
    \label{sum I RN}
    \sum_{i=1}^{N+1}\mathcal I(\varphi_i) = \frac{C_{N,s}}{2} \iint_{\R^N \times \R^N} u(x) u(y) \frac{|\mathcal S(x) - \mathcal S(y)|^2}{|x-y|^{N+2s}} \diff x \diff y .
\end{equation} 
To evaluate this integral further, we pass to $\mathbb S^N$. Set $J_{\mathcal S^{-1}}(\eta) := \det D \mathcal S^{-1} (\eta)$ and define 
\[ U(\eta) := u(\mathcal S^{-1}(\eta)) J_{\mathcal S^{-1}} (\eta)^\frac{1}{p}, \]
so that $\int_{\mathbb S^N} U^p \diff \eta = 1$. Since the distance transforms as 
\[ |\mathcal S^{-1}(\eta) - \mathcal S^{-1} (\xi)| = J_{\mathcal S^{-1}}(\eta)^\frac{1}{2N} |\eta - \xi| J_{\mathcal S^{-1}}(\xi)^\frac{1}{2N}, \]
changing variables in \eqref{sum I RN} gives
\begin{equation}
    \label{sum I SN}
    \sum_{i=1}^{N+1} \mathcal I(\varphi_i) = \frac{C_{N,s}}{2} \iint_{\mathbb S^N \times \mathbb S^N} \frac{U(\eta) U(\xi)}{|\eta - \xi|^{N+2s -2}} \diff \eta \diff \xi
\end{equation}
By applying first Cauchy--Schwarz and then Hölder's inequality, we estimate
\begin{align}
    \sum_{i=1}^{N+1} I(\varphi_i) &\leq \frac{C_{N,s}}{2} \iint_{\mathbb S^N \times \mathbb S^N} \frac{U(\eta)^2}{|\eta - \xi|^{N+2s -2}} \diff \eta \diff \xi \nonumber \\
    &= \frac{C_{N,s}}{2}  \delta_{N,s} \int_{\mathbb S^N} U(\eta)^2 \diff \eta < \frac{C_{N,s}}{2}  \delta_{N,s} |\mathbb S^N|^\frac{2s}{N}, \label{hölder druet proof}
\end{align} 
where the last inequality is strict. Indeed, $U$ vanishes near the south pole of $\mathbb S^N$, hence there cannot be equality in Hölder's inequality applied with the functions $U^2$ and $1$. Moreover, in the above we abbreviated
\[ \delta_{N,s} := \int_{\mathbb S^N} \frac{1}{|\eta - \xi|^{N+2s-2}} \diff \xi \]
(note that this number  is independent of $\eta \in \mathbb S^N$). By transforming back to $\R^n$ and evaluating a Beta function integral, the explicit value of $\delta_{N,s}$ can be computed explicitly to be
\[ \delta_{N,s} = 2^{2-2s} \pi^{N/2} \frac{\Gamma(1-s)}{\Gamma(\frac{N}{2} + 1 - s)}. \]
Inserting this into estimate \eqref{hölder druet proof}, as well as the explicit values of $C_{N,s}$ given in \eqref{eq:constant} and of $S_{N,s}$ given in \eqref{eq:sobolevconstant}, a direct computation then gives
\[ \frac{ \sum_{i=1}^{N+1}\mathcal  I(\varphi_i) }{S_{N,s}} < \frac{1}{S_{N,s}} \frac{C_{N,s}}{2}  \delta_{N,s} |\mathbb S^N|^\frac{2s}{N} = \frac{s 2^{2-2s}}{N-2s} \left(2  \frac{\Gamma(N) \Gamma(\frac 12)}{\Gamma(\frac{N+1}{2})\Gamma(\frac{N}{2})} \right)^\frac{2s}{N} . \]
It can be easily shown by induction over $N$ that 
\[ 2 \frac{\Gamma(N) \Gamma(\frac 12)}{\Gamma(\frac{N+1}{2})\Gamma(\frac{N}{2})}  = 2^N \]
for every $N \in \N$, and hence 
\[  \sum_{i=1}^{N+1} \mathcal I(\varphi_i) < S \frac{4s}{N-2s} = S(p-2). \]
This is the desired contradiction to \eqref{hessian ineq 2}.
\end{proof}

Here is the lemma that we referred to in the previous proof. It expands an argument sketched in \cite[Step 1]{Druet2002}. To emphasize its generality, instead of $u^p$ we state it for a general nonnegative function $h$ with $\int_\Omega h = 1$.

\begin{lemma}
    \label{lemma fixed point druet}
    Let $\Omega \subset \R^N$ be an open bounded set and $0 \leq h \in L^1(\Omega)$ with $\|h\|_1 = 1$. Then there is $(y, t) \in \R^N \times (0, \infty)$ such that 
    \begin{align*}
        F(y,t) &:= \int_\Omega h(x) \frac{2t(x-y)}{1+ t^2|x-y|^2} \diff x = 0, \\
        G(y,t) &:= \int_\Omega h(x) \frac{1 - t^2|x-y|^2}{1 + t^2 |x-y|^2} \diff x = 0. 
    \end{align*} 
\end{lemma}

\begin{proof}
Define $H: \R^N \times \R \to \R^{N+1}$ by 
\[ H(y,s) := \left( F\left(y, \frac{s + \sqrt{s^2 + 4}}{2} \right) + y, G\left(y, \frac{s + \sqrt{s^2 + 4}}{2} \right) + s \right). \]
We claim that 
\begin{equation}
    \label{brouwer bound}
    |H(y,s)| \leq |y|^2 + s^2
\end{equation}
whenever $|y|^2 + s^2$ is large enough. Thus for large enough radii $R >0$, the map $H$ sends $B(0,R) \subset \R^{N+1}$ into itself. By the Brouwer fixed point theorem, $H$ has a fixed point $(y,s)$. Then the pair  $(y,\frac{s + \sqrt{s^2 + 4}}{2})$ satisfies the property stated in the lemma. 

To prove \eqref{brouwer bound}, it is more natural to set $t := \frac{s + \sqrt{s^2 + 4}}{2} > 0$, so that $s = t - t^{-1}$. By writing out $|H(y,s)|^2$, \eqref{brouwer bound} is equivalent to 
\begin{equation}
    \label{brouwer bound 2}
    2 y \cdot F(y,t) + 2 (t - t^{-1}) G(y,t) + |F(y,t)|^2 + |G(y,t)|^2 \leq 0
\end{equation}
whenever $|y|^2 + (t - t^{-1})^2$ is large enough. 

First, it is easy to see that $y \cdot F(y,t)$, $F(y,t)$ and $G(y,t)$ are bounded in absolute value uniformly in $(y,t) \in \R^N \times (0,\infty)$. Moreover, there is $C > 0$ such that 
\[ G(y,t) = 1 - 2 \int_\Omega h(x) \frac{t^2|x-y|^2}{1 + t^2|x-y|^2} \diff x 
\begin{cases}
\geq \frac12 & \text{ if } 0 <  t \leq 1/C, \\
\leq - \frac12 & \text{ if } t \geq C.
\end{cases}
\]
Therefore $(t - t^{-1}) G(y,t) \to -\infty$ as $t \to 0$ or $t \to \infty$. Thus \eqref{brouwer bound 2} holds whenever $(t - t^{-1})^2$ is large enough. 

Thus we assume in the following that $t \in [1/C, C]$ and prove that \eqref{brouwer bound 2} holds if $|y|$ is large enough. For convenience, fix some sequence $(y,t)$ with $|y| \to \infty$ and $t \to t_0 \in (0,\infty)$. Then $|F(y,t)| \to 0$ and $G(y,t) \to -1$. Moreover, since $\Omega$ is bounded, $\frac{|x-y|}{|y|} \to 1$ uniformly in $x \in \Omega$ and hence
\begin{align*}
    2 y \cdot F(y,t) = - \int_\Omega h(x) \frac{4t|y|^2}{1 + t^2|x-y|^2} + \mathcal O \left( \int_\Omega h(x) \frac{t|y|}{1 + t^2 |x-y|^2} \diff x \right) \to - \frac{4}{t_0}.
\end{align*}
Altogether, the quantity on the left side of \eqref{brouwer bound 2} thus tends to $-2 t_0 - 2 t_0^{-1} + 1 \leq -3 < 0$, which concludes the proof of \eqref{brouwer bound 2}. 
\end{proof}

\subsection{Profile decomposition}
\label{ssec:profiledeco}

The following proposition gives an asymptotic decomposition of a general sequence of normalized (almost) minimizers of $S(a+\eps V)$. 

\begin{proposition}[Profile decomposition]
\label{proposition decomposition}
Let $a \in C(\overline{\Omega})$ be critical and let $V \in C(\overline{\Omega})$ be such that $\mathcal N_a(V) \neq \emptyset$. Suppose that $(u_\eps) \subset \ths$ is a sequence such that
\begin{equation}
\label{almost_min prop}
\lim_{\eps \to 0} \frac{\mathcal S_{a+\eps V}[u_\eps] - S(a+\eps V)}{S - S(a+\eps V)} = 0 \quad \text{ and } \quad \int_\Omega u_\eps^p \diff y = \int_{\R^N} U_{0,1}^p \diff y.  
\end{equation}
Then there are sequences $(x_\eps) \subset \Omega$, $(\lambda_\eps) \subset (0, \infty)$, $(w_\eps) \subset T_{x_\eps, \lambda_\eps}^\bot$, and $(\alpha_\eps) \subset \R$ such that, up to extraction of a subsequence,
\begin{equation}
\label{dec_PU+w} u_\eps = \alpha_\eps \left( PU_{x_\eps, \lambda_\eps} + w_\eps \right). 
\end{equation}
Moreover, as $\eps \to 0$, we have 
\begin{align*}
\|(-\Delta)^{s/2} w_\eps\|_2 &\to 0, \\
d(x_\eps) \lambda_\eps &\to \infty, \\
x_\eps &\to x_0, \\
\alpha_\eps &\to \pm 1.
\end{align*}
\end{proposition}

In all of the following, we shall always work with a sequence $u_\eps$ that satisfies the assumptions of Proposition \ref{proposition decomposition}. For readability, we shall often drop the index $\eps$ from $\alpha_\eps$, $x_\eps$, $\lambda_\eps$ and $w_\eps$, and write $d:= d_\eps := d(x_\eps)$. Moreover, we make the convention that we always assume the strict inequality 
\begin{equation}
    \label{strict ineq S(a+eps V) < S}
    S(a + \eps V) < S. 
\end{equation}
In Theorems \ref{thm:13} and \ref{thm:17} we assume $\mathcal N_a(V) \neq \emptyset$, so assumption \eqref{strict ineq S(a+eps V) < S} is certainly justified in view of Corollary \ref{corollary upper bound}. For Theorem \ref{thm:14}, where we assume $\mathcal N_a(V) = \emptyset$, we discuss the role of assumption \eqref{strict ineq S(a+eps V) < S} in the proof of that theorem in Section \ref{sec:proof}.

\begin{proof}
\textbf{Step 1.} We derive a preliminary decomposition in terms of the Sobolev optimizers $U_\zl$ and without orthogonality condition on the remainder, see \eqref{dec_cc_prelim} below. 

The assumptions imply that the sequence $(u_\eps)$ is bounded in $\ths$, hence up to a subsequence we may assume $u_\eps \rightharpoonup u_0$ for some $u_0 \in \ths$. By the argument given in \cite[Proof of Proposition 3.1, Step 1]{FrKoKo2021}, the fact that $\mathcal S_{a + \eps V}[u_\eps] \to S(a) = S$ implies that either $u_0$ is a minimizer for $S(a)$, unless $u_0 \equiv 0$. Since such a minimizer does not exist by Proposition \ref{proposition nonex-min}, we conclude that in fact $u_\eps \rightharpoonup 0$ in $\ths$. 

By Rellich's theorem, $u_\eps \to 0$ strongly in $L^2(\Omega)$, in particular $\int_\Omega (a + \eps V) u_\eps^2 = o(1)$. The assumption \eqref{almost_min} therefore implies that $(u_\eps)$ is a minimizing sequence for the Sobolev quotient $\int_{\R^N} |\ds u|^2 \diff y / \|u\|^2_\frac{2N}{N-2s}$. Therefore the assumptions of \cite[Theorem 1.3]{PaPi2014} are satisfied, and we may conclude by that theorem that there are sequences $(z_\eps) \subset \R^N$, $(\mu_\eps) \subset (0, \infty)$, $(\sigma_\eps)$ such that 
\[ \mu_\eps ^{-\frac{N-2s}{2}} u_\eps (z_\eps + \mu_\eps^{-1} \cdot) \to \beta U_{0,1} \]
in $\dot{H}^s(\R^N)$, for some $\beta \in \R$. By the normalization condition from \eqref{almost_min}, $\beta \in \{ \pm 1\}$. Now, a change of variables $y = z_\eps + \mu_\eps^{-1} x$ implies 
\begin{equation}
\label{dec_cc_prelim} u_\eps(y) = U_{z_\eps, \mu_\eps}(y) + \sigma_\eps, 
\end{equation}
where still $\sigma_\eps \to 0$ in $\dot{H}^s(\R^N)$, since the $\dot{H}^s(\R^N)$-norm is invariant under this change of variable. 

Moreover, since $\Omega$ is smooth, the fact that 
\[ \int_{\mu_\eps \Omega + z_\eps} U_{0,1}^p \diff y = \int_\Omega U_{z_\eps, \mu_\eps}^p \diff y  =  \int_{\R^N} U_{0,1}^p \diff y + o(1) \]
implies $\mu_\eps \text{dist}(z_\eps, \R^N \setminus \Omega) \to \infty$.  

\textbf{Step 2.}  We make the necessary modifications to derive \eqref{dec_PU+w} from \eqref{dec_cc_prelim}. The crucial argument is furnished by \cite[Proposition 4.3]{AbChHa2018}, which generalizes the corresponding statement by Bahri and Coron \cite[Proposition 7]{BaCo1988} to fractional $s \in (0,1)$. It states the following. Suppose that $u \in \ths$ with $\|u\|_\ths = A_{N,s}$ satisfies 
\begin{equation} 
\label{condition_bahri_coron}
\inf \left \{ \| \ds ( u -  PU_\xl)\|_2 \, : \, x \in \Omega, \, \lambda d(x) > \eta^{-1} \right \} < \eta. 
\end{equation}
for some $\eta > 0$. Then if $\eta$ is small enough, the minimization problem
\begin{equation}
\label{inf_bahri_coron}
\inf \left \{ \| \ds( u - \alpha PU_\xl) \|_2 \, : \, x \in \Omega, \, \lambda d(x) > (4\eta)^{-1}, \, \alpha \in (1/2, 2) \right \}
\end{equation}
has a unique solution.

By the decomposition from Step 1 and Lemma \ref{lemma decomp PU}, we have 
\[ \| \ds (u_\eps - PU_{z_\eps, \mu_\eps})\|_2 \leq \|\ds (U_{z_\eps, \mu_\eps} - PU_{z_\eps, \mu_\eps}) \|_2 + \|\ds \sigma_\eps\|_2  \to 0 \]
as $\eps \to 0$, so that \eqref{condition_bahri_coron} is satisfied by $u_\eps$ for all $\eps$ small enough, with a constant $\eta_\eps$ tending to zero. We thus obtain the desired decomposition 
\[ u_\eps = \alpha_\eps (PU_{x_\eps, \lambda_\eps} + w_\eps) \]
by taking $(x_\eps, \lambda_\eps, \alpha_\eps)$ to be the solution to \eqref{inf_bahri_coron} and $w_\eps := \alpha_\eps^{-1} u_\eps -  PU_{x_\eps, \lambda_\eps}$. To verify the claimed asymptotic behavior of the parameters, note that since $\eta_\eps \to 0$, by definition of the minimization problem \eqref{inf_bahri_coron}, we have $\|\ds w_\eps\|_2 < \eta_\eps \to 0$ and $\lambda_\eps d(x_\eps) > (4 \eta_\eps)^{-1} \to \infty$. Since $\Omega$ is bounded, the convergence $x_\eps \to x_0 \in \overline{\Omega}$ is ensured by passing to a suitable subsequence. Finally, using \eqref{almost_min} we have
\[ \int_{\R^N} U_{0,1}^p \diff y  = \int_\Omega u_\eps^p \diff y  = |\alpha_\eps|^p \int_\Omega PU_{x_\eps, \lambda_\eps} \diff y + o(1)  = |\alpha_\eps|^p \int_{\R^N} U_{0,1}^p \diff y + o(1), \]
which implies $\alpha_\eps = \pm 1 + o(1)$. 
\end{proof}

\subsection{Coercivity}
\label{ssec:coercivity}

In the following sections, our goal is to improve the bounds from Proposition \ref{proposition decomposition} step by step. 

The following inequality, and its improvement in Proposition \ref{proposition coerc with a} below, will be central. For $s=1$, these inequalities are due to Rey  \cite[Eq. (D.1)]{Rey1990} and Esposito \cite[Lemma 2.1]{Esposito2004}, respectively, whose proofs inspired those given below.  

\begin{proposition}[Coercivity inequality]
\label{proposition coercivity}
For all $x \in \R^n$ and $\lambda > 0$, we have 
\begin{equation}
\label{coerc_ineq_rey}
\|(-\Delta)^{s/2} v\|_2^2 - c_{N,s} (p-1) \int_\Omega U_\xl ^{p-2} v^2 \diff y \geq \frac{4s}{N+2s+2}  \|(-\Delta)^{s/2} v\|_2^2,
\end{equation}
for all $v \in T_\xl^\bot$.
\end{proposition}

As a corollary, we can include the lower order term $\int_\Omega a v^2$, at least when $d(x) \lambda$ is large enough and at the price of having a non-explicit constant on the right side. This is the form of the inequality which we shall use below to refine our error bounds in Sections \ref{ssec:improvedbounds} and \ref{ssec:refinedex}.  

\begin{proposition}[Coercivity inequality with potential $a$]
\label{proposition coerc with a}
    Let $(x_n) \subset \Omega$ and $(\lambda_n) \subset (0, \infty)$ be sequences such that $\dist (x_n, \partial \Omega) \lambda_n \to \infty$. Then there is $\rho> 0$ such that for all $n$ large enough, 
\begin{equation}
\label{coerc_ineq}
\|(-\Delta)^{s/2} v\|_2^2 + \int_\Omega a v^2 \diff y - c_{N,s} (p-1) \int_\Omega U_{x_n, \lambda_n} ^{p-2} v^2 \diff y \geq \rho \|(-\Delta)^{s/2} v\|_2^2,\qquad \text{ for all  } v \in T_{x_n, \lambda_n}^\bot. 
\end{equation}
\end{proposition}

\begin{proof}
Abbreviate $U_n := U_{x_n, \lambda_n}$ and $T_n := T_{x_n, \lambda_n}$. We follow the proof of \cite{Esposito2004} and define
\[ C_n := \inf \left\{ 1 + \int_\Omega a v^2 \diff y - c_{N,s} (p-1) \int_\Omega U_n ^{p-2} v^2 \diff y \, : \, v \in T_n^\bot ,\,  \|\ds v\| = 1 \right \}. \]
Then $C_n$ is bounded from below, uniformly in $n$. We first claim that $C_n$ is achieved whenever $C_n <1$. Indeed, fix $n$ and let $v_k$ be a minimizing sequence. Up to a subsequence, $v_k \rightharpoonup v_\infty$ in $\ths$ and consequently $\|\ds v_\infty \| \leq 1$ and $\int_\Omega a v_k^2 - c_{N,s} (p-1) \int_\Omega U_n^{p-2} v_k^2 \diff y \to \int_\Omega a v_\infty^2 - c_{N,s} (p-1) \int_\Omega U_n^{p-2} v_\infty^2 \diff y$, by compact embedding $\ths \hookrightarrow L^2(\Omega)$. Thus 
\begin{align*}
&\quad (1- C_n) \|\ds v_\infty\|^2 + \int_\Omega a v_\infty^2 \diff y  - c_{N,s} (p-1) \int_\Omega U_n ^{p-2} v_\infty^2 \diff y \\
&\leq (1 - C_n) + \int_\Omega a v_\infty^2 \diff y - c_{N,s} (p-1) \int_\Omega U_n ^{p-2} v_\infty^2 \diff y = 0.
\end{align*}
On the other hand, the left hand side of the above inequality must itself be non-negative, for otherwise $\tilde{v} := v_\infty/\|\ds v_\infty\|$ (notice that $C_n < 1$ enforces $v_\infty \not \equiv 0$) yields a contradiction to the definition of $C_n$ as an infimum. Thus the above inequality must be in fact an equality, whence $\|\ds v_\infty\| = 1$. We have thus proved that $C_n$ is achieved if $C_n < 1$. 

Now, assume for contradiction, up to passing to a subsequence, that $\lim_{n \to \infty} C_n =: L \leq 0$. By the first part of the proof, let $v_n$ be a minimizer satisfying 
\begin{equation}
    \label{eulereq vn}
(1- C_n) \int_{\R^N} \ds v_n \ds w \diff y + \int_\Omega a v_n w \diff y  - c_{N,s} (p-1) \int_\Omega U_n^{p-2} v_n w \diff y = 0 
\end{equation}
for all $w \in T_n^\bot$. Up to passing to a subsequence, we may assume $v_n \rightharpoonup v \in \ths$. We claim that 
\begin{equation}
    \label{eulereq limit v}
    (1-L) \Ds v + a v = 0 \qquad \text{ in } (\ths)'.  
\end{equation}
Assuming \eqref{eulereq limit v} for the moment, we obtain a contradiction as follows. Testing \eqref{eulereq limit v} against $v \in \ths$ gives
\[ \| \ds v\|^2 + \int_\Omega a v^2 \diff y = L \|\ds v\|^2 \leq 0. \]
On the other hand, by coercivity of $\Ds + a$, the left hand side must be nonnegative and hence $v \equiv 0$. By compact embedding, we deduce $v_n \to 0$ strongly in $L^2(\Omega)$ and thus 
\[ C_n = 1 - c_{N,s} (p-1) \int_\Omega U_n^2 v_n^2 \diff y + o(1) \geq \frac{4s}{N+2s+2} + o(1). \]
This is the desired contradiction to $\lim_{n \to \infty} C_n \leq 0$. 

At last, we prove \eqref{eulereq limit v}. Let $\varphi \in \ths$ be given and write $\varphi = u_n + w_n$, with $u_n \in T_n$ and $w_n \in T_n^\bot$. By \eqref{eulereq vn} and using $ \int_{\R^N} \ds v_n \ds u_n = 0$,
\begin{align}
    &\quad  (1- C_n) \int_{\R^N} \ds v_n \ds \varphi \diff y  + \int_\Omega a v_n \varphi \diff y - c_{N,s} (p-1) \int_\Omega U_n^{p-2} v_n \varphi \diff y  \label{euler with varphi} \\
    &=  \int_\Omega a v_n u_n \diff y - c_{N,s} (p-1) \int_\Omega U_n^{p-2} v_n u_n \diff y = \mathcal O(\|\ds u_n\|). \label{euler Hs bound}
\end{align}
On the one hand, we have 
\[ \left|\int_\Omega U_n^{p-2} v_n \varphi \diff y \right| \leq \| U_n^{p-2} \varphi\|_\frac{p}{p-1} \to 0 \]
because $\varphi^\frac{p}{p-1} \in L^{p-1} = (L^\frac{p-1}{p-2}(\Omega))'$ and $U_n^\frac{(p-2)p}{p-1} \rightharpoonup 0$ weakly in $L^\frac{p-1}{p-2}(\Omega)$. Thus, by weak convergence, the expression in \eqref{euler with varphi} tends to
\[  (1- L) \int_{\R^N} \ds v \ds \varphi \diff y + \int_\Omega a v \varphi \diff y \]
as desired. In view of \eqref{euler Hs bound}, the proof of \eqref{eulereq limit v} is thus complete if we can show $\|\ds u_n \| \to 0$. This is again a consequence of weak convergence to zero of the $U_n$. Indeed, by Lemmas  \ref{lemma int Uq} and \ref{lemma Hs norms PU} we have 
\begin{align*}
\left| \int_{\R^N} \ds \frac{PU_n}{\|\ds PU_n\|} \ds \varphi \diff y \right|  &\lesssim  \int_{\R^N} U_n^{p-1} |\varphi| \diff y = o(1) , \\
\left| \int_{\R^N} \ds  \frac{\partial_\lambda PU_n}{\| \ds \pl PU_n\| } \ds \varphi \diff y \right| &\lesssim \lambda \int_{\R^N} U_n^{p-2}  \partial_\lambda U_n \varphi \diff y \lesssim \|U_n^{p-2}\varphi\|_\frac{p}{p-1} = o(1),
\end{align*}
and similarly 
\[ \int_{\R^N} \ds (\lambda^{-N+2s} \partial_{x_i} PU_n) \ds \varphi \diff y = o(1). \]
Here we used again that $U_n^{p-1} \rightharpoonup 0$ in $L^\frac{p}{p-1}$ and $U_n^\frac{(p-2)p}{p-1} \rightharpoonup 0$ in $L^\frac{p-1}{p-2}$ weakly.  $\int_\Omega U^\frac{(p-2)p}{p-1} \varphi^\frac{p}{p-1} \diff y = o(1)$ by weak convergence. 

From the convergence to zero of these scalar products, one can conclude $u_n \to 0$ by using the fact that the $PU_n$, $\pl PU_n$, $\pxi PU_n$ are 'asymptotically orthogonal' by the bounds of Lemma \ref{lemma Hs norms PU}. For a detailed argument, we refer to Lemma \ref{lemma beta gamma delta} below, see also \cite[Lemma 6.1]{FrKoKo2021}. 
\end{proof}

Let us now prepare the proof of Proposition \ref{proposition coercivity}. 
We recall that $\Ste: \R^N \to \Sph \setminus \{S\}$ (where $S = -e_{N+1}$ is the southpole) denotes the inverse stereographic projection defined in \eqref{ste pro definition}, with Jacobian $J_\Ste(x) := \det D \Ste(x) = \left(\frac{2}{1+|x|^2}\right)^N$. 

Given a function $v$ on $\R^N$, we may define a function $u$ on $\Sph$ by setting 
\[ u(\omega) := v_{\Ste^{-1}}(\omega) := v(\Ste^{-1}(\omega)) J_{\Ste^{-1}}(\omega)^\frac{N-2s}{2N}, \qquad \omega \in \Sph \setminus \{S\}. \]
The inverse of this map is of course given by 
\[ v(y) := u_\Ste(y) := u(\Ste(y)) J_\Ste(y)^\frac{N-2s}{2N}, \qquad y \in \R^N. \]
The exponent in the determinantal factor is chosen such that $\|v\|_{L^p(\R^N)} = \|u\|_{L^p(\Sph)}$.

For a basis $(Y_{l,m})$ of $L^2(\Sph)$ consisting of $L^2$-normalized spherical harmonics, write $u \in L^2(\Sph)$ as $u = \sum_{l,m} u_{l,m} Y_{l,m}$ with coefficients $u_{l,m} \in \R$. With 
\begin{equation}
    \label{paneitz eigenvalues} 
    \lambda_l = \frac{\Gamma(l + \frac{N}{2}+ s)}{\Gamma(l + \frac{N}{2}- s)}. 
\end{equation} 
the Paneitz operator $\mathcal P_{2s}$ is defined by
\[ \Ps u := \sum_{l,m} \lambda_l u_{l,m} Y_{l,m} \]
for every $u \in L^2(\Sph)$ such that $\sum_{l,m} \lambda_l u_{l,m}^2 < \infty$. 

It is well-known (see \cite{Beckner1993}) that, for every $v \in C^\infty_0(\R^N)$, we have,
\begin{equation}
\label{relation Ds P2s}
\Ds v(x) = J_\Ste (x)^\frac{N+2s}{2N}  \Ps u (\Ste(x)),
\end{equation}
where $u = v_\Ste$. Thus we have 
\[
    \int_{\R^N} |\ds v|^2 \diff y= \int_{\R^N} v \Ds v \diff y = \int_{\Sph} u \Ps u \diff y = \sum_{l,m} \lambda_l u_{l,m}^2. 
\]
Since $C^\infty_0(\R^N)$ is dense in the space $\mathcal D^{s,2}(\R^N) := \{ v \in L^\frac{2N}{N-2s}(\R^N) \, : \, \ds v \in L^2(\R^N) \}$ (see e.g. \cite{BrGoVa2021}), the equality 
\begin{equation}
    \label{ds Ps norm equality}
    \int_{\R^N} |\ds v|^2 \diff y= \sum_{l,m}^2 \lambda_l u_{l,m}^2 
\end{equation} 
extends to all $v \in \mathcal D^{s,2}(\R^N)$. In particular it holds for $v \in \ths$. 

\begin{proof}
[Proof of Proposition \ref{proposition coercivity}]
We first prove \eqref{coerc_ineq_rey} for $(x,\lambda) = (0,1)$. Let $v \in T_{0,1}^\perp$ and denote $u = v_\Ste$. We claim that the orthogonality conditions on $v$ imply that $u_{l,m} = 0$ for $l = 0,1$. Indeed, e.g. from $v \perp \pl PU_{0,1}$ and recalling $J_\Ste(y) = (\frac{2}{1+|y|^2})^N = 2^N U_{0,1}^\frac{2N}{N-2s}$ we compute 
\begin{align*}
    0 &= \int_{\R^N} \ds v \ds \pl PU_{0,1} \diff y = c_{N,s} \frac{N+2s}{N-2s} \int_{\R^N} v U_{0,1}^\frac{4s}{N-2s} \pl U_{0,1} \diff y \\
    &= c_{N,s} \frac{N+2s}{N-2s} 2^{-\frac{N+2s}{2}} \int_{\R^N} v(y) J_\Ste(y)^\frac{N+2s}{2N} \frac{1-|y|^2}{1+|y|^2} \diff y = c_{N,s} \frac{N+2s}{N-2s} 2^{-\frac{N+2s}{2}} \int_\Sph u(\omega) \omega_{N+1} \diff \sigma(\omega). 
\end{align*}
Analogous calculations show that $v \perp PU_{0,1}$ implies $\int_\Sph u = 0$ and that $v \perp \pxi PU_{0,1}$ implies $\int_\Sph u \omega_i = 0$ for $i = 1,...,N$. Since the functions $1$ and $\omega_i$ ($i = 1,...,N+1$) form a basis of the space of spherical harmonics of angular momenta $l = ´0$ and $l=1$ respectively, we have proved our claim. 

Since the eigenvalues $\lambda_l$ of $\Ps$ are increasing in $l$, changing back variables to $\R^N$, we deduce from \eqref{ds Ps norm equality} that 
\[ 
\int_{\R^N} |\ds v|^2 \diff y = \sum_{l,m} \lambda_l u_{l,m}^2  \geq \lambda_2 \int_\Sph u(\omega)^2 \diff \sigma(\omega) =  2^{2s} \lambda_2 \int_{\R^N} v^2(y) U_{0,1}^{p-2}(y) \diff y.
\]
By an explicit computation using the numerical values of $\lambda_2$ given by \eqref{paneitz eigenvalues} and $c_{N,s}$ given in Lemma \ref{lemma constants}, this is equivalent to 
\begin{equation}
\label{coerc v U01 final}
\|(-\Delta)^{s/2} v\|_2^2 - c_{N,s} (p-1) \int_\Omega U_{0,1}^{p-2} v^2 \diff y \geq \frac{4s}{N + 2s +2}  \|(-\Delta)^{s/2} v\|_2^2,
\end{equation}
which is the desired inequality. 

The case of general $(x, \lambda) \in \Omega \times (0, \infty)$ can be deduced from this by scaling. Indeed, for $v \in T_\xl^\bot$, set $v_\xl(y):= v(x - \lambda^{-1}y)$. Then $v_\xl \in T_{0,1}^\bot$ with respect to the set $\lambda(x - \Omega)$, so that by the above $v_\xl$ satisfies 
\[ \|(-\Delta)^{s/2} v_\xl\|_2^2 - c_{N,s} (p-1) \int_{\lambda (x -\Omega)} U_{0,1} ^{p-2} v_\xl^2 \diff y \geq \frac{4s}{N + 2s +2}  \|(-\Delta)^{s/2} v_\xl\|_2^2. \]
Changing back variables now yields \eqref{coerc v U01 final}.  
\end{proof}

\subsection{Improved a priori bounds}
\label{ssec:improvedbounds}

The main section of this section is the following proposition, which improves Proposition \ref{proposition decomposition}. It states that the concentration point $x_0$ does not lie on the boundary of $\Omega$ and gives an optimal quantitative bound on $w$.  

\begin{proposition}
\label{prop w and d}
As $\eps \to 0$, 
\begin{equation}
\label{bound w}
\|\ds w \| = \mathcal O(\lambda^{-\frac{N-2s}{2}})
\end{equation}
and 
\begin{equation}
\label{bound d-1}
d^{-1} = \mathcal O(1). 
\end{equation}
In particular, $x_0 \in \Omega$. 
\end{proposition}

The proposition will readily follow from the following expansion of $\mathcal S_{a + \eps V}[u_\eps]$ with respect to the decomposition $u_\eps = \alpha (PU_\xl + w)$ obtained in the previous section. 

\begin{lemma}
\label{lemma expansion PU+w}
As $\eps \to 0$, we have 
\begin{align*}
 \mathcal S_{a + \eps V}[u_\eps] &= S + 2^{2s} \pi^{N/2} \frac{\Gamma(s)}{\Gamma(\frac{N-2s}{2})} \left( \frac{S}{c_{N,s}} \right)^{- \frac{N-2s}{2s}} \phi_0(x) \lambda^{-N + 2s} \\
 &\quad + \left( \frac{S}{c_{N,s}} \right)^{- \frac{N-2s}{2s}} \left( \|\ds w\|^2 + \int_\Omega a w^2 \diff y - c_{N,s} (p-1) \int_\Omega U_\xl^{p-2} w^2 \diff y \right) \\
 & \quad + \mathcal O \left(\lambda^{-\frac{N-2s}{2}} \|\ds w\|\right) + o\left((d \lambda)^{-N+2s}\right) + o\left(\|\ds w\|^2\right). 
 \end{align*}
\end{lemma}

\begin{proof}
[Proof of Proposition \ref{prop w and d}]
Using the almost minimality assumption \eqref{almost_min} and the coercivity inequality from Proposition \ref{proposition coerc with a}, the expansion from Lemma \ref{lemma expansion PU+w} yields the inequality
\[ 0 \geq (1+ o(1))( S - S(a + \eps V)) + c \phi_0(x) \lambda^{-N+2s} + c \|\ds w\|^2 +  \mathcal O \left(\lambda^{-\frac{N-2s}{2}} \|\ds w\|\right) + o\left((d \lambda)^{-N+2s}\right) \]
for some $c > 0$. By Lemma \ref{lemma greens fct}, we have the lower bound $\phi_0(x) \gtrsim d^{-N+2s}$. Together with the estimate 
\[ \mathcal O\left(\lambda^{-\frac{N-2s}{2}} \|\ds w\|\right) \leq \delta \|\ds w \|^2 + C_\delta \lambda^{-N+2s}, \]
we obtain by taking $\delta$ small enough
\[ C_\delta \lambda^{-N+2s}  \geq (1+ o(1))( S - S(a + \eps V)) + c (d \lambda)^{-N+2s} + c \|\ds w\|^2. \]
Since all three terms on the right side are nonnegative, the proposition follows. 
\end{proof}

\begin{proof}
[Proof of Lemma \ref{lemma expansion PU+w}]
\textbf{Step 1:} \textit{Expansion of the numerator.  } By orthogonality, we have 
\[ \|\ds (PU_\xl+w)\|^2 = \|\ds PU_\xl\|^2 + \|\ds w\|^2. \]
The main term can be written as 
\begin{align*}
    \| \ds PU_\xl\|^2 &= \int_\Omega PU_\xl \Ds PU_\xl \diff y = c_{N,s} \int_\Omega U_\xl^{p-1} PU_\xl \diff y \\
    &= c_{N,s} \int_\Omega U_\xl^p \diff y + c_{N,s}  \lambda^{-\frac{N-2s}{2}} \int_\Omega U_\xl^{p-1} H_0(x,\cdot) \diff y + \mathcal O(\|f_\xl\|_\infty \int_\Omega U_\xl^{p-1} \diff y),
\end{align*}
where we used $PU_\xl = U_\xl - \lambda^{-\frac{N-2s}{2}} H_0(x,\cdot) + f_\xl$ with $\|f_\xl\|_\infty \lesssim \lambda^\frac{-N+4-2s}{2} d^{-N -2 +2s}$, by Lemma \ref{lemma decomp PU}. Thus 
\[ \|f\|_\infty \int_\Omega U_\xl^{p-1} \diff y \lesssim (d\lambda)^{-N+2s -2} = o((d\lambda)^{-N+2s}). \]
Next, we have 
\[ \int_{\R^N \setminus \Omega} U_\xl^p \diff y \leq \int_{\R^N \setminus B_d} U_\xl^p \diff y \lesssim \int_{d\lambda}^\infty \frac{r^{N-1}}{(1+r^2)^N} \diff r \lesssim (d\lambda)^{-N} = o((d\lambda)^{-N+2s}) \]
and thus
\[  c_{N,s} \int_\Omega U_\xl^p \diff y  = c_{N,s} \|U_{0,1}\|_p^p +   o((d\lambda)^{-N+2s}). \]
Finally, using that $H_0(x,y) = \phi_0(x) + \mathcal O( \|\nabla_y H_0(x,\cdot)\|_\infty |x-y|) = \phi_0(x) + \mathcal O( d^{-N+2s-1} |x-y|)$ by Lemma \ref{lemma greens fct}, we have
\begin{align*}
 \lambda^{-\frac{N-2s}{2}}  \int_\Omega U_\xl^{p-1} H_0(x,\cdot) \diff y
 &=  \lambda^{-\frac{N-2s}{2}} \phi_0(x) \int_{B_d} U_\xl^{p-1} \diff y + \mathcal O(d^{-N+2s-1}  \lambda^{-\frac{N-2s}{2}} \int_{B_d} U_\xl^{p-1} |x-y| \diff y ) \\
 &\quad +  \lambda^{-\frac{N-2s}{2}} \int_{\Omega \setminus B_d} U_\xl^{p-1} H_0(x,y)\diff y .
\end{align*}
Since $H_0(x,y) \lesssim d^{-N+2s}$ by Lemma \ref{lemma greens fct}, the last term is
\[  \lambda^{-\frac{N-2s}{2}} \int_{\Omega \setminus B_d} U_\xl^{p-1} H_0(x,\cdot) \diff y \lesssim d^{-N+2s}  \lambda^{-\frac{N-2s}{2}} \int_{\R^N \setminus B_d} U_\xl^{p-1} \diff y \lesssim (d\lambda)^{-N} = o((d\lambda)^{-N+2s}). \]
Similarly, 
\begin{align*} \lambda^{-\frac{N-2s}{2}} \phi_0(x) \int_{B_d} U_\xl^{p-1} \diff y& = \lambda^{-\frac{N-2s}{2}} \phi_0(x) \|U_{0,1}\|_{p-1}^{p-1} + \mathcal O\left(\phi_0(x) \lambda^{-\frac{N-2s}{2}} \int_{\R^N \setminus B_d} U_\xl^{p-1} \diff y \right)  \\
&= \lambda^{-\frac{N-2s}{2}} \phi_0(x) \|U_{0,1}\|_{p-1}^{p-1}  + \mathcal O((d\lambda)^{-N}) \\
&=   \lambda^{-\frac{N-2s}{2}} \phi_0(x) \|U_{0,1}\|_{p-1}^{p-1}  + o((d\lambda)^{-N+2s}).
\end{align*}
Finally, 
\begin{align*}
    d^{-N+2s-1}  \lambda^{-\frac{N-2s}{2}} \int_{B_d} U_\xl^{p-1} |x-y| &\lesssim (d\lambda)^{-N+2s-1} \int_0^{d\lambda} \frac{r^N}{(1+r^2)^\frac{N+2s}{2}}\diff r = o((d\lambda)^{-N+2s})
\end{align*}
(where one needs to distinguish the cases where $1-2s$ is positive, negative or zero because the $\diff r$-integral is divergent if $1-2s \geq 0$).

Collecting all the previous estimates, we have proved
\begin{align}
\label{ds PU expansion}
\| \ds PU_\xl\|^2 &= c_{N,s} \|U_\xl\|_p^p + c_{N,s} \lambda^{-N+2s} \|U_\xl\|_{p-1}^{p-1} \phi_0(x) + o( (d\lambda)^{-N +2s}).
\end{align} 

The potential term splits as 
\[ \int_\Omega (a+\eps V) (PU_\xl+w)^2 \diff y = \int_\Omega (a+ \eps V) PU_\xl^2 \diff y + \int_\Omega a w^2 \diff y + \int_\Omega \left( (a+ \eps V) PU_\xl w + \eps V w^2 \right) \diff y \]
and we can estimate 
\[ \left | \int_\Omega (a+ \eps V) PU_\xl^2 \diff y \right| \lesssim \|U_\xl\|_2^2 \lesssim \lambda^{-N + 2s} \]
as well as 
\[ \int_\Omega \left( (a+ \eps V) PU_\xl w + \eps V w^2 \right) \diff y  \lesssim \|PU_\xl\|_{p'} \|w\|_p + \eps \|w\|^2 = \mathcal O\left( \lambda^{-\frac{N-2s}{2}} \|\ds w\|) + o(\|\ds w\|^2\right). \]
In summary we have, for the numerator of $\mathcal S_{a+ \eps V}[u_\eps]$,
\begin{align*}
& \qquad \alpha^{-2}  \left( \|\ds u\|^2 + \int_\Omega (a + \eps V) u^2 \diff y \right) \\
 &= c_{N,s} \|U_\xl\|_p^p + c_{N,s} \lambda^{-N+2s} \|U_\xl\|_{p-1}^{p-1} \phi_0(x) + \|\ds w\|^2 + \int_\Omega a w^2 \diff y  \\
 & \quad +\mathcal O(\lambda^{-N+2s})  +\mathcal O( \lambda^{-\frac{N-2s}{2}} \|\ds w\| + o( (d\lambda)^{-N +2s}) + o(\|\ds w\|^2) .
\end{align*}
\textbf{Step 2:} \textit{Expansion of the denominator.  }
By Taylor's formula,
\[ (PU_\xl+w)^p = PU_\xl^p + p PU_\xl^{p-1} w + \frac{p(p-1)}{2} PU_\xl^{p-2} w^2 + \mathcal O(PU_\xl^{p-3} |w|^3  + |w|^p). \]
Note that, strictly speaking, we use this formula if $p \geq 3$. If $2< p \leq 3$, the same is true without the remainder term $PU^{p-3} |w|^3$, which does not affect the rest of the proof.
To evaluate the main term, write $PU_\xl = U_\xl - \varphi_\xl$ with $\varphi_\xl := \lambda^{-1/2} H_0(x, \cdot) + f_\xl$, see Lemma \ref{lemma decomp PU}. Then
\begin{align*} \int_\Omega PU_\xl^p \diff y &= \int_\Omega U_\xl^p \diff y - p \int_\Omega U_\xl^{p-1} \varphi_\xl  \diff y + \mathcal O\left(\int_\Omega (U_\xl^{p-2} \varphi_\xl^2 + \varphi_\xl^p) \diff y \right) \\
& \|U_{0,1}\|_p^p - p \lambda^{-N+2s} \|U_{0,1}\|_{p-1}^{p-1} \phi_0(x) + o((d\lambda)^{-N+2s}) 
\end{align*}
where we used that by Lemmas \ref{lemma decomp PU} and \ref{lemma int Uq} $\int_\Omega U_\xl^{p-2} \varphi_\xl^2 \diff y \leq \|U_\xl\|_{p-2}^{p-2} \|\varphi_\xl\|^2_\infty \lesssim (d\lambda)^{-2N+4s} = o((d\lambda)^{-N+2s})$  and $\|\varphi_\xl\|_p^p \lesssim (d\lambda)^{-N} = o((d\lambda)^{-N+2s})$.

Next, the integral of the remainder term is controlled by 
\[ \int_\Omega (PU_\xl^{p-3} |w|^3  + |w|^p) \diff y \lesssim \|PU_\xl\|^{p-3}_p \|w\|_p^3 + \|w\|_p^p = o(\|\ds w\|^2). \]
The term linear in $w$ is 
\[  \int_\Omega PU_\xl^{p-1} w \diff y = \int_\Omega U_\xl^{p-1} w \diff y + \mathcal O \left(\int_\Omega (U_\xl^{p-2} \varphi |w| + \varphi^{p-1} |w|) \diff y \right). \]
Now by orthogonality of $w$, we have
\[ \int_\Omega U_\xl^{p-1} w \diff y = c_{N,s}^{-1} \int_\Omega (-\Delta)^s U_\xl w \diff y = \int_{\R^N} \ds U_\xl \ds w \diff y = 0. \]
Moreover, using $\|\varphi_\xl\|_p  \lesssim (d \lambda)^{-\frac{N-2s}{2}}$ by Lemma \ref{lemma decomp PU}, we get 
\[ \left|\int_\Omega \varphi_\xl^{p-1} w \diff y \right| \leq \|\varphi_\xl\|_p^{p-1} \|w\|_p \leq (d\lambda)^{-\frac{N+2s}{2}} \|\ds w\| = o((d\lambda)^{-N+2s}). \]
Using additionally that $\|\varphi_\xl\|_\infty \lesssim d^{-N+2s} \lambda^{-\frac{N-2s}{2}}$ by Lemma \ref{lemma decomp PU}, by the same computation as in \cite[Lemma A.1]{FrKoKo2020} we get $\|U_\xl^{p-2} \varphi_\xl\|_{\frac{p}{p-1}} \lesssim (d\lambda)^{-N+2s}$ and therefore 
\[ \int_\Omega U_\xl^{p-2} \varphi_\xl |w| \diff y \lesssim  (d\lambda)^{-N+2s} \|\ds w\|. \]
In summary we have, for the denominator of $\mathcal S_{a+ \eps V}[u_\eps]$, 
\begin{align*}
     \alpha^{-p} \int_\Omega u_\eps^p \diff y &=  \|U_{0,1}\|_p^p - p \lambda^{-N+2s} \|U_{0,1}\|_{p-1}^{p-1} \phi_0(x) + \frac{p(p-1)}{2} \int_\Omega U_\xl^{p-2} w^2 \diff y \\
     & + \mathcal O((d\lambda)^{-N+2s} \|\ds w\|) + o(\|\ds w \|^2) +  o((d\lambda)^{-N+2s}). 
\end{align*}

\textbf{Step 3:} \textit{Expansion of the quotient.  }
Using Taylor's formula, we find, for the denominator,
\begin{align*}
\alpha^{-2} \left( \int_\Omega u_\eps^p \diff y \right)^{-2/p} &= \|U_{0,1}\|_p^{-2} + 2 \|U_{0,1}\|_p^{-p-2} \|U_{0,1}\|_{p-1}^{p-1} \lambda^{-N+2s} \phi_0(x) \\
& \qquad - c_{N,s} (p-1) \|U_{0,1}\|_p^{-p-2} \int_\Omega U_\xl^{p-2} w^2 \diff y   \\
& \qquad + \mathcal O((d\lambda)^{-N+2s} \|\ds w\|) +o(\|\ds w \|^2) +  o((d\lambda)^{-N+2s}). 
\end{align*}
Multiplying this with the expansion for the denominator found above, we obtain 
\begin{align*}
\mathcal S_{a+ \eps V}[u_\eps] &= c_{N,s} \|U_{0,1}\|_p^{p-2} + \lambda^{-N+2s} c_{N,s} \|U_{0,1}\|_p^{-2} \|U_{0,1}\|_{p-1}^{p-1} \phi_0(x)  \\
& \qquad \|U_{0,1}\|_p^{-2} \left(\|\ds w\|^2 + \int_\Omega a w^2 \diff y -  c_{N,s} (p-1)  \int_\Omega U_\xl^{p-2} w^2 \diff y \right) \\
&\qquad +  \mathcal O((d\lambda)^{-N+2s} \|\ds w\|) +o(\|\ds w \|^2) +  o((d\lambda)^{-N+2s}). 
\end{align*}
Expressing the various constants using Lemma \ref{lemma constants}, we find 
\begin{align*}
c_{N,s} \|U_{0,1}\|_p^{p-2} &= S , \\
\|U_{0,1}\|_p^{-2} &= \left( \frac{S}{c_{N,s}} \right)^{- \frac{N-2s}{2s}}, \\
c_{N,s} \|U_{0,1}\|_p^{-2} \|U_{0,1}\|_{p-1}^{p-1} &= 2^{2s} \pi^{N/2} \frac{\Gamma(s)}{\Gamma(\frac{N-2s}{2})} \left( \frac{S}{c_{N,s}} \right)^{- \frac{N-2s}{2s}}.
\end{align*}
This yields the expansion claimed in the lemma. 
\end{proof}

\section{Proof of Theorem \ref{theorem critical}}
\label{ssec:criticality}

At this point, we have collected sufficiently precise information on the behavior of a general almost minimizing sequence to prove Theorem \ref{theorem critical}.

The main difficulty of the argument consists in constructing, for a critical potential $a$, a point $x_0 \in \Omega$ at which $\phi_a(x_0) = 0$. To do so, we carry out some additional analysis for a sequence $u_\eps$ which we assume to consist of true minimizers of $S(a - \eps)$, not only almost minimizers as in the rest of this paper. We make this additional assumption essentially for convenience and brevity of the argument, see the remark below Lemma \ref{lemma esposito remainder}. 

Indeed, since $a$ is critical, we have $S(a - \eps) < S$ for every $\eps > 0$. By the results of \cite{SeVa2015}, which adapts the classical lemma of Lieb contained in \cite{BrNi} to the fractional case, this strict inequality implies existence of a minimizer $u_\eps$ of $S(a - \eps)$. Normalizing $\int_\Omega u_\eps^\frac{2N}{N-2s} \diff y = A_{N,s}$ as in \eqref{almost_min}, $u_\eps$ satisfies the equation
\begin{align}
\label{u eps equation esposito}
    \Ds u_\eps + (a - \eps) u_\eps = \frac{S(a-\eps)}{A_{N,s}^{2s/N}} u_\eps^\frac{N+2s}{N-2s} \quad \text{on } \Omega, \qquad u \equiv 0 \quad \text{ on } \R^N \setminus \Omega.
\end{align} 

By using equation \eqref{u eps equation esposito}, we can conveniently extract the leading term of the remainder term $w_\eps$. We do this in the following lemma, which is the key step in the proof of Theorem \ref{theorem critical}. 

\begin{lemma}
    \label{lemma esposito remainder}
    Let $u_\eps$ be minimizers of $S(a - \eps)$ which satisfy \eqref{u eps equation esposito}. Then we have 
\begin{equation}
    \label{esposito expansion}
    S(a-\eps) = S + c_{N,s} A_{N,s}^{-2/p} \phi_a(x) \lambda^{-N+2s} + o(\lambda^{-N+2s}). 
\end{equation}
\end{lemma}

If $8s/3 < N$, Lemma \ref{lemma esposito remainder} is in fact implied by the more refined analysis carried out in Section \ref{sec:LB2} below, which does not use the equation \eqref{u eps equation esposito}. If $2s < N \leq 8s/3$, we speculate than one can prove Lemma \ref{lemma esposito remainder} for almost minimizers not satisfying \eqref{u eps equation esposito} by arguing like in \cite[Section 5]{FrKoKo2021}, but we do not pursue this explicitly here. 

\begin{proof}
[Proof of Lemma \ref{lemma esposito remainder}]
Clearly, the analysis carried out in Section \ref{sec:LB1} so far applies to the sequence $(u_\eps)$. Thus, up to passing to a subsequence, we may assume that $u_\eps = \alpha_\eps (PU_{x_\eps, \lambda_\eps} + w_\eps)$ with $\alpha_\eps \to 1$, $x_\eps \to x_0 \in \Omega$ and $\|\ds w_\eps\|_2 \lesssim \lambda^{-\ns}$ as $\eps \to 0$. 

Thus the sequence $\tilde{w}_\eps := \lambda_\eps^{\ns} w_\eps$ is bounded in $\ths$ and converges weakly in $\ths$, up to a subsequence, to some $\tilde{w}_0 \in \ths$. Inserting the expansion $u = \alpha( PU_\xl + w)$ in \eqref{u eps equation esposito}, the equation fulfilled by $\tilde{w}_\eps$ reads
\begin{equation}
    \label{tilde w equation}
    \Ds \tilde{w}_\eps + (a - \eps) \tilde{w}_\eps = -(a - \eps) PU_\xl \lambda^\ns +  \lambda^{-2s} \frac{S(a-\eps)}{A_{N,s}^{2s/N}} \left( PU_\xl \lambda^\ns + \tilde{w}_\eps \right)^\frac{N+2s}{N-2s}. 
\end{equation} 
By Lemma \ref{lemma decomp PU}, we can write
\[ PU_\xl \lambda^{\ns} = G_0(x, \cdot) - \lambda^{N-2s} h(\lambda(x - \cdot)) - \lambda^{N-2s} f_\xl \]
with $h$ as in Lemma \ref{lemma h}. By the bounds on $h$ and $f_\xl$ from Lemmas \ref{lemma decomp PU} and \ref{lemma h}, this yields
\begin{equation}
    \label{PU to G0}
    PU_\xl \lambda^{\ns} \to G_0 (x_0, \cdot) \, \text{ uniformly on compacts of } \Omega \setminus \{x_0\} \, \text{ and in } L^q(\Omega), \,   q < \tfrac{N}{N-2s}. 
\end{equation}
Letting $\eps \to 0$ in \eqref{tilde w equation}, we obtain 
\begin{equation}
    \label{tilde w 0 weak eq}
    \int_\Omega \ds \tilde{w}_0 \ds \varphi \diff y + \int_\Omega a \tilde{w}_0 \varphi \diff y = - \int_\Omega a G_0(x, \cdot) \varphi \diff y 
\end{equation}
for every $\varphi \in C^\infty_c(\Omega \setminus \{x_0\})$. Now it is straightforward to show that $C^\infty_c(\Omega \setminus \{x_0\})$ is dense in $\ths$, by using a cutoff function argument together with identity \eqref{frac prod rule}. Thus by approximation, \eqref{tilde w 0 weak eq} even holds for every $\varphi \in \ths$. In other words, $\tilde{w}_0$ weakly solves the equation 
\[ (-\Delta)^s \tilde{w}_0 + a \tilde{w}_0 = - a G_0(x_0, \cdot) \quad \text{ on } \Omega, \qquad \tilde{w}_0 \equiv 0 \quad \text{ on } \R^N \setminus \Omega.  \]
By uniqueness of solutions, we conclude $\tilde{w}_0 = H_0(x_0, \cdot) - H_a(x_0, \cdot)$. 

We will now use this information to prove the desired expansion \eqref{esposito expansion} of the energy $S(a-\eps) = \mathcal S_{a -\eps} [PU_\xl + w]$. Indeed, using the already established bound $\|\ds w\| \lesssim \lambda^{-\ns}$, the numerator is
\begin{align}
\label{esposito num 1}
    \|\ds PU_\xl\|^2 + \int_\Omega a (PU_\xl^2 + 2 PU_\xl w) \diff y + \|\ds w\|^2 + \int_\Omega a w^2 \diff y + o(\lambda^{-N+2s}).  
\end{align}
By integrating the equation for $w$ against $w$ and recalling $\frac{S(a-\eps)}{A_{N,s}^{2s/N}} = c_{N,s} + o(1)$ , we easily find the asymptotic identity (compare \cite[eq. (8)]{Esposito2004} for $s = 1$)
\[ \|\ds w\|^2 + \int_\Omega a w^2 \diff y = c_{N,s} (p-1) \int_\Omega U_\xl^{p-2} w^2 \diff y - \int_\Omega a PU_\xl w \diff y. \]
Inserting this in \eqref{esposito num 1}, together with the expansion of $\|\ds PU_\xl\|^2$ given in \eqref{ds PU expansion}, the numerator of $\mathcal S_{a -\eps} [PU_\xl + w]$ becomes
\begin{equation}
    \label{esposito num 2}
   c_{N,s} A_{N,s}  - c_{N,s} a_{N,s} \phi_0(x) \lambda^{-N+2s} + \int_\Omega a (PU_\xl^2 + PU_\xl w) \diff y + c_{N,s} (p-1) \int_\Omega U_\xl^{p-2} w^2 \diff y  + o(\lambda^{-N+2s}).
\end{equation}
The numerator of $\mathcal S_{a -\eps} [PU_\xl + w]$, by the computations in the proof of Lemma \ref{lemma expansion PU+w}, is given by 
\begin{equation}
    \label{esposito den 1}
    \left( \int_\Omega (PU_\xl + w)^p \diff y \right)^{-2/p} = A_{N,s}^{-\frac{2}{p}} - A_{N,s}^{-\frac{2}{p} - 1} \left( - 2 \phi_0(x) \lambda^{-N+2s} + (p-1) \int_\Omega U_\xl^{p-2} w^2 \diff y \right) . 
\end{equation}
Multiplying out \eqref{esposito num 2} and \eqref{esposito den 1}, the terms in $\int_\Omega U^{p-2} w^2 \diff y$ cancel precisely and we obtain 
\begin{align}
    \label{esposito quot 1}
    S(a-\eps) &= S + c_{N,s} A_{N,s}^{-2/p} a_{N,s} \phi_0(x) \lambda^{-N+2s}  \\
    & \qquad + A_{N,s}^{-2/p} \lambda^{-N+2s} \left( \int_\Omega a \left( (\lambda^\ns PU_\xl)^2 + \lambda^\ns PU_\xl \tilde{w}\right) \diff y \right)  + o(\lambda^{-N+2s}).  \nonumber
\end{align}
Now we are ready to return to our findings about $\tilde{w}_0$. Indeed, by \eqref{PU to G0}, and observing that $G_0(x, \cdot)$ is an admissible test function in \eqref{tilde w 0 weak eq}, we get
\begin{align}
     &\qquad \int_\Omega a \left( (\lambda^\ns PU_\xl)^2 + \lambda^\ns PU_\xl \tilde{w}\right) \diff y = \int_\Omega a \left( G_0(x, \cdot)^2 + G_0(x, \cdot) \tilde{w}_0 \right) \diff y + o(1) \\
     &=- \int_\Omega \ds \tilde{w}_0 \ds G_0(x, \cdot) \diff y +o(1)= - \tilde{w}_0(x) = \gamma_{N,s} (\phi_a(x) - \phi_0(x)) + o(1).
\end{align}
By inserting this into \eqref{esposito quot 1} and observing that $\gamma_{N,s} = c_{N,s} a_{N,s}$ by the numerical values given in Lemma \ref{lemma constants}, the proof is complete. 
\end{proof}

Now we have all the ingredients to give a quick proof of our first main result. 

\begin{proof}
[Proof of Theorem \ref{theorem critical}]
As explained after the statement of the theorem, it only remains to prove the implication $(ii) \Rightarrow (i)$. Suppose thus $S(a) < S$ and let $c > 0$ be the smallest number such that $\bar{a}:= a +c$ satisfies $S(\bar{a}) = S$. For $\eps > 0$, let $u_\eps$ be the sequence of minimizers $S(\bar{a} - \eps)$, normalized to satisfy \eqref{u eps equation esposito}. By Lemma \ref{lemma esposito remainder}, we have 
\[ S > S(\bar{a} - \eps) = S - c_{N,s} A_{N,s}^{-2/p} \phi_a(x) \lambda^{-N+2s} + o(\lambda^{-N+2s}). \]
Letting $\eps \to 0$, this shows $\phi_{\bar{a}}(x_0) \leq 0$. By the resolvent identity, we have for every $x \in \Omega$
\[ \phi_{\tilde{a}}(x) = \phi_a(x) +  \int_{\Omega} (\tilde{a}-a)(z) G_a(x,z) G_{a+c}(z,x) \diff z < \phi_a(x), \]
and hence $\phi_a(x_0)$ is strictly monotone in $a$. Thus $\phi_a(x_0) < \phi_{a+c}(x_0) = 0$, and the proof is complete. 
\end{proof}

\section{Proof of the lower bound II: a refined expansion}
\label{sec:LB2}

This section is the most technical of the paper. It is devoted to extracting the leading term of the remainder $w$ and to obtaining sufficiently good bounds on the new error term. In Section \ref{ssec:refinedex} we will need to work under the additional assumption $8s/3 < N$ in order to obtain the required precision. 

Concretely, we write 
\[ w = \lambda^{-\frac{N-2s}{2}} (H_0(x, \cdot) - H_a(x, \cdot)) + q \]
and decompose the remainder further into a tangential and an orthogonal part
\[ q = t +r, \qquad t \in T_\xl, \quad r \in T_\xl^\bot. \]
(We keep omitting the subscript $\eps$.) A refined expansion of $\mathcal S_{a + \eps V}[u_\eps]$ then yields an error term in $r$ which can be controlled using the coercivity inequality of Proposition \ref{proposition coerc with a}. The refined expansion is derived in Section \ref{ssec:refinedex} below. 

On the other hand, since $t$ is an element of the $(N+2)$-dimensional space $T_\xl$, it can be bounded by essentially explicit computations. This is achieved in Section \ref{ssec:s}.

\begin{remark}
\label{remark section 5 s=1}
The present Section \ref{sec:LB2} thus constitutes the analogon of \cite[Section 6]{FrKoKo2021}, where the same analysis is carried out for the case $s=1$ and $N=3$. We emphasize that, despite these similarities, our approach is conceptually somewhat simpler than that of \cite{FrKoKo2021}. Indeed, the argument in \cite{FrKoKo2021} relies on an intermediate step involving a spectral cutoff construction, through which the apriori bound $\|\nabla q\| =o(\lambda^{1/2}) = o(\lambda^{-\ns})$ is obtained. 

On the contrary, we are able to conduct the following analysis with only the weaker bound $\|\nabla q\| = \mathcal O(\lambda^{-\ns})$ at hand (which follows from Proposition \ref{prop w and d}). This comes at the price of some additional explicit error terms in $r$, which can however be conveniently absorbed (see Lemmas \ref{lemma expansion refined denominator} and \ref{lemma main results coercivity}). Since $N > 8s/3$ is fulfilled when $N=3$, $s=1$, this simplified proof of course also works in the particular situation of \cite{FrKoKo2021}. 
\end{remark}

\subsection{A precise description of $t$}
\label{ssec:s}

For $\lambda$ large enough, the functions $PU_\xl$, $\pl PU_\xl$ and $\pxi PU_\xl$, $i = 1,...,N$ are linearly independent. There are therefore uniquely determined coefficients $\beta, \gamma, \delta_i$, $i = 1,...,N$, such that 
\begin{equation}
\label{expansion s} 
t = \beta \lambda^{-N+2s} PU_\xl + \gamma \lambda^{-N+2s+1}  \pl PU_\xl + \sum_{i = 1}^N \delta_i \lambda^{-N+2s-2} \pxi PU_\xl. 
\end{equation}
Here the choice of the different powers of $\lambda$ multiplying the coefficients is justified by the following result. 

\begin{lemma}
    \label{lemma beta gamma delta}
    As $\eps \to 0$, we have $\beta, \gamma, \delta_i = \mathcal O(1)$. 
\end{lemma}

As a corollary, we obtain estimates on $t$ in various norms. 

 \begin{lemma}
 \label{lemma bounds on s}
As $\eps \to 0$, 
\[ \|\ds t \|_2 \lesssim \lambda^{-N+2s} \quad \text{ and } \quad  \|t\|_\frac{2N}{N+2s} \lesssim \|t\|_2 \lesssim \lambda^{-\frac{3N -6s}{2}}. \]
 \end{lemma}
 
 \begin{proof}
 Recall that $PU_\xl = U_\xl - \lambda^{-\ns} H_0(x,\cdot) + f_\xl$. Then all bounds follow in a straightforward way from \eqref{expansion s} together with Lemma \ref{lemma beta gamma delta} and the standard bounds from Lemmas \ref{lemma greens fct}, \ref{lemma decomp PU}, \ref{lemma int Uq} and  \ref{lemma Hs norms PU}. 
 \end{proof}

\begin{proof}
[Proof of Lemma \ref{lemma beta gamma delta}]
\textbf{Step 1.} We introduce the normalized basis functions
\begin{equation}
\label{tilde varphi}
    \tilde{\varphi}_1 := \frac{PU_\xl}{\|\ds PU_\xl\|}, \quad \tilde{\varphi}_2 :=  \frac{\pl PU_\xl}{\|\ds \pl PU_\xl\|}, \quad \tilde{\varphi}_j :=  \frac{\partial_{x_{j-2}} PU_\xl}{\|\ds \partial_{x_{j-2}}  PU_\xl\|},
\end{equation}
and prove that 
\begin{equation}
    \label{tilde scalar product}
    a_j := \int_{\R^N} \ds \tilde{\varphi}_j \ds t \diff y = \begin{cases}
    \mathcal O(\lambda^{-N+2s}), & j = 1,2, \\
    \mathcal O(\lambda^{-N+2s-1}), & j = 3,..., N+2.
    \end{cases}
\end{equation}
Since $\lambda^{-\frac{N-2s}{2}} (H_0(x,\cdot) - H_a(x,\cdot)) + t + r = w \in T_\xl^\bot$, and $r \in T_\xl^\bot$, we have
\begin{align*} a_j & = \lambda^{-\frac{N-2s}{2}} \int_{\R^N} \ds \tilde{\varphi}_j  \ds \left(H_a(x,\cdot) - H_0(x,\cdot)\right) \diff y \\
&= \lambda^{-\ns} \int_{\R^N} \Ds \tilde{\varphi_j} \left(H_a(x,\cdot) - H_0(x,\cdot)\right) \diff y .
\end{align*}
Thus, 
\begin{align*}
    a_1 &=  \lambda^{-\ns} \|\ds PU_\xl\|^{-1} c_{N,s} \int_\Omega U_\xl^\frac{N+2s}{N-2s} \left(H_a(x,\cdot) - H_0(x,\cdot)\right) \diff y  \\
    &\lesssim \lambda^{-N+2s},
\end{align*}
where we used that by Lemma \ref{lemma Hs norms PU}, $\|\ds PU_\xl\|^{-1} \lesssim 1$. The bound for $a_2$ follows similarly. To obtain the claimed improved bound for $a_j$, $j = 3,...,N+2$, we write 
\begin{align*}
    a_{i+2} &\lesssim  \lambda^{-\ns}  \|\ds \pxi PU_\xl\|^{-1}  \int_\Omega U_\xl^\frac{4s}{N-2s} \pxi U \left(H_a(x,\cdot) - H_0(x,\cdot)\right) \diff y   \\
    & \lambda^{-\ns} \|\ds \pxi PU_\xl\|^{-1} \mathcal O \left(  \int_{\R^N \setminus B_d}  U_\xl^\frac{4s}{N-2s} |\pxi U_\xl| \diff y +  \int_\Omega  U_\xl^\frac{4s}{N-2s} |\pxi U_\xl| |x-y| \diff y  \right) \\
    &\lesssim \lambda^{-N+2s-1}.
\end{align*}
Here we wrote $H_a(x,y) - H_0(x,y) = \phi_a(x) -\phi_0(x) + \mathcal O(|x-y|)$ and used that by oddness of $\pxi U$, 
\[ (\phi_a(x) -\phi_0(x)) \int_{B_d} U_\xl^\frac{4s}{N-2s} \pxi U_\xl = 0. \]
This concludes the proof of \eqref{tilde scalar product}. 

\textbf{Step 2.} We write 
\[ t = \sum_{j = 1}^{N+2} b_j \tilde{\varphi}_j, \]
with 
\begin{align*} b_1 &= \beta \lambda^{-N+2s} \|\ds PU_\xl\|, \quad b_2 = \gamma \|\ds \pl PU_\xl\|, \\
 b_j &= \delta_j \lambda^{-N+2s-2} \|\ds \partial_{x_{j-2}} PU_\xl\|, \qquad j= 3,...,N+2.
 \end{align*}
Our goal is to show that 
\begin{equation}
    \label{bj and aj}
    b_j = a_j + \mathcal O(\lambda^{-N+2s}) \sup_k a_k, \qquad j = 1,..,N+2. 
\end{equation}
From \eqref{bj and aj} we conclude by the estimates on the $a_j$ from \eqref{tilde scalar product} and Lemma \ref{lemma Hs norms PU}. 

To prove \eqref{bj and aj}, we define the Gram matrix $G$ by 
\[ G_{j,k} := (\ds \tilde{\varphi}_j, \ds \tilde{\varphi}_k). \]
By Lemma \ref{lemma Hs norms PU} and the definition of the $\tilde{\varphi}_j$ it is easily checked that  
\[ G_{j,k} = \delta_{j,k} + \mathcal O(\lambda^{-N+2s}), \qquad j,k = 1,..., N+2. \]
Thus for sufficiently large $\lambda$, $G$ is invertible with 
\begin{equation}
    \label{G-1bound}
    (G^{-1})_{j,k} = \delta_{j,k} +  \mathcal O(\lambda^{-N+2s}). 
\end{equation} 
By definition of $G$, 
\begin{equation}
    \label{gram onb}
\psi_j := \sum_{k=1}^{N+2}  (G^{-1/2})_{j,k} \tilde{\varphi}_k 
\end{equation}
is an orthonormal basis of $T_\xl$. We can therefore write
\begin{align*}
    t &= \sum_j \left(\ds \psi_j, \ds t \right) \psi_j = \sum_j \sum_k (G^{-1/2})_{j,k} \left(\ds \tilde{\varphi}_k, \ds t \right) \phi_j \\
    &= \sum_j \sum_k (G^{-1/2})_{j,k} a_k \sum_l (G^{-1/2})_{j,l} \tilde{\varphi}_l \\
    &= \sum_l \left( \sum_k \left( \sum_j (G^{-1/2})_{l,j} (G^{-1/2})_{j,k} \right) a_k \right) \tilde{\varphi}_l \\
    &= \sum_l \left( \sum_k  (G^{-1})_{l,k} a_k \right) \tilde{\varphi}_l
\end{align*}
Thus $b_l =  \sum_k  (G^{-1})_{l,k} a_k$ and \eqref{bj and aj} follows from \eqref{G-1bound}. 
 \end{proof}
 
 \begin{remark}
By treating the terms in the above proof more carefully, it can be shown in fact that $\lambda^{N - 2s} \beta$, $\lambda^{N- 2s-1}\gamma$ and $\lambda^{N-2s+2} \delta_i$ have a limit as $\lambda \to \infty$. Indeed, for instance, the leading orders of the expressions $\int_\Omega U_\xl^\frac{N+2s}{N-2s} (H_a(x, \cdot) - H_0(x, \cdot)) \diff y$ and $\|\ds PU_\xl\|$ going into the leading behavior of $\beta$ can be explicitly evaluated, see Lemma \ref{lemma Up-k Hk} and the proof of Lemma \ref{lemma Hs norms PU} respectively. We do not need the behavior of the coefficients $\beta, \gamma, \delta_i$ to that precision in the following, so we do not state them explicitly. 
 \end{remark}
 
\subsection{The new expansion of $\mathcal S_{a+\eps V} [u]$}
\label{ssec:refinedex}

Our goal is now to expand the value of the energy functional $\mathcal S_{a + \eps V} [u_\eps]$ with respect to the refined decomposition introduced above, namely
\[ u = \alpha(\psi_\xl + q) = \alpha\left( PU_\xl + \left(\lambda^{-\ns} H_a(x, \dot) - H_0(x, \cdot)\right) + t + r \right). \]

In all that follows, we work under the important assumption that 
\begin{equation}
    \label{-3N+6s < -2s}
    -3N + 6s < -2s, \quad \text{ i.e. } \tfrac{8}{3} s < N
\end{equation}
so that $\lambda^{-3N + 6s} = o(\lambda^{-2s})$. Assumption \eqref{-3N+6s < -2s} has the consequence that, using the available bounds on $t$ and $r$, we can expand the energy $\mathcal S_{a + \eps V}[u]$ up to $o(\lambda^{-2s})$ errors in a way that does not depend on $t$. This is the content of the next lemma. 

\begin{lemma}
    \label{lemma expansion refined quotient}
    As $\eps \to 0$, we have
 \begin{align*}
     \mathcal S_{a + \eps V}[u_\eps] &= \mathcal S_{a + \eps V} [\psi_\xl]  + D_0^{-2/p} \left( \mathcal E_0[r] - \frac{2 N_0}{p D_0} I[r] + o(\|\ds r\|^2) \right)  \\
     &\quad + o(\lambda^{-2s}) + o(\eps \lambda^{-N+ 2s}) + o(\phi_a(x) \lambda^{-N+2s}). 
 \end{align*}
 Here, 
 \begin{equation}
     \label{definition N0 D0}
     N_0 := \|\ds \psi_\xl\|_2^2 + \int_\Omega (a+ \eps V) \psi_\xl^2 \diff y , \qquad D_0 := \int_\Omega \psi_\xl^p \diff y, 
 \end{equation}
 and $I[r]$ is as defined in \eqref{definition Ir} below.
\end{lemma}

We emphasize that the contribution of $t$ enters only into the remainders  $o(\lambda^{-2s}) + o(\eps \lambda^{-N+ 2s})+ o(\phi_a(x) \lambda^{-N+2s})$. This is remarkable because $t$ enters to orders $\lambda^{-N+2s} >> \lambda^{-2s}$ and $\lambda^{-2N+4s} >> \lambda^{-2s}$ (if $N< 3s$) into both the numerator and the denominator of $ \mathcal S_{a + \eps V}[u_\eps]$, see Lemmas \ref{lemma expansion refined numerator} and \ref{lemma expansion refined denominator} below. When calculating the quotients, these contributions cancel precisely, as we verify in Lemma \ref{lemma N_1 D_1} below. Heuristically, such a phenomenon is to be expected because (up to projection onto $\ths$ and perturbation by $a + \eps V$) by definition $t$ represents the directions along which the quotient functional is invariant. As already pointed out in the introduction, we suspect, but cannot prove, that in the absence of assumption \eqref{-3N+6s < -2s} the contributions of $t$ to the higher order coefficients $\lambda^{-kN + 2ks}$ for $3 \leq k \leq \tfrac{2N}{N-2s}$ would continue to cancel. 

We prove Lemma \ref{lemma expansion refined quotient} by separately expanding the numerator and the denominator of $\mathcal S_{a + \eps V}[u_\eps]$. We abbreviate
\[ \mathcal E_\eps[u]:= \|\ds u\|^2 + \int_\Omega (a + \eps V)u^2 \diff y \]
and write $\mathcal E_\eps[u,v]$ for the associated bilinear form. 

\begin{lemma}[Expanding the numerator]
    \label{lemma expansion refined numerator}
    As $\eps \to 0$, 
    \begin{align*}
    |\alpha|^{-2} \mathcal E_\eps [u_\eps] &= \mathcal E_\eps [\psi_\xl] + \left( 2 \mathcal E_0[\psi_\xl, t]  + \| \ds t\|^2  \right) + \mathcal E_0[r]  +  o(\lambda^{-2s}) + o(\eps \lambda^{-N+ 2s}).
    \end{align*}
\end{lemma}

\begin{proof}
We write $\alpha^{-1} u_\eps = \psi_\xl + t + r$ and therefore 
\begin{equation}
 \label{E eps u}
 \mathcal E_\eps[u_\eps] =  \mathcal E_\eps[\psi_\xl] +  2 \mathcal E_\eps[\psi_\xl, t+r]  +  \mathcal E_\eps[t+r] .
\end{equation} 
The third term on the right side is 
\[  \mathcal E_\eps[t+r] =  \mathcal E_0[t] +  2\mathcal E_0[t,r] + \mathcal E_0[r] + \eps \int_\Omega V (t+r)^2 \diff y . \]
Now $\int_{\R^N} \ds t \ds r \diff y = 0$ by orthogonality and therefore, by Lemma \ref{lemma bounds on s},
\[ \mathcal E_0[t,r] = \int_\Omega a t r \diff y  = \mathcal O(\|\ds r\| \|t\|_\frac{2N}{N+2s}) = \mathcal O(\lambda^\frac{-3N+6s}{2} \|\ds r\|) = o(\lambda^{-2s}) + o(\|\ds r\|^2), \]
where the last equality is a consequence of assumption \eqref{-3N+6s < -2s} and Young's inequality. 
Finally, again by Lemma \ref{lemma bounds on s},
\[ \eps \int_\Omega V (t+r)^2 \diff y =  \mathcal O( \eps (\|t\|^2 + \|\ds r\|^2)) =   o(\eps \lambda^{-N+2s}) + o( \|\ds r\|^2). \] 

The second term on the right side of \eqref{E eps u} is
\[  2 \mathcal E_\eps[\psi_\xl, t+r] = 2  \mathcal E_0[\psi_\xl, t] + 2  \mathcal E_0[\psi_\xl, r ] + 2 \eps  \int_\Omega V \psi_\xl (t + r) \diff y . \]
To start with, using Lemma \ref{lemma bounds on s},
\begin{align*} & \qquad \eps  \int_\Omega V \psi_\xl (t + r) \diff y = \mathcal O( \eps ( \|t\|_\frac{2N}{N+2s} + \|\psi\|_\frac{2N}{N+2s} \|\ds r\|)) \\
&= \mathcal O( \eps \lambda^{-\ns} \|\ds r\|) + o(\eps \lambda^{-N+2s}) = o(\eps \lambda^{-N+2s}) + o(\|\ds r\|^2)),
\end{align*}
again by Young's inequality. Moreover, using that $r \in T_\xl^\bot$, that $\Ds H_a(x, \cdot) = a G_a(x, \cdot)$ and $\Ds H_0(x,\cdot) = 0$, and integrating by parts,
\begin{align*}
    \mathcal E_0[\psi_\xl, r ] &= \lambda^{-\ns} \int_\Omega \ds (H_0 - H_a)(x, \cdot) \ds r \diff y + \int_\Omega a \psi_\xl r \diff y \\
    &= \int_\Omega a (- \lambda^{-\ns} G_a(x, \cdot) +  \psi_\xl) r \diff y.
    \end{align*}
Since we can write $\lambda^{-\ns} G_a(x,y) - \psi_\xl(y) = \lambda^\ns h(\lambda(x-y)) + f_\xl$ with $h$ as in Lemma \ref{lemma h} and $f_\xl$ as in Lemma \ref{lemma decomp PU}, we get 
  \begin{align*}  
   \mathcal E_0[\psi_\xl, r ]  & \lesssim \left(\lambda^{-\ns} \|h(\lambda \cdot)\|_\frac{2N}{N+2s} + \|f_\xl\|_\infty\right) \|\ds r\| \lesssim \lambda^{-2s} \|\ds r\| = o(\lambda^{-2s})
\end{align*}
by the bounds in those lemmas. Finally, 
\[ \mathcal E_0[ t] = \|\ds t\|^2 + \int_\Omega a t^2 \diff y \]
and $\int_\Omega a t^2 \diff y \lesssim \|t\|_2^2 \lesssim \lambda^{-3N+6s} = o(\lambda^{-2s})$ by Lemma \ref{lemma bounds on s} and Assumption \eqref{-3N+6s < -2s}. 
\end{proof}

\begin{lemma}[Expanding the denominator]
    \label{lemma expansion refined denominator}
    As $\eps \to 0$, 
    \begin{align*}
        |\alpha|^{-p} \int_\Omega u_\eps^p \diff y &= \int_\Omega \psi_\xl^p \diff y + \left( p \int_\Omega \psi_\xl^{p-1} t \diff y + \frac{p(p-1)}{2} \int_\Omega \psi_\xl^{p-2} t^2 \diff y \right) + \frac{p(p-1)}{2} \int_\Omega \psi_\xl^{p-2} r^2  \diff y \\
        & \quad +  \mathcal O\left(  \lambda^{-\ns}  \int_\Omega U^{p-2} |H_a| |r| \diff y  \right) + o(\|\ds r\|^2)  + o(\lambda^{-2s}).
    \end{align*}
\end{lemma}

\begin{proof}
Write $\alpha^{-1} u_\eps = \psi_\xl +t+ r$. We expand
\begin{align*} \int_\Omega (\psi_\xl+ t +r)^p \diff y &= \int_\Omega (\psi_\xl+r)^p \diff y+ p \int_\Omega (\psi_\xl + r)^{p-1} t \diff y + \frac{p(p-1)}{2} \int_\Omega (\psi_\xl + r)^{p-2} t^2 \diff y \\
& \qquad + \mathcal O\left( \|\psi_\xl + r\|_p^{p-3} \|\ds t\|^3 + \|\ds t\|^p \right).
\end{align*}
By Lemma \ref{lemma bounds on s} together with assumption \eqref{-3N+6s < -2s}, the last term is $o(\lambda^{-2s})$. The third term is, by Lemma \ref{lemma bounds on s},
\[ \int_\Omega (\psi_\xl + r)^{p-2} t^2 \diff y = \int_\Omega \psi_\xl^{p-2} t^2 \diff y+ \mathcal O\left(\lambda^{-2N+4s} \|\ds r\| \right). \]
The second term is
\[ \int_\Omega (\psi_\xl + r)^{p-1} t \diff y = \int_\Omega \psi_\xl^{p-1} t \diff y+ (p-1) \int_\Omega  \psi_\xl^{p-2} r t \diff y+ o( \|\ds r\|^2 ). \]
The remaining term $\int_\Omega \psi_\xl^{p-2} r t\diff y$ needs to be expanded more carefully. Using $\psi_\xl = U_\xl - \lambda^{-\ns} H_a(x, \cdot) - f_\xl$ with $\| \lambda^{-\ns} H_a(x, \cdot) + f_\xl\|_\infty \lesssim \lambda^{-\ns}$, we write 
\begin{align*}  \int_\Omega \psi_\xl^{p-2} r t \diff y&= \int_\Omega U_\xl^{p-2} r t \diff y+ \mathcal O \left( \lambda^{-\ns} \| \ds r\| \|\ds t\|   \right)
\end{align*}
and using assumption \eqref{-3N+6s < -2s}, the remainder is bounded by
\[ \lambda^{-\ns} \| \ds r\| \|\ds t\|  \lesssim \lambda^{\frac{-3N+6s}{2}} \|\ds r\| = o(\lambda^{-2s}) + o(\|\ds r\|^2). \]
Now using orthogonality of $r$ and the expansion \eqref{expansion s} of $s$, by some standard calculations, whose details we omit, one obtains 
\begin{align*}
    \int_\Omega U_\xl^{p-2} r t \diff y &= \mathcal O \left( \lambda^{\frac{-3N+6s}{2}}\right \| \ds r\| =  o(\lambda^{-2s}) + o(\|\ds r\|^2),
\end{align*} 
where we used again assumption \eqref{-3N+6s < -2s} for the last equality. 

It remains only to treat the $t$-independent term $\int_\Omega (\psi_\xl + r)^p\diff y$. We find
\begin{align*} \int_\Omega (\psi_\xl + r)^p \diff y&= \int_\Omega \psi_\xl^p \diff y + p \int_\Omega \psi_\xl^{p-1} r \diff y + \frac{p(p-1)}{2} \int_\Omega \psi_\xl^{p-2} r^2 \diff y + o( \|\ds r\|^2). 
\end{align*}
Using orthogonality of $r$, we get that $\int_\Omega U_\xl^{p-1} r \diff y = 0$ and hence
\[ \int_\Omega \psi_\xl^{p-1} r \diff y= \mathcal O\left( \lambda^{-\ns}  \int_\Omega U_\xl^{p-2} |H_a(x, \cdot)| |r|  \diff y + \lambda^{-\frac{N+2s}{2}} \|\ds r\| \right) \]
and $\lambda^{-\frac{N+2s}{2}} \|\ds r\| \lesssim \lambda^{-N} = o(\lambda^{-2s})$. Finally, we have 
\[ \int_\Omega \psi_\xl^{p-2} r^2 \diff y = \int_\Omega U_\xl^{p-2} r^2 \diff y + o(\|\ds r\|^2). \]
Collecting all the estimates gives the claim of the lemma. 
\end{proof}

We can now prove the claimed expansion of the energy functional.

\begin{proof}
[Proof of Lemma \ref{lemma expansion refined quotient}]
We write the expansions of the numerator and the denominator as 
\[ \mathcal E_\eps [u_\eps] = N_0 + N_1 + \mathcal E_0(r) + o(\lambda^{-2s} + (\eps + \phi_a(x)) \lambda^{-N+2s}, \]
where 
\[ N_0 = \mathcal E_\eps [\psi], \quad N_1 := 2 \mathcal E_0[\psi_\xl, t]  + \| \ds t\|^2 , \]
and 
\[ \int_\Omega u_\eps^p = D_0 + D_1 + \mathcal I[r] + o(\lambda^{-2s}), \]
where 
\[ D_0 = \int_\Omega \psi^p , \qquad D_1 := p \int_\Omega \psi_\xl^{p-1} t + \frac{p(p-1)}{2} \int_\Omega \psi_\xl^{p-2} t^2, \]
and 
\begin{equation}
    \label{definition Ir}
\mathcal I[r] := \frac{p(p-1)}{2} \int_\Omega \psi_\xl^{p-2} r^2 + \mathcal O(\lambda^{-\ns}  \int_\Omega U^{p-2} |H_a| |r| \diff y ). 
\end{equation}
Taylor expanding up to and including second order, we find 
\begin{align*}
    \left( \int_\Omega u_\eps^p \right)^{-2/p} &= D_0^{-2/p} \left( 1 - \frac{2}{p} \frac{D_1 + I[r]}{D_0} +  \frac{p+2}{p^2} \frac{(D_1 + I[r])^2}{D_0^2} \right) + o(\lambda^{-2s}). 
\end{align*}
We now observe $\mathcal I[r] \lesssim \|\ds r\|^2 + o(\phi_a(x) \lambda^{-N+2s} + \lambda^{-2s})$ since
\[ \lambda^{-\ns} \int_\Omega U^{p-2} |H_a| |r| \lesssim \int_\Omega U^{p-2} r^2 + \lambda^{-N+2s} \int_\Omega U^{p-2} H_a^2 \lesssim \|\ds r\|^2 + o(\lambda^{-N+2s} \phi_a(x)) \]
Hence we can simplify the expression of the denominator to 
\[    \left( \int_\Omega u_\eps^p \right)^{-2/p} = D_0^{-2/p} \left( 1 - \frac{2}{p} \frac{D_1 + I[r]}{D_0} +  \frac{p+2}{p^2} \frac{D_1^2}{D_0^2} \right)  + o(\|\ds r\|^2)  + o(\phi_a(x) \lambda^{-N+2s})+ o(\lambda^{-2s}).  \]
Multiplying this with the expansion of the numerator from above, we find 
\begin{align*}
    \mathcal S_{a + \eps V}[u_\eps] &= D_0^{-2/p} N_0 + D_0^{-2/p} \left( N_1 - \frac{2}{p} \frac{N_0}{D_0} D_1 - \frac{2}{p} \frac{N_1 D_1}{D_0} + \frac{p+2}{p^2} \frac{D_1^2 N_0}{D_0^2} \right) \\
    & \qquad + D_0^{-2/p} \left( \mathcal E_0[r] - \frac{2 N_0}{p D_0} I[r] + o(\|\ds r\|^2) \right) + o(\lambda^{-2s}) + o(\phi_a(x)\lambda^{-N+2s}).  
    \end{align*}
We show in Lemma \ref{lemma N_1 D_1} below that the bracket involving the terms $N_1$ and $D_1$ involving $s$ vanishes up to order $o(\lambda^{-2s})$, due to cancellations. Noting that $D_0^{-2/p} N_0$ is nothing but $\mathcal S_{a+ \eps V}[\psi]$, the expansion claimed in Lemma \ref{lemma expansion refined quotient} follows. 
\end{proof}

\begin{lemma}
    \label{lemma N_1 D_1}
    Assume \eqref{-3N+6s < -2s} and let $N_0$, $N_1$, $D_0$, $D_1$ be defined as in the proof of Lemma \ref{lemma expansion refined quotient}. Then
\begin{align*}
    N_1 &= 2 \beta c_{N,s}A_{N,s} \lambda^{-N+2s} \\
    & \quad + c_{N,s} \lambda^{-2N + 4s} \left( \beta^2 A_{N,s} + \gamma^2 (p-1) B_{N,s}  - 2  a_{N,s} \phi_0(x) (\beta - \ns \gamma)  \right) + o(\lambda^{-2s})
    \end{align*} 
    and 
    \begin{align*}
        D_1 &= \lambda^{-N+2s} p \beta A_{N,s} + \lambda^{-2N + 4s} \left( \frac{p(p-1)}{2} (\beta^2 A_{N,s} + \gamma^2 B_{N,s}) -p (\beta - \frac{N-2s}{2} \gamma) a_{N,s} \phi_0(x) \right) \\
        &\qquad + o(\phi_a(x) \lambda^{-N+2s}) + o(\lambda^{-2s}). 
    \end{align*}
    where we abbreviated $B_{N,s} := \int_{\R^N} U_{0,1}^{p-2} |\pl U_{0,1}|^2 \diff y$. 
    
    In particular,
    \[ N_1 - \frac{2}{p} \frac{N_0}{D_0} D_1 - \frac{2}{p} \frac{N_1 D_1}{D_0} + \frac{p+2}{p^2} \frac{D_1^2 N_0}{D_0^2} = o(\lambda^{-2s}) + o(\phi_a(x) \lambda^{-N+2s}) . \]
\end{lemma}

\begin{proof}
We start with expanding $N_1 = 2\mathcal E_0[\psi, t] + \|\ds t\|^2 $. From Lemma \ref{lemma Hs norms PU} and the expansion \eqref{expansion s} for $t$, one easily sees that 
\begin{align*}
     \|\ds t\|^2 &= \beta^2 \lambda^{-2N+4s} \| \ds PU_\xl\|^2 + \gamma \lambda^{-2N+4s+2} \|\ds \pl PU_\xl\|^2 \\
     &= \beta^2 c_{N,s} A_{N,s} \lambda^{-2N + 4s} + \gamma^2 (p-1) c_{N,s}B_{N,s} +  o(\lambda^{-2s}),
\end{align*} 
where we also used assumption \eqref{-3N+6s < -2s}. Next, recalling $(\Ds + a) \psi_\xl = c_{N,s} U_\xl^{p-1} - a (\lambda^{\ns} h(\lambda(x - \cdot) + f_\xl)$ with $h$ as in Lemma \ref{lemma h}, we easily obtain
\begin{align*}
    2 \mathcal E_0[\psi_\xl, t] &= 2 c_{N,s} \int_\Omega U_\xl^{p-1} t \diff y + o(\lambda^{-2s}) =  2 \beta c_{N,s} A_{N,s} \lambda^{-N+2s}  \\
    & \qquad - 2 c_{N,s} a_{N,s} \phi_0(x) \lambda^{-2N+4s}(\beta + \ns \gamma) + o(\lambda^{-2s}).
\end{align*}
(Observe that the leading order term with $\gamma$ vanishes because $\int_{\R^N} U_{0,1}^{p-1} \pl U_{0,1} = 0$.)
This proves the claimed expansion for $N_1$. For $D_1$, we have 
\begin{align*}
    \int_\Omega \psi_\xl^{p-1} t \diff y &= \lambda^{-N+2s} \beta \int_\Omega \psi_\xl^{p-1} PU_\xl \diff y + \gamma \lambda^{-N+2s+1} \int_\Omega \psi_\xl^{p-1} \pl PU_\xl \diff y + o(\lambda^{-2s}).
\end{align*}
Writing out $\psi_\xl = U_\xl - \lambda^{-\ns} H_0(x, \cdot) - f_\xl$ and $PU_\xl = U_\xl - \lambda^{-\ns} H_0(x, \cdot) - f$, by the usual bounds together with assumption \eqref{-3N+6s < -2s} we get 
\begin{align*}
    \lambda^{-N+2s} \beta \int_\Omega \psi^{p-1} PU_\xl \diff y &= \lambda^{-N+2s} \beta A_{N,s} - \lambda^{-2N+4s} \beta  a_{N,s} \phi_0(x) + o(\lambda^{-N+2s} \phi_a(x))  + o(\lambda^{-2s}).
\end{align*}
Similarly, 
\begin{align*}
    \gamma \lambda^{-N+2s+1} \int_\Omega \psi^{p-1} \pl PU_\xl \diff y &= \gamma \frac{N-2s}{2} \lambda^{-2N+4s} a_{N,s} \phi_0(x) + o(\lambda^{-N+2s} \phi_a(x)) + o(\lambda^{-2s}). 
\end{align*}
(Observe that the leading order term with $\gamma$ vanishes because $\int_{\R^N} U_{0,1}^{p-1} \pl U_{0,1} = 0$.) Finally, 
\begin{align*}
    \int_\Omega \psi_\xl^{p-2} t^2 \diff y &= \lambda^{-2N + 4s} \left( \beta^2  A_{N,s} + \gamma^2 B_{N,s} \right) + o(\lambda^{-2s}).
\end{align*}
Putting together the above, we end up with the claimed expansion for $D_1$. 

The last assertion of the lemma follows from the expansions of $N_0$, $D_0$, $N_1$ and $D_1$ by an explicit calculation whose details we omit.
\end{proof}

Based on the refined expansion of $\mathcal S_{a+ \eps V}[u_\eps]$ obtained in Lemma \ref{lemma expansion refined quotient}, we are now in a position to give the proofs of our main results. 

We first use the coercivity inequality from Proposition \ref{proposition coercivity} to control the terms involving $r$ that appear in Lemma \ref{lemma expansion refined quotient}.

\begin{lemma}[Coercivity result]
    \label{lemma main results coercivity}
    There is $\rho > 0$ such that, as $\eps \to 0$, 
    \[ \mathcal E_0[r] - \frac{2 N_0}{p D_0} I[r] \geq \rho \| \ds r\|^2. \] 
\end{lemma}

\begin{proof}
Recalling the definition \eqref{definition Ir} of $\mathcal I[r]$ and observing that $N_0/D_0 = c_{N,s}$, we find by Proposition \ref{proposition coerc with a} that
\begin{align*}
 \mathcal E_0[r] - \frac{2 N_0}{p D_0} I[r] &= \|\ds r\|^2 + \int_\Omega a r^2 \diff y - c_{N,s} (p-1) \int_\Omega U_\xl^{p-2} r^2 \diff y \\
&\qquad + \mathcal O\left( \lambda^{-\ns} \int_\Omega U_\xl^{p-2} |H_a(x \cdot)| |r| \diff y \right)  +o(\|\ds r\|^2) \\
&\geq \rho \|\ds r\|^2 +  \mathcal O\left( \lambda^{-\ns} \int_\Omega U_\xl^{p-2} |H_a(x \cdot)| |r| \diff y \right)
\end{align*} 
for some $\rho > 0$. The remaining error term can be bounded as follows.
\begin{align*}
    \lambda^{-\ns} \int_\Omega U_\xl^{p-2} |H_a(x, \cdot)| |r| \diff y &\leq \delta' \int_\Omega U_\xl^{p-2} r^2 \diff y + C  \lambda^{-N+2s} \int_\Omega U_\xl^{p-2} H_a(x,\cdot)^2 \diff y   \\
    &\leq \delta \|\ds r\|^2 + \mathcal O(\lambda^{-2N+4s} \phi_a(x)^2) + o(\lambda^{-2s}) \\
    &\leq \delta \|\ds r\|^2  + o(\lambda^{-N+2s} \phi_a(x) + \lambda^{-2s})
\end{align*}
where we used Lemma \ref{lemma Up-k Hk}. By choosing $\delta > 0$ small enough, we obtain the conclusion. 
\end{proof}

\section{Proof of the main results}
\label{sec:proof}

Combining Lemma \ref{lemma main results coercivity} with Lemma \ref{lemma expansion refined quotient} gives a lower bound on $\mathcal S_{a + \eps V}[u_\eps]$. Using the almost-minimizing assumption \eqref{almost_min} and the expansion from Theorem \ref{theorem upper bound}, this lower bound can be stated as follows: 
 	\begin{align}    \label{S[u] lower bound}
0 &\geq (1 + o(1)) (S - S(a + \eps V))   + \mathcal R \\
& \qquad + A_{N,s}^{-\frac{N-2s}{N}} \left(  (Q_V(x) +o(1))  \eps \lambda^{-N+2s} - (a(x)+o(1))(\alpha_{N,s} + c_{N,s} d_{N,s} b_{N,s} ) \lambda^{-2s}    \right), \nonumber
	\end{align}
where 
\[ \mathcal R := A_{N,s}^\frac{N-2s}{N} \left(a_{N,s} c_{N,s} (1 + o(1)) \phi_a(x) \lambda^{-N + 2s} + \mathcal T_2(\phi_a(x), \lambda)\right) + \rho \| \ds r\|_2^2,
\]
for some $\rho > 0$, and $\mathcal T_2(\phi_a(x), \lambda)$ as in \eqref{T i definition}. 

Recall that $\phi_a \geq 0$ by Corollary \ref{corollary phia geq 0, a leq 0} and that $\phi_a(x)$ is bounded because $x_0 \in \Omega$.  Since $\mathcal T_2$ is a sum of higher powers $(\phi_a(x) \lambda^{-N+2s})^k$ with $k \geq 2$, we have $\mathcal R \geq 0$ for $\eps$ small enough.  

\begin{lemma}
\label{lemma x0 in Na}
As $\eps \to 0$, $\phi_a(x) = o(1)$. In other words, $x_0 \in \mathcal N_a$. 
\end{lemma}

\begin{proof}
Since $S - S(a + \eps V) \geq 0$ and $\|\ds r\|_2^2 \geq 0$, the bound \eqref{S[u] lower bound} gives
\[ \phi_a(x) \lesssim \eps + \lambda^{N-4s} + \lambda^{N-2s} \mathcal T_2(\phi_a(x), \lambda). \]
Since $\phi_a(x)$ is uniformly bounded, we can bound $\mathcal T_2(\phi_a(x), \lambda) \lesssim \lambda^{-2N + 4s}$, which concludes.
\end{proof}

\begin{lemma}
\label{lemma x0 in NaV}
    If $\mathcal N_a(V) \neq \emptyset$, then $x_0 \in \mathcal N_a(V)$.
\end{lemma}

In the proof of this lemma, we need the assumption \eqref{assumption a}, i.e. that $a(x) < 0$ on $\mathcal N_a$.

\begin{proof}
By Lemma \ref{lemma x0 in Na} we only need to prove that $Q_V(x_0) < 0$. 

Inserting the upper bound from Corollary \ref{corollary upper bound} on $S - S(a+ \eps V)$ into \eqref{S[u] lower bound}, and using $\mathcal R \geq 0$, we obtain that 
\[   (Q_V(x) +o(1)) \eps \lambda^{-N+2s} \leq - C_1 \eps^\frac{2s}{4s-N} + C_2 \lambda^{-2s}. \]
Here the numbers $C_1$ and $C_2$ are given by 
\[ C_1 := (1+o(1)) \sigma_{N,s} \sup_{x \in \mathcal N_a(V).} \frac{|Q_V(x)|^\frac{2s}{4s-N}}{|a(x)|^\frac{N-2s}{4s-N}}, \qquad C_2:= -a(x)+o(1)  \]
Using Lemma \ref{lemma x0 in Na} and the assumption $a < 0$ on $\mathcal N_a$, we have that $C_2$ is strictly positive and remains bounded away from zero by assumption. Since $\mathcal N_a(V)$ is not empty, the same is clearly true for $C_1$. Thus by Young's inequality
\[ - C_1 \eps^\frac{2s}{4s-N} + C_2 \lambda^{-2s} \leq - c \eps \lambda^{-N+2s} \]
for some $c > 0$. This implies $Q_V(x_0) < -c < 0$ as desired. 
\end{proof}

Now we are ready to prove our main results Theorems \ref{thm:13}, \ref{thm:14} and \ref{thm:17}. 

\begin{proof}
[Proof of Theorem \ref{thm:13}]
By using $\mathcal R \geq 0$ and minimizing the last term over $\lambda$, like in the proof of Corollary \ref{corollary upper bound}, the bound \eqref{S[u] lower bound} implies 
\begin{align*} 
(1+ o(1)) (S(a + \eps V) - S) &\geq - \sigma_{N,s}\frac{|Q_V(x_0)|^\frac{2s}{4s-N}}{|a(x_0)|^\frac{N-2s}{4s-N}} \eps^\frac{2s}{4s-N} + o(\eps^\frac{2s}{4s-N} ) \\
&\geq - \sigma_{N,s} \sup_{x \in \mathcal N_a(V)} \frac{|Q_V(x)|^\frac{2s}{4s-N}}{|a(x)|^\frac{N-2s}{4s-N}} \eps^\frac{2s}{4s-N} + o(\eps^\frac{2s}{4s-N} ),
\end{align*}
where the last inequality follows from Lemma \ref{lemma x0 in NaV}. This is equivalent to 
\[ S(a + \eps V) \geq S  - \sigma_{N,s} \sup_{x \in \mathcal N_a(V)} \frac{|Q_V(x)|^\frac{2s}{4s-N}}{|a(x)|^\frac{N-2s}{4s-N}} \eps^\frac{2s}{4s-N} + o(\eps^\frac{2s}{4s-N} ). \]
Since the matching upper bound has already been proved in Corollary \ref{corollary upper bound}, the proof of the theorem is complete. 
\end{proof}

\begin{proof}
[Proof of Theorem \ref{thm:14}]
Since $x_0 \in \mathcal N_a$ by Lemma \ref{lemma x0 in Na}, by assumption we have $Q_V(x_0) \geq 0$ and $a(x_0) < 0$. Together with $\mathcal R \geq 0$, the bound \eqref{S[u] lower bound} then implies 
\[ 0 \geq (1 + o(1)) (S - S(a + \eps V)) + c \lambda^{-2s} + o(\eps \lambda^{-N+2s}) \]
for some $c > 0$. Since $o(\eps \lambda^{-N+2s}) \geq - \frac{c}{2} \lambda^{-2s} + o(\eps^\frac{2s}{4s - N})$ by Young, this implies $S(a+\eps V) \geq S + o(\eps^\frac{2s}{N-4s}$. Since the inequality 
\begin{equation}
\label{S(a + epsV) leq S}
    S(a +\eps V) \leq S
\end{equation} always holds (e.g. by Theorem \ref{theorem upper bound}), we obtain $S(a+\eps V) \geq S + o(\eps^\frac{2s}{N-4s})$ as desired. 

Now assume that additionally $Q_V(x_0) > 0$. With $\mathcal R \geq 0$, \eqref{S[u] lower bound} implies, for $\eps > 0$ small enough and some $C_1, C_2 > 0$
\[ S(a + \eps V) - S \geq C_1 \eps \lambda^{-N+2s} + C_2 \lambda^{-2s} > 0, \]
which contradicts \eqref{S(a + epsV) leq S}. Thus assumption \eqref{strict ineq S(a+eps V) < S}, under which we have worked so far, cannot be satisfied, and we must have $S(a + \eps_0 V) = S$ for some $\eps_0 > 0$. Since $S(a + \eps V)$ is concave in $\eps$ (being the infimum of functions linear in $\eps$) and since $S(a) = S$, we must have $S(a + \eps V) = S$ for all $\eps \in [0, \eps_0]$. 
\end{proof}

\begin{proof}
[Proof of Theorem \ref{thm:17}]
We may first observe that the upper and lower bounds on $S(a + \eps V)$ already discussed in the proof of Theorem \ref{thm:13} imply 
\begin{equation}
    \label{x_0 proof 17}
    \frac{|Q_V(x_0)|^\frac{2s}{4s-N}}{|a(x_0)|^\frac{N-2s}{4s-N}} = \sup_{x \in \mathcal N_a} \frac{|Q_V(x)|^\frac{2s}{4s-N}}{|a(x)|^\frac{N-2s}{4s-N}}. 
\end{equation}
Now, by using additionally Lemma \ref{lem-taylor}, the estimate \eqref{S[u] lower bound} becomes 
\begin{equation}
\label{bound proof 17}
    (1 + o(1)) \left((S(a +\eps V) -S\right) \geq   - \sigma_{N,s} \frac{|Q_V(x_0)|^\frac{2s}{4s-N}}{|a(x_0)|^\frac{N-2s}{4s-N}} \epsilon^\frac{2s}{4s-N} + \mathcal R' + o(\epsilon^\frac{2s}{4s-N}), 
\end{equation} 
where 
\[ \mathcal R' = \begin{cases}
\mathcal R + c_0 \epsilon^\frac{2s-2}{4s-N}  \left(  \lambda^{-1} - \lambda_0(\epsilon)^{-1}\right)^2 & \text{ if } \quad \left(\frac{A_\epsilon }{B_\epsilon}\right)^{\frac{1}{4s-N}} \epsilon^{-\frac{1}{4s-N}} \lambda^{-1} \leq 2 \left(\frac{2s}{N-2s}\right)^\frac{1}{N-4s} , \\
\mathcal R + c_0 \epsilon ^\frac{2s}{4s-N}  & \text{ if } \quad  \left(\frac{A_\epsilon }{B_\epsilon}\right)^{\frac{1}{4s-N}} \epsilon^{-\frac{1}{4s-N}} \lambda^{-1} > 2 \left(\frac{2s}{N-2s}\right)^\frac{1}{N-4s},
\end{cases} 
\]  
in the notation of Lemma \ref{lem-taylor}, for $A_\eps := A_{N,s}^{-\frac{N-2s}{N}} (\alpha_{N,s} + c_{N,s} d_{N,s} b_{N,s} )(|a(x_0)|+o(1))$ and $B_\eps :=  A_{N,s}^{-\frac{N-2s}{N}} (|Q_V(x_0)| + o(1)) $, with $\lambda_0(\eps)$ given by \eqref{f eps minimum}. Now applying in \eqref{bound proof 17} the upper bound on $S(a +\eps V)$ from Corollary \ref{corollary upper bound} yields 
\[ \mathcal R'  = o(\eps^\frac{2s}{4s-N} ). \]
The terms that make up $\mathcal R'$ being separately nonnegative, this implies $\mathcal R = o(\eps^\frac{2s}{4s-N})$ and $(\lambda^{-1} - \lambda_0(\eps)^{-1})^2 = o(\eps^\frac{2}{4s-N})$, that is,
\begin{align} \lambda &= \left(\frac{2s A_\epsilon}{(N-2s) B_\epsilon } \right)^\frac{1}{4s-N}  \epsilon^{-\frac{1}{4s-N}} + o(\epsilon^{-\frac{1}{4s-N}}) \nonumber \\
&= \left(\frac{2s (\alpha_{N,s} + c_{N,s} d_{N,s} b_{N,s} ) |a(x_0)| }{(N-2s) |Q_V(x_0)| } \right)^\frac{1}{4s-N}  \epsilon^{-\frac{1}{4s-N}} + o(\epsilon^{-\frac{1}{4s-N}}) \label{bound lambda proof 17}
\end{align}
and
\begin{equation}
    \label{bound r proof 17}
    \|\ds r\|_2 = o(\eps^\frac{s}{4s-N} ). 
\end{equation}
Inserting the asymptotics of $\lambda$ back into  $\mathcal R = o(\eps^\frac{2s}{4s-N})$ now gives 
\begin{equation}
    \label{bound phi proof 17}
    \phi_a(x) = o(\eps). 
\end{equation} 
It remains to derive the claimed expansion for $\alpha$. From Lemma \ref{lemma expansion refined denominator}, we deduce
\[ |\alpha|^{-\frac{2N}{N-2s}} \int_\Omega u_\eps^\frac{2N}{N-2s} \diff y = \int_\Omega \psi_\xl^\frac{2N}{N-2s} \diff y + \mathcal O\left(\|\ds s\|_2 + \|\ds r\|^2_2 + \lambda^{-N+2s} \right). \]
Using the bound $\|\ds s\|_2 \lesssim \lambda^{-N+2s}$, together with \eqref{bound lambda proof 17}, \eqref{bound r proof 17} and the expansion of $\int_\Omega \psi_\xl^\frac{2N}{N-2s}$ from Theorem \ref{theorem upper bound}, we obtain
\[ |\alpha|^{-p} = 1 + \mathcal O(\lambda^{-N+2s}) = 1 + \mathcal O(\eps^\frac{N-2s}{4s-N}). \]
This completes the proof of Theorem \ref{thm:17}. 
\end{proof}

\appendix
\vspace{3mm}

\section{Green's function}
\label{sec:app:greens}

\subsection{The Green's function $G_0$ and the projections $PU_\xl$}
\label{subsection G_0 and PU}

We begin by studying the case $a = 0$. The next lemma collect some important estimates on the regular part $H_0(\cdot, \cdot)$ of the Green's function and the Robin function $\phi_0(x) = H_0(x,x)$, which will turn out very important for our analysis. Similar estimates for $s=1$ have been derived in \cite[Section 2 and Appendix A]{Rey1990}. 

We denote in the following $d(x) := \text{dist}(x, \partial \Omega)$. 

\begin{lemma}
	\label{lemma greens fct}
	Let $x \in \Omega$ and $N > 2s$. Then $y \mapsto H_0(x,y)$ is continuous on $\Omega$ and we have, for all $y \in \Omega$,
	\begin{align*}
	0 \leq H_0(x,y) &\lesssim  d(x)^{2s-N} , \\
	| \nabla_y H_0(x, y)| &\lesssim d(x)^{2s - N- 1} . 
	\end{align*}
	Moreover, the Robin function $\phi_0$ satisfies the two-sided bound 
	\begin{equation}
	\label{phi_0_bound}
	d(x)^{2s-N} \lesssim \phi_0(x) \lesssim  d(x)^{2s-N}.
	\end{equation}
\end{lemma}

\begin{proof}
	$H_0(x, \cdot)$ satisfies 
	\begin{align*}
	(-\Delta)^s H_0(x, \cdot) &= 0 \qquad \text{ on } \Omega,   \\
	H_0(x, \cdot) &= \frac{1}{|x - \cdot|^{N-2s}} \qquad \text{ on } \R^N \setminus \Omega. 
	\end{align*}
	Thus we can write 
	\[ H_0(x, y) = \int_{\R^N \setminus \Omega}  \frac{1}{|x - z|^{N-2s}} \diff P_\Omega^y(z) \]
	where $P_\Omega^y$ denotes harmonic measure for $\Ds$, see \cite[Theorem 7.2]{Kwasnicki2019}. Since $P_\Omega^y$ is a probability measure, this implies 
	\[0 \leq  H_0(x,y)  \lesssim d(x)^{-N+2s}. \]

	Similarly, since 
	\begin{align*}
	(-\Delta)^s \nabla H_0(x, \cdot) &= \nabla (-\Delta)^s  H_0(x, \cdot) = 0 && \text{ on } \Omega,   \\
	\nabla H_0(x, \cdot) &= \nabla \frac{1}{|x - \cdot|^{N-2s}} && \text{ on } \R^N \setminus \Omega,
	\end{align*}
	we have 
	\[ |\nabla_y H_0(x, y)| = \left|\int_{\R^N \setminus \Omega} \left( \nabla_z \frac{1}{|x - z|^{N-2s}} \right) \diff P_\Omega^y(z)\right| \lesssim d(x)^{-N+2s-1}. \]
The lower bound $d(x)^{2s-N} \lesssim \phi_0(x)$ is proved in \cite[Lemma 7.6]{DaLoSi2017}. 
\end{proof}

The following important lemma shows the relation between the regular part $H_0(x, \cdot)$ and the projections $PU_\xl$ introduced in \eqref{eq:projintro}. For the classical case $s = 1$, this is \cite[Proposition 1]{Rey1990}. For fractional $s \in (0,1)$, a slightly weaker version relying on the extension formulation of $(-\Delta)^s$ appears in \cite[Lemma C.1]{ChKiLe2014}.

\begin{lemma}
	\label{lemma decomp PU}
	Let $x \in \Omega$ and $N > 2s$. 
	\begin{enumerate}
		\item[(i)] We have
		\begin{equation}
		\label{PU_leq_U}
		0 \leq PU_\xl \leq U_\xl
			\end{equation}
	and the function $\varphi_\xl := U_\xl - PU_\xl$ satisfies the estimates
		\begin{equation}
		\label{varphi_Linfty_bound}
		\|\varphi_\xl\|_{L^\infty(\R^N)} \lesssim  d(x)^{-N + 2s} \lambda^{\frac{-N+2s}{2}}
		\end{equation}
		and
		\begin{equation}
		\label{varphi_Lp_bound}
		\|\varphi_\xl\|_{L^p(\R^N)} \lesssim \left( d(x) \lambda \right)^{-\frac{N-2s}{2}} . 
		\end{equation} 
		\item [(ii)]
		Moreover, the expansion
		\begin{equation}
		\label{exp_PU}
		PU_\xl = U_\xl - \lambda^\frac{N-2s}{2} H_0(x, y) + f_\xl,
		\end{equation}
holds with
		\begin{align*}
		\|f_\xl\|_{L^\infty(\Omega)} & \lesssim  d(x)^{-N-2 + 2s}  \lambda^{-\frac{N + 4-2s}{2}}
		\end{align*}
	\end{enumerate}
\end{lemma}

\begin{proof}

\textit{Claim (i).  }		
	Our proof follows mostly \cite[Appendix A]{Rey1990}. Since
		\begin{align*}
		(-\Delta)^s PU_\xl &\geq 0 \qquad \text{ on } \Omega,   \\
		PU_\xl &\equiv 0 \qquad \text{ on } \R^N \setminus \Omega. 
		\end{align*}
		the maximum principle (see e.g. \cite[Proposition 2.17]{Silvestre2007}) implies that $PU_\xl \geq 0$. 
		Similarly, $\varphi_\xl = U_\xl - PU_\xl$ satisfies
		\begin{align}
		(-\Delta)^s \varphi_\xl &= 0 && \text{ on } \Omega,  \label{varphi eq}  \\
		\varphi_\xl &= U_\xl \geq 0 && \text{ on } \R^N \setminus \Omega. \nonumber
		\end{align}
	and thus $\varphi_\xl \geq 0$ by the maximum principle. This completes the proof of \eqref{PU_leq_U}. 
		
By \eqref{varphi eq}, we can moreover write	
\[	\varphi_\xl(y) = \int_{\R^N \setminus \Omega} U_\xl(z) \diff P_\Omega^y(z), \quad y \in \Omega. \]
Thus $\|\varphi_\xl\|_{L^\infty(\R^N)} = \|U_\xl\|_{L^\infty(\R^N \setminus \Omega)}  \lesssim \lambda^{\frac{-N+2s}{2}} d(x)^{-N + 2s}$. 
		
Next, let us prove the $L^p$ estimate on $\varphi_\xl$. Since $\varphi_\xl \in H^s(\R^N)$, by the Sobolev inequality we have
		\begin{align}
		\| \varphi_\xl\|^2_{L^p(\R^N)}  &\lesssim \|(-\Delta)^{s/2} \varphi_\xl \|^2_2 \nonumber \\
		&= \|(-\Delta)^{s/2} U_\xl \|^2_2 + \|(-\Delta)^{s/2} PU_\xl \|^2_2 - 2 \int_{\R^N} (-\Delta)^s U_\xl PU_\xl \diff y. \label{est_varphi_xl}
		\end{align} 
		The second summand in \eqref{est_varphi_xl} can be written as 
		\begin{align*}
		\|(-\Delta)^{s/2} PU_\xl \|^2_2 &= c_{N,s} \int_\Omega PU_\xl U_\xl^{p-1} \diff y = \|(-\Delta)^{s/2} U_\xl \|^2_2 - c_{N,s} \int_\Omega \varphi_\xl U_\xl^{p-1} \diff y \\
		&= \|(-\Delta)^{s/2} U_\xl \|^2_2 + \mathcal O\left(\|\varphi_\xl\|_\infty \int_\Omega U_\xl^{p-1} \diff y \right) =  \|(-\Delta)^{s/2} U_\xl \|^2_2  + \mathcal O \left((d(x) \lambda)^{-N+2s} \right)
		\end{align*}
		by \eqref{varphi_Linfty_bound}. 
		
		Similarly, the third summand in \eqref{est_varphi_xl} is 
		\begin{align*}
		- 2 \int_{\R^N} (-\Delta)^s U_\xl PU_\xl \diff y &= -2 c_{N,s} \int_\Omega U_\xl^p \diff y + 2 c_{N,s} \int_\Omega U_\xl^{p-1} \varphi_\xl \diff y \\
		&= -2 \|(-\Delta)^{s/2} U_\xl \|^2_2 + \mathcal O\left((d(x) \lambda)^{-N+2s}\right),
		\end{align*}
		where we also used the bound 
		\[ \int_{\R^N \setminus \Omega} U_\xl^p \diff y \lesssim (d(x) \lambda)^{-N}. \]
		Collecting these estimates and returning to \eqref{est_varphi_xl}, we obtain 
		\[ \| \varphi_\xl\|^2_{L^p(\R^N)} \lesssim (d(x) \lambda)^{-N+2s}. \]
		This concludes the proof of \eqref{varphi_Lp_bound}.

\textit{Claim (ii).  }		
The function $f_\xl := \varphi_\xl - \lambda^{-\frac{N-2s}{2}} H_0(x, \cdot)$ satisfies 
		\begin{align*}
		(-\Delta)^s f_\xl &= 0 && \text{ on } \Omega,   \\
		f_\xl &= U_\xl - \frac{ \lambda^{-\frac{N-2s}{2}}}{|x - \cdot|^{N-2s}} && \text{ on } \R^N \setminus \Omega. 
		\end{align*}
		As in the proof of Lemma \ref{lemma greens fct}, we have
		\begin{equation*}
		f_\xl(y) = \int_{\R^N \setminus \Omega} \left( U_\xl(y) - \frac{ \lambda^{-\frac{N-2s}{2}}}{|x - z|^{N-2s}} \right) \diff P_\Omega^y(z), 
		\end{equation*}
		and hence, since $P_\Omega^x$ is a probability measure, we have 
		\begin{align*}
		\|f_\xl\|_{L^\infty(\Omega)} &\leq \left\|U_\xl(y) - \frac{ \lambda^{-\frac{N-2s}{2}}}{|x - y|^{N-2s}}\right\|_{L^\infty(\R^N \setminus \Omega)} 
		\\ &= \mathcal O \left(\lambda^{-\frac{N + 4-2s}{2}} \textrm{dist}(x, \partial \Omega)^{-N-2 + 2s}\right) 
		\end{align*}
		by Lemma \ref{lemma h} below. 
\end{proof}

\subsection{Expanding the regular part $H_b(x,y)$ near the diagonal}
\label{subsection Hb near diagonal}

We now turn to the Green's function $G_b$, for a general potential $b \in C^1(\Omega) \cap C(\overline{\Omega})$ such that $\Ds + b$ is coercive. By noting the potential $b$ rather than $a$, we emphasize the fact that criticality of $b$ is not needed for the following expansions. 
Moreover, in contrast to the previous subsection, we specialize to the condition $2s < N < 4s$ again, which plays a role in the proof of Lemmas \ref{lemma Ha expansion} and \ref{lemma Up-k Hk} below. 

\begin{lemma}
	\label{lemma Ha expansion}
	Let $x \in \Omega$ and $2s < N < 4s$.
	\begin{enumerate}
		\item[(i)] If $4s - N < 1$, then as $y \to x$
		\[ H_b(x,y) = \phi_b(x) - d_{N,s} b(x) |x-y|^{4s - N} + o( |x-y|^{4s - N}). \]
		\item[(ii)] If $4s - N \geq 1$, then there is $\xi_x \in \R^N$ such that
		\[ H_b(x,y) = \phi_b(x) + \xi_x \cdot (y-x) - d_{N,s} b(x) |x-y|^{4s - N} + o( |x-y|^{4s - N}). \]
	\end{enumerate}
	Here the constant $d_{N,s} > 0$ is given by \eqref{definition dNs}. 
	The asymptotics are uniform for $x$ in compact subsets of $\Omega$. 
\end{lemma}

\begin{proof}
	Fix $x \in \Omega$ and let 
	\[ \psi_x(y) := H_b(x,y) - \phi_b(x) + d_{N,s} b(x) |x-y|^{4s - N}, \]
	with $d_{N,s}$ as in \eqref{definition dNs}. We use the facts that, in the distributional sense, 
	\[ (-\Delta)^s_y H_b(x,y) = b(y) G_b(x,y) = \frac{b(y)}{|x-y|^{N-2s}} - b(y) H_b(x,y) \]
	and, by Lemma \ref{lemma constants},
	\[ \Ds |x|^{4s - N} = -d_{N,s}^{-1} |x|^{2s - N}. \]
Thus $\psi_x$ solves, in the distributional sense, the equation
	\begin{equation}
	\label{eq psi}
	(-\Delta)^s_y  \psi_x(y) = F_x(y),
	\end{equation}
	with 
	\[ F_x(y) = \frac{b(y) - b(x)}{|x-y|^{N-2s}} - b(y) H_b(x,y).  \]
	Since $b \in C^1(\Omega)$, we have
	\[ \frac{|b(x) - b(y)|}{|x-y|^{N-2s}} \lesssim |x-y|^{-N + 2s + 1}. \]
	We will deduce the assertion of the lemma in each case from elliptic estimates on the equation \eqref{eq psi} and appropriate bounds on $F_x$.
	
	\emph{Case $- N + 2s + 1 < 0$.  }	Since the second summand $b(y) H_b(x,y)$ is in $L^\infty$, we have $F_x \in L^p(\Omega)$ for every $p < \frac{N}{N-2s-1}$. For the following, fix some $p \in (\frac{N}{2s}, \frac{N}{N-2s-1})$. (The assumption $N < 4s$ guarantees that this interval is not empty.) 
	
	Define $\tilde{\psi_x} := (-\Delta)^{-s} F_x$, where $(-\Delta)^{-s}$ is convolution with the Riesz potential. Then by \cite[Theorem 1.6.(iii)]{RoSe2014-2} we have $[\tilde{\psi_x}]_{C^\alpha(\R^N)} \lesssim \|F_x\|_{L^p(\R^N)}$, where $\alpha = 2s - \frac{N}{p}$. Moreover $(-\Delta)^s (\psi_x - \tilde{\psi_x}) = 0$ on $\Omega$. Since $s$-harmonic functions are smooth (see e.g. \cite[Section 2]{Abatangelo2015}), we conclude that $\psi_x \in C^{2s - \frac{N}{p}}(B_{d/2}(x))$. 
	
	Since $\psi_x(x) = 0$, we conclude that as $y \to x$, 
	\begin{equation}
	\label{psi x o} \psi_x(y) = \mathcal O(|x-y|^{2s - \frac{N}{p}}). 
	\end{equation}
	If we choose $p \in (\frac{N}{N-2s}, \frac{N}{N-2s-1})$, then $2s - \frac{N}{p} > 4s - N$. (As a consequence of $N < 4s$, we have the inclusion $(\frac{N}{N-2s}, \frac{N}{N-2s-1}) \subset (\frac{N}{2s}, \frac{N}{N-2s-1})$. Together with the definition of $\psi_x$, \eqref{psi x o} then implies
	\[ H_b(x,y) = \phi_b(x) - d_{N,s} b(x) |x-y|^{4s - N}
	+ o(|x-y|^{4s-N}), \]
	which is the assertion of the lemma. 
		
	\emph{Case $- N + 2s + 1 \geq 0$. }
	In this case $F_x \in L^\infty(\Omega)$. More precisely, we have 
	\[ F_x \in \
	\begin{cases}
	L^\infty(\Omega) & \text{ if } N = 2s +1, \\
	C^{0,  - N + 2s +1}(\Omega)  & \text{ if } 0 < - N + 2s +1.
	\end{cases}
	\]
	Notice that we always have $ - N + 2s +1 < 1$, since $N > 2s$.
	As above, define $\tilde{\psi_x} = (-\Delta)^{-s} F_x$. By \cite{Silvestre2007}, we find using $N < 4s$ that in any of the above cases, $\tilde{\psi_x} \in C^{1, \alpha}$ for all $\alpha \in (0,1]$ with $\alpha < 4s - N$. Using Hölder continuity of the gradient, we easily find
	\[ \psi_x(y) = \psi_x(x) + \nabla \psi_x(x) \cdot(y-x) + \mathcal O(|x-y|^{\alpha + 1}). \]
	Choosing $\alpha > 4s - N - 1$ and inserting the definition of $\psi_x$, we find 
	\[ H_b(x,y) = \phi_b(x) + \nabla \psi_x(x) \cdot (y-x) - d_{N,s} b(x) |x-y|^{4s - N} + o( |x-y|^{4s - N}), \]
	which is the assertion of the lemma with $\xi_x :=  \nabla \psi_x(x)$. 
\end{proof}

\begin{lemma}
\label{lemma Up-k Hk}
Let $k \in \N$ with $k \leq p = \frac{2N}{N-2s}$. 
If $k > \frac{2s}{N-2s}$, then 
\[ \lambda^{-\frac{k}{2}(N-2s)} \int_\Omega U_\xl^{p-k} H_a(x, \cdot)^k \diff y = o(\lambda^{-2s}) .\]
If $2 \leq k \leq \frac{2s}{N-2s}$, then
\[ \lambda^{-\frac{k}{2}(N-2s)} \int_\Omega U_\xl^{p-k} H_a(x, \cdot)^k \diff y = \left( \int_{\R^N} U_{0,1}(y)^{p-k} \diff y \right) \phi_a(x)^k \lambda^{-k(N-2s)} + o(\lambda^{-2s}). \]
If $k = 1, $
	\[ \lambda^{-\ns} \int_\Omega U_\xl^{p-1} H_b(x, \cdot) \diff y = a_{N,s} \phi_b(x) \lambda^{-N+2s} - d_{N,s} b_{N,s} b(x) \lambda^{-2s} + o (\lambda^{-2s}) + o(\phi_b(x) \lambda^{-N+2s}). \]
The asymptotics are uniform for $x$ in compacts of $\Omega$. 
\end{lemma}

\begin{proof}
 Let us start with the easy case of $k > \frac{2s}{N-2s}$. In that case, since $H_a(x, \cdot)$ is uniformly bounded, we have 
 \[ \lambda^{-\frac{k}{2}(N-2s)} \int_\Omega U_\xl^{p-k} H_a(x, \cdot)^k \diff y \lesssim \lambda^{-k ( N -2s)} \int_{B_{R\lambda}} U_{0,1}^{p-k} \diff y \lesssim \begin{cases}
 \lambda^{-k (N-2s)} & \text{ if } k < \frac{N}{N-2s}, \\
 \lambda^{-N} \ln \lambda &\text{ if } k = \frac{N}{N-2s}, \\
 \lambda^{-N} &\text{ if } k > \frac{N}{N-2s}.
 \end{cases}
 \]
 In any case, this is $o(\lambda^{-2s})$. 
 
Now assume  $1 \leq k \leq \frac{2s}{N-2s}$. Let us abbreviate $d = d(x)$ and $B_d = B_d(x)$ and show that the integral over $\Omega \setminus B_d$ is $o(\lambda^{-2s})$. Indeed, since $H_a(x, \cdot)$ is uniformly bounded, 
\begin{align}
    \lambda^{-\frac{k}{2}(N-2s)} \int_{\Omega \setminus B_d} U_\xl^{p-k} H_a(x, y)^k \diff y & \lesssim \lambda^{-\frac{k}{2}(N-2s)}\int_{\R^N \setminus B_d} U_\xl^{p-k} \diff y \\
    & = \lambda^{-k(N-2s)} \int_{\R^N \setminus B_{d\lambda}} U_{0,1}^{p-k} \diff y = \lambda^{-N} = o(\lambda^{-2s}) . 
\end{align}
To evaluate the remaining integral over $B_d$, we use the formula 
\begin{equation}
    \label{Ha formula lemma integral}
    (H_a(x,y))^k = \left( \phi_a(x) + \xi_x \cdot (y-x) - d_{N,s} (a(x)+o(1)) |x-y|^{4s - N} \right) ^k 
\end{equation} 
by Lemma \ref{lemma Ha expansion} (where $\xi_x$ may be zero if we are in case (i) of that lemma). After multiplying out the right side, every term containing the factor $\xi_x \cdot (y-x)$ only once vanishes by oddness. 

Let now $k \geq 2$. Since $\phi_a(x)$ and $a(x)$ are uniformly bounded and $\Omega$ is bounded, it is clear that we can estimate 
\[ H_a(x,y)^k = \phi_a(x)^k + \mathcal O( |y-x|^2 +  |y-x|^{4s - N})  \leq \phi_a(x)^k + \mathcal O( |y-x|^{4s - N}). \]
For the last step we used that $4s - N \leq 2 + 2s - N < 2$. Now 
\begin{align}
  \lambda^{-\frac{k}{2}(N-2s)}  \int_{B_d} U_\xl^{p-k} |x-y|^{4s-N}  \diff y &= \lambda^{-k(N-2s)} \lambda^{N-4s} \int_{B_{d \lambda}(0)} U_{0,1}^{p-k} |y|^{4s - N} \diff y \\
  & \lesssim \lambda^{-N} \lambda^{-(k-2)(N-2s)} \times \begin{cases}
  \ln \lambda  & \text{ if } k = 2, \\
  \lambda^{(k-2)(N-2s)} & \text{ if } k > 2.
  \end{cases}
\end{align}
In any case, this is $o(\lambda^{-2s})$. 

Finally, if $k=1$, plugging in expansion \eqref{Ha formula lemma integral}, the term involving $a(x)$ is not negligible anymore. Instead, it gives 
\[ \lambda^{-\ns} \int_{B_d} U_\xl^{p-1} (a(x) + o(1)) |x-y|^{4s-N} \diff y = \lambda^{-2s} a_{N,s} a(x) + o(\lambda^{-2s}),  \]
which completes the proof. 
\end{proof}

\section{Auxiliary computations}
\label{sec:app:auxiliary}

In this appendix, we collect some technical results and computations used throughout the paper. 

First, we compute the $L^q$ norm of $U_{x,\lambda}$ for various values of $q$.
\begin{lemma}[$L^q$-norm of $U_{x,\lambda}$]
	\label{lemma int Uq}
	Let $x \in \Omega$ and $q \in [1, \infty]$. As $\lambda \to \infty$, we have, uniformly for $x$ in compact subsets, 
	\[ \|U_\xl \|_{L^q(\Omega)} \sim 
	\begin{cases} \lambda^{\frac{N-2s}{2}- \frac{N}{q}}, & q > \frac{N}{N-2s}, \\
	\lambda^{-\frac{N-2s}{2}} (\ln \lambda)^{\frac{N-2s}{N}}, & q = \frac{N}{N-2s}, \\
	\lambda^{-\frac{N-2s}{2}}, & q < \frac{N}{N-2s}. 
	\end{cases} 
	\]
	Moreover, for $\pl U_\xl = \frac{N-2s}{2} \lambda^{\frac{N-2s-2}{2}} \frac{1 - \lambda^2 |x-y|^2}{(1+ \lambda^2 |x-y|^2)^\frac{N-2s+2}{2}}$, we have $|\pl U_\xl | = \mathcal O(\lambda^{-1} U_\xl)$ pointwise and therefore 
	\[ \| \pl U_\xl \|_q \lesssim \lambda^{-1} \|U_\xl\|_q, \qquad q \in [1, \infty]. \]
	Finally, for $\pxi U_\xl = (-N +2s) \lambda^\frac{N-2s+2}{2} \frac{\lambda(x-y)}{(1 + \lambda^2 |x-y|^2)^\frac{N-2s+2}{2}}$, we have 
	\[ \|\pxi U_\xl \|_{L^q(\Omega)} \sim 
	\begin{cases} \lambda^{\frac{N-2s+2}{2}- \frac{N}{q}}, & q > \frac{N}{N-2s+1}, \\
	\lambda^{-\frac{N-2s}{2}} (\ln \lambda)^{\frac{N-2s+1}{N}}, & q = \frac{N}{N-2s+1}, \\
	\lambda^{-\frac{N-2s}{2}}, & q < \frac{N}{N-2s+1}. 
	\end{cases} 
	\]
	
\end{lemma}

\begin{lemma}
    \label{lemma Hs norms PU}
    We have 
    \[ \| \ds PU\| \sim 1, \qquad \|\ds \pl PU\| \sim \lambda^{-1}  , \qquad \|\ds \pxi PU\| \sim \lambda . \]
    Moreover, 
    \begin{align*}
         \int_{\R^N} \ds PU_\xl \ds \pl PU_\xl \diff y & \lesssim \lambda^{-N+2s-1},  \\
      \int_{\R^N} \ds PU_\xl \ds \pxi PU_\xl \diff y &\lesssim \lambda^{-N+2s}, \\
         \int_{\R^N} \ds \pl PU_\xl \ds \pxi PU_\xl \diff y &\lesssim \lambda^{-N +2s -1}, \\
         \int_{\R^N} \ds \pxi PU_\xl \ds \pxj PU_\xl \diff y&\lesssim \lambda^{-N+2s}. 
    \end{align*} 
\end{lemma}

We remark that the bounds of Lemma \ref{lemma Hs norms PU} are consistent with the ones proved in \cite[Appendix B]{Rey1990}. 

\begin{lemma}
	\label{lemma h}
	We have 
	\begin{equation}
	\label{identity h} 0 \leq \frac{ \lambda^{-\frac{N-2s}{2}}}{|x - y|^{N-2s}} - U_\xl(y) = \lambda^\frac{N-2s}{2} h(\lambda (x-y)), 
	\end{equation}
	with 
	\[ h(z) := \left( \frac{1}{1 + |z|^2}\right)^\frac{N-2s}{2} - \frac{1}{|z|^{N-2s}}. \]
	Moreover $h(z) \sim |z|^{-N-2 + 2s}$ and $|\nabla h(z)| \sim |z|^{-N + 2s - 3}$ as $|z| \to \infty$. Consequently, $h \in L^p(\R^N)$ for every $p \in [1, \frac{N}{N-2s})$ and $\nabla h \in L^p(\R^N)$ for every $p \in [1, \frac{N}{N-2s+1})$, where the latter interval is possibly empty.
\end{lemma}

\begin{lemma}
	\label{lemma T xl}
	Let $b \in C(\overline{\Omega}) \cap C^1(\Omega)$. As $\lambda \to \infty$, uniformly for $x$ in compact subsets of $\Omega$, 
	\[ \int_\Omega b(y) U_\xl(y) \lambda^{\frac{N-2s}{2}} h(\lambda(x-y)) \diff y = \alpha_{N,s} \lambda^{-2s} b(x) + o( \lambda^{-2s}). \]
	The numerical value of $\alpha_{N,s} = \int_{\R^N} U_{0,1}(y) h(y) \diff y$ is given in Lemma \ref{lemma constants} below. 
\end{lemma}

\begin{proof}
Abbreviate $d=d(x)$ and $B_d = B_d(x)$. We integrate separately over $B_d$ and over $\Omega \setminus B_d$. 
	
	For the outer integral, from Lemma \ref{lemma h} we get that $U_{0,1}(y) h(y) \sim |y|^{-2N + 4s -2}$. Thus
	\begin{align*}
	\int_{\Omega \setminus B_d}  b(y) U_\xl(y) \lambda^{\frac{N-2s}{2}} h(\lambda(x-y)) \diff y \lesssim \lambda^{-N + 2s - 2} = o(\lambda^{-N}) = o(\lambda^{-2s}). 
	\end{align*}

	For the inner integral, using that $b \in C^1(\Omega)$, we write $b(y) = b(x) + \nabla b(x) \cdot (y-x) + o(|x-y|)$ for $y \in B_d$. Then (the integral over $\nabla b(x) \cdot (y-x)$ cancels due to oddness) 
	\begin{align*}
	&\int_{B_d}  b(y) U_\xl(y) \lambda^{\frac{N-2s}{2}} h(\lambda(x-y)) \diff y \\
	&= b(x) \lambda^{-2s} \int_{B_{\lambda d}(0)} U_{0,1}(y) h(y) \diff y + o\left( \lambda^{-2s-1} \int_{B_{\lambda d}(0)} U_{0,1}(z) h(z) |z| \diff z \right) \\
	&= b(x) \lambda^{-2s} \alpha_{N,s} + o(\lambda^{-2s} +  o\left( \lambda^{-2s-1} \int_{B_{\lambda d}(0)} U_{0,1}(z) h(z) |z| \diff z \right) .
	\end{align*}
	
	To show that the last term is $o(\lambda^{-2s})$ as well, note that by Lemma \ref{lemma h} we have $U_{0,1}(z) h(z) |z| \lesssim |z|^{-2N + 4s -1}$. Thus
	\[ \lambda^{-2s-1} \int_{B_{\lambda d}(0)} U_{0,1}(z) h(z) |z| \diff z \lesssim \begin{cases}
	\lambda^{-2s - 1} & \text{ if } N > 4s-1, \\
	\lambda^{-2s -1} \log \lambda & \text{ if } N = 4s-1, \\
	\lambda^{-N+4s-1} & \text{ if } N < 4s-1.  \\
	\end{cases}
	\]
	This is $o(\lambda^{-2s})$ in all cases. 
\end{proof}

We compute explicitly the constants that appear in the asymptotic expansions throughout the paper. 

\begin{lemma}[Constants]
	\label{lemma constants}
		For $N > 2s$ and $p = \frac{2N}{N-2s}$, let $U_{0,1}(y) =  \left(\frac{1}{1+|y|^2} \right)^{\frac{N-2s}{2}} $ and $h(y)= \frac{1}{|y|^{N-2s}} - \frac{1}{(1 + |y|^2)^{\frac{N-2s}{2}}}$. Then for every $0 \leq k < \frac{N}{N-2s}$ we have 
	\begin{align*}
	a_{N,s}(k) &:= \int_{\R^N}  U_{0,1}(y)^{p-k} \diff y = \frac{\pi^{N/2} \Gamma\left(\frac{N}{2} (1-k) + ks\right)}{\Gamma\left(\frac{N}{2} (2-k) + ks\right)}.
	\end{align*}
We denote $A_{N,s} := a_{N,s}(0)$ and $a_{N,s} := a_{N,s}(1)$. Further,
	\begin{align*}
	b_{N,s} &:= \int_{\R^N} U_{0,1}(y)^\frac{N+2s}{N-2s} |y|^{4s-N} \diff y = \pi^{N/2} \frac{\Gamma(2s)\Gamma\left(\frac{N}{2} - s\right)}{\Gamma\left(\frac{N}{2}\right)\Gamma\left(\frac{N}{2} +s \right)}, \\
	\alpha_{N,s} &:= \int_{\R^N} U_{0,1}(y) h(y) \diff y =  \frac{\pi^{N/2}}{\Gamma\left(\frac{N}{2}\right)} \Gamma(\frac{N}{2}-2s)\left(\frac{\Gamma(s)}{\Gamma(\frac{N}{2}-s)}-\frac{\Gamma(\frac{N}{2})}{\Gamma(N-2s)}\right),  
	\end{align*}

	Moreover, the constant in $(-\Delta)^s u(x) := C_{N,s} P.V. \int_{\R^N} \frac{u(x) - u(y)}{|x-y|^{N+2s}} \diff y$ is given by
\[ 
C_{N,s} :=   \frac{2^{2s} \Gamma(\frac{N+2s}{2})}{\pi^{N/2} s \Gamma (1-s)} \]
and the constant in $\Ds U_{0,1} = c_{n,s} U_{0,1}^\frac{N+2s}{N-2s}$ is given by
	\[ c_{N,s} =  2^{2s} \frac{\Gamma(\frac{N+2s}{2})}{\Gamma(\frac{N-2s}{2})}. \]
	The explicit value of the best fractional Sobolev constant in $\|\ds u\|^2 \geq S \|u\|_\frac{2N}{N-2s}^2$ is   
	\[ S := S_{N,s} = 2^{2s} \pi^s \frac{\Gamma(\frac{N+2s}{2})}{\Gamma(\frac{N-2s}{2})} \left( \frac{\Gamma(N/2)}{\Gamma(N)}\right)^{2s/N}. \]
The constant in $\Ds |x|^{4s - N} = -d_{N,s}^{-1} |x|^{2s-N}$ is given by
	\begin{equation}
    \label{definition dNs}
    d_{N,s} := - 2^{-2s} \frac{\Gamma(\frac{N-4s}{2}) \Gamma(s)}{\Gamma(\frac{N-2s}{2}) \Gamma(2s)} > 0.
\end{equation}
The constant $\gamma_{N,s}$ in $(\Ds + a) G_a(x, \cdot) = \gamma_{N,s} \delta_x$ is given by
\[ \gamma_{N,s} = \frac{2^{2s} \pi^{N/2} \Gamma(s)}{\Gamma(\frac{N-2s}{2})}. \]
\end{lemma}

\begin{proof}
The values of $a_{n,s}(k)$ and $b_{N,s}$ are a consequence of the following computation. For $\alpha, \beta >0$,
\begin{align}\label{eq:integralcomp}
\begin{split}
        \int_{\R^N} \left(\frac{1}{1+|y|^2} \right)^\alpha |y|^\beta \dd y &= \frac{2\pi^{N/2}}{\Gamma\left(\frac{N}{2}\right)} \int_0^\infty \left(\frac{1}{1+r^2}  \right)^\alpha r^{N-1+\beta} \dd r \\
    &= \frac{\pi^{N/2}}{\Gamma\left(\frac{N}{2}\right)}  B\left(\tfrac{\beta + N}{2},  \alpha - \tfrac{\beta + N}{2}  \right) =\frac{\pi^{N/2}}{\Gamma\left(\tfrac{N}{2}\right)}  \frac{\Gamma\left(\frac{\beta + N}{2} \right) \Gamma\left(\alpha - \frac{\beta + N}{2} \right)}{\Gamma(\alpha)}. 
\end{split}
\end{align}

To compute $\alpha_{N,s} $, we write 
\begin{align*}
    \alpha_{N,s} &= \int_{\R^N} U_{0,1}(y) h(y) \diff y\\ &= \frac{2\pi^{N/2}}{\Gamma\left(\frac{N}{2} \right)} \int_0^\infty \underbrace{\Bigg(\Bigg(\frac{1}{1+r^2}  \Bigg)^{\frac{N-2s}{s}} r^{2s-1} - \left(\frac{1}{1+r^2}  \right)^{N-2s} r^{N-1}  \Bigg)}_{I(r,N,s)} \dd r. 
\end{align*}
If $N > 4s$, then the summands of $I(r, N,s)$ are separately integrable, in which case \eqref{eq:integralcomp} gives 
\begin{align}
\label{alpha Ns}
 \alpha_{N,s} &= \frac{\pi^{N/2}}{\Gamma\left(\frac{N}{2}\right)} \Gamma\left(\frac{N}{2}-2s\right)\left(\frac{\Gamma(s)}{\Gamma\left(\frac{N}{2}-s\right)}-\frac{\Gamma(\frac{N}{2})}{\Gamma(N-2s)}\right).
\end{align}
To extend this formula to the case $2s < N < 4s$ which concerns us, we remark that the right side of \eqref{alpha Ns} defines a holomorphic function of $s$ in the complex subdomain $\mathcal D_N := \{ 0 < \text{ Re}(s) < N/2 \} \setminus \{N/4\} \subset \C$. On the other hand, by a cancellation $I(r, N,s)$ remains integrable in $r \in (0, \infty)$ for every $s \in (0,1)$ and $N \in (2s, 4s)$. Indeed,
\begin{align*}
    I(r,N,s) &\sim (r^{2s-1}-r^{N-1}) \qquad \text{ as $r \to 0$}, \\
    I(r,N,s) &\sim \left( r^{-2\frac{N-2s}{2}}r^{2s-1}(\frac{1}{1+\frac{1}{r^2}})^{\frac{N-2s}{2}}-r^{-2(N-2s)}r^{N-1}(\frac{1}{1+\frac{1}{r^2}})^{N-2s}\right)
    \\&\quad=\left(r^{-2\frac{N-2s}{2}}r^{2s-1}(1+\frac{1}{r^2})^{-\frac{N-2s}{2}}-r^{-2(N-2s)}r^{N-1}(1+\frac{1}{r^2})^{-(N-2s})\right)
    \\&\quad=r^{-N+4s-1}\left(1-\frac{\frac{N}{2}-s}{r^2}+O(\frac{1}{r^4})-1+\frac{N-2s}{r^2}+O(\frac{1}{r^4})\right)
    \\&\quad=\left(\frac{N}{2}-s\right)r^{-N+4s-3}+O(r^{-N+4s-5}) \quad \text{ as $r \to \infty$}.
\end{align*}
By a standard argument, this implies that $\int_0^\infty I(r,N,s) \diff r$ is holomorphic in $\mathcal D_N$ as a function of $s$. By the identity theorem for analytic functions, the formula \eqref{alpha Ns} thus holds also for $s \in (N/4, N/2)$, which is what we wanted to show. 

Finally, the claimed value of $S$ can be found e.g. in \cite[Theorem 1.1]{CoTa2004} and that of $d_{N,s}$ in \cite[Table 1, p. 168]{Kwasnicki2019}.
\end{proof}

\begin{lemma}\label{lem-taylor}
Let $2s < N <4s$ and let $f_\epsilon: (0, \infty) \to \R$ be given by
\[ f_\epsilon(\lambda) = \frac{A_\epsilon }{\lambda^{2s}} - B_\epsilon  \frac{\epsilon }{\lambda^{N-2s}}\]
with $A_\epsilon, B_\epsilon > 0$ uniformly bounded away from 0 and $\infty$. The unique global minimum of $f_\epsilon$ is given by 
\begin{equation}
\label{f eps minimum}
    \lambda_0 = \lambda_0(\epsilon) = \left(\frac{2s A_\epsilon}{(N-2s) B_\epsilon } \right)^\frac{1}{4s-N}  \epsilon^{-\frac{1}{4s-N}}. 
\end{equation} 
with corresponding minimal value
\begin{equation}
    \label{f eps min value}
     \min_{\lambda> 0} f_\eps (\lambda) =  f_\eps(\lambda_0) = - \eps^\frac{2s}{4s-N} \frac{B_\eps^\frac{2s}{4s-N}}{A_\eps^\frac{N-2s}{4s-N}} \left(\frac{N-2s}{2s} \right)^\frac{2s}{4s-N} \frac{4s - N}{N-2s}. 
\end{equation}
Moreover, there is a $c_0 > 0$ such that, for all $\epsilon > 0$, we have
\begin{equation}
\label{f eps lower bound}
    f_\epsilon(\lambda) - f_\epsilon(\lambda_0) \geq 
\begin{cases}
c_0 \epsilon^\frac{2s-2}{4s-N}  \left(  \lambda^{-1} - \lambda_0(\epsilon)^{-1}\right)^2 & \text{ if } \quad \left(\frac{A_\epsilon }{B_\epsilon}\right)^{\frac{1}{4s-N}} \epsilon^{-\frac{1}{4s-N}} \lambda^{-1} \leq 2 \left(\frac{2s}{N-2s}\right)^\frac{1}{N-4s} , \\
c_0 \epsilon ^\frac{2s}{4s-N}  & \text{ if } \quad  \left(\frac{A_\epsilon }{B_\epsilon}\right)^{\frac{1}{4s-N}} \epsilon^{-\frac{1}{4s-N}} \lambda^{-1} > 2 \left(\frac{2s}{N-2s}\right)^\frac{1}{N-4s}.
\end{cases} 
\end{equation} 
\end{lemma}

\begin{proof}
The values of $\lambda_0$ and $f_\eps(\lambda_0)$ are obtained by standard computations. Thus we only prove \eqref{f eps lower bound}. 
Let $F(t):= t^{2s} - t^{N-2s}$ and denote by $t_0 := (\frac{2s}{N-2s})^{-\frac{1}{4s-N}}$ the unique global minimum  of $F$ on $(0, \infty)$. Then, there exists $c > 0$ such that 
\[ F(t) - F(t_0) \geq 
\begin{cases}
c(t - t_0)^2 & \text{ if } \quad 0 < t  \leq 2 t_0, \\
c t_0^{N-2s} & \text{ if } \quad t > 2 t_0.
\end{cases}
\]
The assertion of the lemma now follows by rescaling. Indeed, it suffices to observe that 
\[ f_\epsilon(\lambda) =  A_\epsilon^{-\frac{N-2s}{4s-N}} B_\epsilon^{\frac{2s}{4s-N}}  \epsilon^{\frac{2s}{4s-N}}   F\left( \left(\frac{A_\epsilon }{B_\epsilon}\right)^{\frac{1}{4s-N}} \epsilon^{-\frac{1}{4s-N}} \lambda^{-1} \right) \]
and to use the boundedness of $A_\epsilon$ and $B_\epsilon$. 
\end{proof}

\vspace{5mm}

\section*{Acknowledgments}

We acknowledge the kind hospitality of the Universität Ulm during the winter school ``Gradient Flows and Variational Methods in PDEs'' (2019), where parts of this project were discussed. 

N. De Nitti is a member of the Gruppo Nazionale per l’Analisi Matematica, la Probabilità e le loro Applicazioni (GNAMPA) of the Istituto Nazionale di Alta Matematica (INdAM). He has been partially supported by the Alexander von Humboldt Foundation and by the TRR-154 project of the Deutsche Forschungsgemeinschaft (DFG). 

T. König acknowledges partial support through ANR BLADE-JC ANR-18-CE40-002. 

\vspace{3mm}

\bibliographystyle{abbrv}
\bibliography{BNAsympt.bib}

\vfill

\end{document}